\documentclass{article}
\usepackage{hyperref}
\hypersetup{bookmarksdepth=3}

\usepackage[utf8]{inputenc}
\usepackage[T1]{fontenc}

\usepackage{amsmath}       
\usepackage{amsthm}        
\usepackage{amssymb}       
\usepackage{amsfonts}
\usepackage[usenames,dvipsnames,svgnames,table]{xcolor}

\usepackage{wrapfig}
\usepackage{url}

\usepackage[normalem]{ulem}  
\usepackage{cancel}

\usepackage{pifont}

\usepackage[backend=biber,maxbibnames=99, style=alphabetic]{biblatex}
\usepackage{tikz}
\usepgflibrary{arrows.meta}

\newcommand*\circled[1]{\tikz[baseline=(char.base)]{
            \node[shape=circle,draw,inner sep=2pt] (char) {#1};}}

\makeatletter
\newtheorem*{rep@theorem}{\rep@title}
\newcommand{\newreptheorem}[2]{%
\newenvironment{rep#1}[1]{%
 \def\rep@title{#2 \ref{##1}}%
 \begin{rep@theorem}}%
 {\end{rep@theorem}}}
\makeatother

\theoremstyle{plain} 
\newtheorem{theorem}{Theorem}[section]

\newreptheorem{theorem}{Theorem}

\newtheorem*{theorem*}{Theorem}
\newtheorem{lemma}[theorem]{Lemma}
\newtheorem{corollary}[theorem]{Corollary}          
\newtheorem{proposition}[theorem]{Proposition}              

\theoremstyle{definition} 
\newtheorem{definition}[theorem]{Definition}

\theoremstyle{remark}  
\newtheorem{remark}[theorem]{Remark}
\newtheorem{example}[theorem]{Example}

\newtheorem*{question*}{{Question}}

\newcommand{\Bord}{\mathrm{Bord}}
\DeclareMathOperator{\cl}{cl}  
\newcommand{\Cob}{\mathrm{Cob}}
\DeclareMathOperator{\Cat}{Cat}

\DeclareMathOperator*{\colim}{colim}

\newcommand{\Fun}{\mathrm{Fun}}

\DeclareMathOperator{\Map}{Map}
\DeclareMathOperator{\Maps}{Maps}  

\newcommand{\op}{\mathrm{op}}

\newcommand{\Alg}{\mathsf{Alg}}

\newcommand{\sTop}{\mathsf{sTop}}
\newcommand{\Top}{\mathsf{Top}}
\newcommand{\Tor}{\mathsf{Tor}}
\newcommand{\TQFT}{\mathsf{TQFT}}
\newcommand{\Vect}{\mathsf{Vect}}

\def\bDelta{\mathbf \Delta}

\def\cC{\mathcal C}\def\cD{\mathcal D}
\def\cF{\mathcal F}
\def\cJ{\mathcal J}
\def\cP{\mathcal P}

\def\cZ{\mathcal Z}

\def\C{\mathbb C}\def\D{\mathbb D}
\def\F{\mathbb F}

\def\N{\mathbb N}\def\P{\mathbb P}
\def\R{\mathbb R}
\def\S{\mathbb S}

\def\Z{\mathbb Z}

\newcommand*\Bm{\ensuremath{\boldsymbol m}}

\newcommand*\Bt{\ensuremath{\boldsymbol t}}
\newcommand*\BW{\ensuremath{\boldsymbol W}}
\newcommand*\Bxi{\ensuremath{\boldsymbol \xi}}

\definecolor{CSPcolor}{rgb}{0.0,0.5,0.75}	

\begin{document}


\pagestyle{empty}

\begin{minipage}{0.05\textwidth}
{\hspace{0.0cm}}
\tikz{\draw (0,0) -- (0, \textheight);}
\end{minipage}
\begin{minipage}{0.95\textwidth}
	\textbf{\LARGE Invertible Topological Field Theories}\\[1.5cm]
	\text{\Large \emph{Christopher Schommer-Pries}}
	
	\vspace{8cm}
	
	{\large \today}
\vfill

\end{minipage}

\cleardoublepage


\subsection*{Abstract}

A $d$-dimensional invertible topological field theory is a functor from the symmetric monoidal $(\infty,n)$-category of $d$-bordisms (embedded into $\R^\infty$ and equipped with a tangential $(X,\xi)$-structure) which lands in the Picard subcategory of the target symmetric monoidal $(\infty,n)$-category. We classify these field theories in terms of the cohomology of the $(n-d)$-connective cover of the Madsen-Tillmann spectrum. This is accomplished by identifying the classifying space of the $(\infty,n)$-category of bordisms with
	$\Omega^{\infty-n}MT\xi$ as an $E_\infty$-spaces. This generalizes the celebrated result of Galatius-Madsen-Tillmann-Weiss \cite{MR2506750} in the case $n=1$, and of B\"okstedt-Madsen \cite{MR3243393} in the $n$-uple case. 
 We also obtain results for the $(\infty,n)$-category of $d$-bordisms embedding into a fixed ambient manifold $M$, generalizing results of Randal-Williams \cite{MR2764873} in the case $n=1$. We give two applications: (1) We completely compute all extended and partially extended invertible TFTs with target a certain category of $n$-vector spaces (for $n \leq 4$), and (2) we use this to give a negative answer to a question raised by Gilmer and Masbaum in \cite{MR3100961}.


\clearpage
\pagestyle{plain}

\tableofcontents


\section{Introduction}

One of the original reasons that topological field theories (TFTs) were studied, though now far from the only or most important reason, is that TFTs provide a source of computable manifold invariants. Manifolds are inherently interesting and being able to distinguish them is commensurately useful. Manifold invariants can allow us to do this. One amusing invariant is the following $\Z/2\Z$-value invariant of connected smooth 4-manifolds: It gives the value one precisely on the smooth 4-sphere and otherwise the value zero. Together with its cousins, these invariants distinguish all smooth 4-manifolds. Of course, this absurd tautological invariant is useless without a way to compute it. 

Topological field theories provide manifold invariants which are \emph{computable}. These invariants satisfy a \emph{locality property}. Given a manifold $M$ we can imagine dividing it along a codimension-one submanifold $Y$:
\begin{equation*}
	M \cong M_1 \cup_Y M_2.
\end{equation*}
A topological field theory $\cZ$ assigns invariants $\cZ(M_1)$, $\cZ(M_2)$ to the two halves, and from these  we can recover the invariant $\cZ(M)$ of the whole manifold. In the simplest situation $\cZ(M_1)$ and $\cZ(M_2)$ would simply be numbers (say complex numbers) and we would obtain $\cZ(M) = \cZ(M_1) \cdot \cZ(M_2)$ as the product. The general situation is more complicated. In part this is motivated by physics, the historical examples of quantum Chern-Simons theory, and by a desire for the richest manifold invariants possible. In these cases the topological field theory assigns to $Y$ a vector space $\cZ(Y)$. Then $\cZ(M_2) \in \cZ(Y)$ is a vector and $\cZ(M_1)\in \cZ(Y)^*$ is a covector. The invariant for $M$ is obtained via pairing:
\begin{equation*}
	\cZ(M) = \langle \cZ(M_1), \cZ(M_2)\rangle.
\end{equation*}
Of course one also requires that these invariants satisfy an associativity property whenever $M$ is sliced along two parallel disjoint codimension-one submanifolds. 

Symmetric monoidal categories provide a convenient algebraic framework in which to encode these structures and requirements. The Atiyah-Segal axiomatization \cite{a88-tqft,segal} defines a topological field theory as a symmetric monoidal functor:
\begin{equation*}
	\cZ:\Cob_d \to \Vect,
\end{equation*}
where the source is the symmetric monoidal category $\Cob_d$ whose objects are closed compact $(d-1)$-dimensional manifolds, morphisms are equivalences classes of $d$-dimensional bordisms between these, the monoidal structure is given by the disjoint union of manifolds, and where the target is the category of vectors spaces with its standard tensor product monoidal structure. 

We will consider many variants of this notion. First we can replace $\Vect$ by any symmetric monoidal category of our choosing. Second we can require that our bordisms are equipped with a specified \emph{tangential structure}. Fix a fibration $\xi: X \to BO(d)$. An $(X, \xi)$-structure on a $d$-manifold $M$ is a lift $\theta$:
\begin{center}
\begin{tikzpicture}
		\node (LB) at (0, 0) {$M$};
		\node (RT) at (2, 1.5) {$X$};
		\node (RB) at (2, 0) {$BO(d)$};
		\draw [->, dashed] (LB) -- node [above left] {$\theta$} (RT);
		\draw [->>] (RT) -- node [right] {$\xi$} (RB);
		\draw [->] (LB) -- node [below] {$\tau_M$} (RB);
\end{tikzpicture}
\end{center}
where $\tau_M$ is the classifying map of the tangent bundle of $M$. There is a symmetric monoidal category of $\Cob_d^{(X,\xi)}$ where all of our manifolds and bordisms are equipped with $(X, \xi)$-structures.

In this work we will also be mainly concerned with \emph{extended topological field theories}, which were first introduced by Freed and Lawrence \cite{freed-higher,Lawrence} and subsequently studied by Baez-Dolan, Lurie, and many others \cite{bd95-hda, Bartlett:2015aa, Douglas:2013aa, Feshbach:2011aa, kapustin, kl01, lurie-tft, Schommer-Pries:2011aa, segal2010locality, MR3351570}. A traditional  
\begin{wrapfigure}{r}{5cm}
\includegraphics[width=5cm, height = 4.5cm]{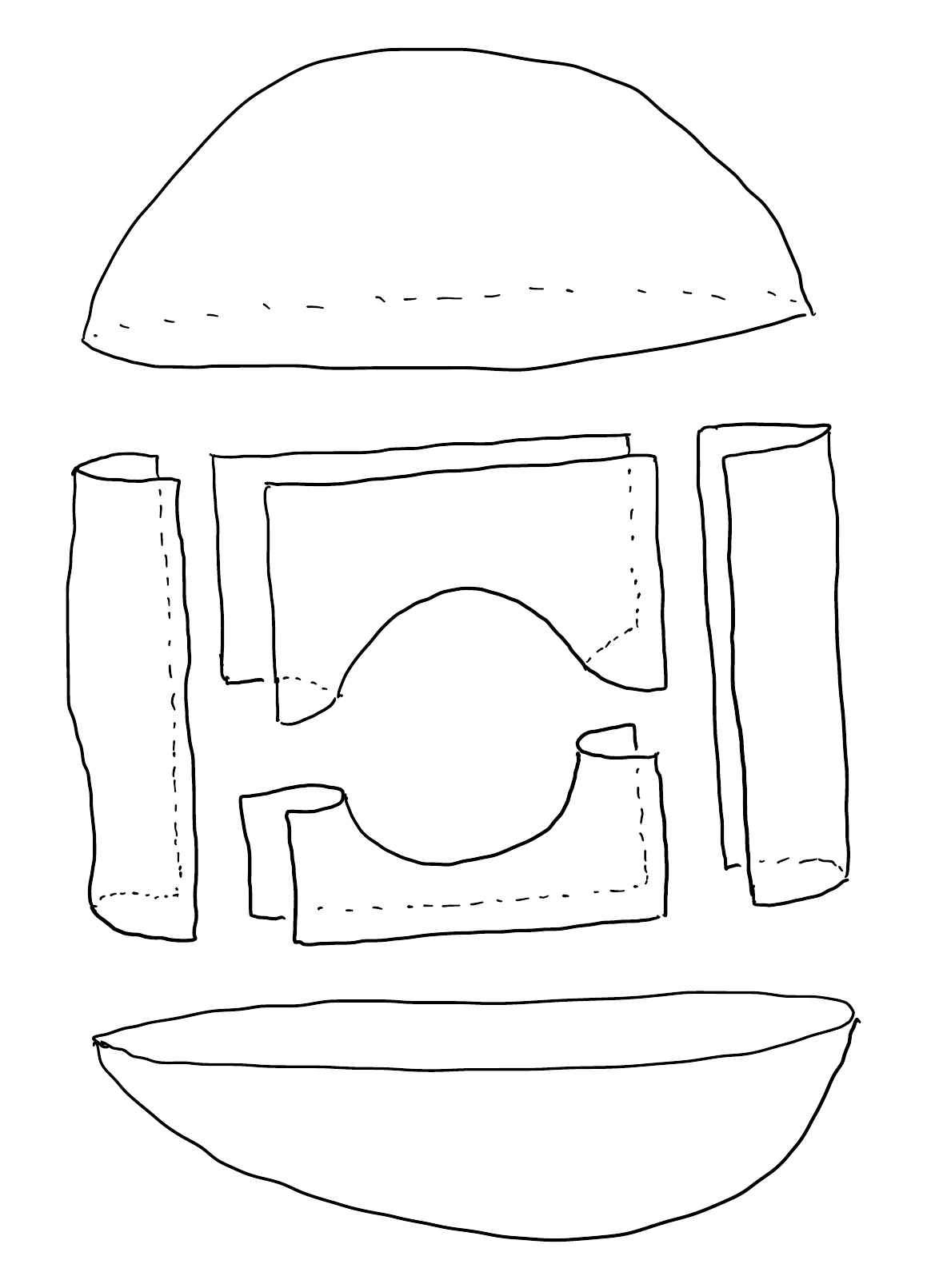}
\caption{A 2-categorical decomposition of a torus.}\label{wrap-fig:1}
\end{wrapfigure} 
topological field theory provides an invariant that is computable because slicing a manifold along codimension-one submanifolds allows us to realize it as a composite of more elementary  bordisms. 
In high dimensions even these elementary bordisms can be quite complicated\footnote{The simplest bordisms you can obtain by slicing along \emph{parallel} submanifolds correspond to arbitrary handle attachments in arbitrary $(d-1)$-manifolds. For example in dimension $d=4$, every knot in a 3-manifold gives a distinct elementary bordism.}. In an extended  topological field theory we are allowed to slice our manifold in multiple directions (the pieces will be manifolds with corners). 

Symmetric monoidal $n$-categories provide a convenient algebraic framework encoding our ability to slice our manifolds in $n$-different directions. These directions correspond to the $n$ different ways of composing morphisms in an $n$-category. As we increase the number of directions in which we can slice, the elementary pieces into which we decompose arbitrary manifolds become simpler and simpler. For example Figure~\ref{wrap-fig:1} shows a 2-categorical decomposition of the torus. Each piece in this decomposition is topologically a disk. 

In fact we will use the even more sophisticated framework of \emph{symmetric monoidal $(\infty,n)$-categories}. This framework is at once both more general and on better foundational grounds. There are several equivalent models for the theory of $(\infty,n)$-categories \cite{Barwick:2011aa, MR3109865}. 
In this work we use a topological variant of Barwick's theory of $n$-fold complete Segal spaces (see \cite{lurie-tft} and Sect.~\ref{sec:n_fold_segal}). Thus, as far as this paper is concerned, an $(\infty,n)$-category is simply a particular kind of $n$-fold simplicial space, a functor $(\Delta^\op)^{\times n} \to \Top$.

Any space may be regarded as a constant simplicial space and hence as an $(\infty,n)$-category. Thus any topological operad may correspondingly be regarded as  an operad in $(\infty,n)$-categories. For example the little cubes operads $E_p$ or $E_\infty$ can be regarded as operads in $(\infty,n)$-categories. A \emph{Symmetric monoidal} $(\infty,n)$-category is an $(\infty,n)$-category which is an $E_\infty$-algebra, and symmetric monoidal functor means $E_\infty$-homomorphism.

A $d$-dimensional extended topological field theory is  a symmetric monoidal functor:
\begin{equation*}
	\cZ: \Bord_{d;n}^{(X, \xi)} \to \cC
\end{equation*}
between symmetric monoidal $(\infty,n)$-categories. Here $\cC$ is an arbitrary target symmetric monoidal $(\infty,n)$-category, which means it is an $E_\infty$-algebra in certain kinds of $n$-fold simplicial spaces. $\Bord^{(X,\xi)}_{d;b}$ is a specific concrete $E_\infty$-algebra in $n$-fold simplicial spaces which we describe in detail in Section~\ref{sec:Bordncat} (see also \cite{lurie-tft,CalSch1509,Nguyen-thesis} for closely related treatments), and the extended field theory $\cZ$ is an $E_\infty$-map $\cZ: \Bord^{(X,\xi)}_{d;n} \to \cC$ between  $E_\infty$-$n$-fold simplicial spaces. 

Philosophically $\Bord_{d;n}^{(X,\xi)}$ is the $(\infty, n)$-category where: 
\begin{itemize}
	\item objects are compact $(d-n)$-manifolds embedded in $\R^\infty$;
	\item 1-morphisms are compact $(d-n+1)$-dimensional cobordisms embedded in $\R^\infty$;
	\item 2-morphisms are compact $(d-n+2)$-dimensional cobordisms between cobordisms embedded in $\R^\infty$;
	\item[] ...
	\item There is a space of $n$-morphisms which is the moduli space of $d$-dimensional cobordisms between cobordisms between cobordisms, etc. embedded in $\R^\infty$; 
\end{itemize}
Moreover all our manifolds are equipped with $(X,\xi)$-structures. 

The main theorem of this paper gives a way to classify a certain subclass of the topological field theories using methods from stable homotopy theory.
This builds on the context of the past decade, which has seen several significant advances in our methods and ability to classify topological field theories. In low dimensions $d$ and low category number $n \leq 2$ (and non $\infty$-categorically) one method is to use Morse theory, Cerf theory, and their generalizations to directly obtain a generators and relations presentation of the bordism $n$-category, see for example: \cite{MR1414088,MR1359651,Schommer-Pries:2011aa,PaperI,Bartlett:2014aa, Bartlett:2015aa,Pstragowski:2014aa,Juhasz:2014aa}. 
This gives a complete classification for arbitrary targets, but so far only works with classical $n$-categories ($(n,n)$-categories, not $(\infty,n)$-categories). 

Another method was developed by Hopkins and Lurie, re-envisioning the Baez-Dolan \emph{cobordism hypothesis}. See \cite{bd95-hda,lurie-tft,MR2994995,MR2742424}. This classification is valid for all $(\infty,n)$-category targets and works in all dimensions. However it only applies in the fully-local case where $d=n$. Moreover at the time of writing, a complete proof of the cobordism hypothesis has not yet appeared in the literature.  

In this paper we consider a subclass of the topological field theories, the so-called \emph{invertible topological field theories}, which we describe in the next section. This subclass can be regarded as topological field theories satisfying a certain property or equivalently as topological field theories taking values in a particular class of symmetric monoidal $(\infty, n)$-categories (the \emph{Picard} $\infty$-categories). The main theorem of this paper is valid for all category number $n$ and in all dimensions, and completely reduces the classification of invertible TFTs to approachable computations in stable homotopy. 

\subsection{Invertible topological field theories}

A topologicial field theory is \emph{invertible} if it sends every $k$-morphism of $\Bord^{(X, \xi)}_{d;n}$ ($1 \leq k \leq n$) to an invertible morphism in the target, and moreover if every object is sent to a $\otimes$-invertible object. This means that the TFT, with target $\cC$, factors through the maximal \emph{$\infty$-Picard} subcategory of $\cC$. 
An $\infty$-Picard category is a symmetric monoidal $(\infty,n)$-category $E$ in which all objects and morphisms are invertible. This can be defined in a model independent way as a symmetric monoidal $(\infty,n)$-category $E$ in which the shear map
\begin{equation*}
	(\otimes, \textrm{proj}_1): E \times E \to E \times E
\end{equation*}
is an equivalence. In this second definition it is clear that every object is $\otimes$-invertible, but in fact it also implies that every 1-morphism, 2-morphism, etc. is invertible. 

All $(\infty,n)$-categories satisfy a condition called \emph{completeness} (see Def.~\ref{def:univalent_segal_space}). When their morphisms are invertible, this forces the underlying $n$-fold simplicial space to be levelwise equivalent to a constant $n$-fold simplicial space\footnote{Hence $E$ is actuality a symmetric monoidal $(\infty,0)$-category.}. It follows that a symmetric monoidal $(\infty,n)$-category is Picard precisely if it is \emph{essentially constant} (all the multisimplicial maps are weak equivalences, see Section~\ref{sec:n_fold_segal}) and \emph{group-like}. In particular Picard $(\infty,n)$-categories are a model for connective spectra.

The inclusion of the constant $E_\infty$-$n$-fold simplicial spaces into all symmetric monoidal $E_\infty$-categories has a homotopical left adjoint given by the fat geometric realization (taken in each simplicial direction separately). When $n \geq 1$, the $E_\infty$-space $||\Bord_{d;n}^{(X, \xi)}||$ is automatically group-like and hence Picard. 
It then follows (see for example the discussion in \cite[Sect~2.5]{lurie-tft} or \cite[Sect~7]{Freed:2012aa}) that extended topological field theories valued in the Picard $\infty$-category $E$ are in natural bijection with 
\begin{equation*}
	\pi_0\Map_{E_\infty}(||\Bord_{d;n}^{(X, \xi)}||, E),
\end{equation*} 
that is homotopy classes of  $E_\infty$-maps from $||\Bord_{d;n}^{(X, \xi)}||$ to $E$. Equivalently this is the $E$-cohomology of the spectrum corresponding to $||\Bord_{d;n}^{(X, \xi)}||$. 

Our main theorem, which is described in detail in the next section, identifies $||\Bord_{d;n}^{(X, \xi)}||$ as an infinite loop space. Hence it reduces the classification of invertible topological field theories, in all dimensions $d$ and category number $n$ (and all tangential structure $(X, \xi)$) to computing the $E$-cohomology of a certain spectrum. In many cases this gives a complete solution to the classification of invertible field theories. This is important in part because
invertible TFTs occur often `in nature': 
\begin{itemize}
	\item Many bordism invariants such as characteristic numbers or the signature can be expressed as invertible topological field theories (which consequently gives rise to local formulas for these invariants). We will see some examples shortly and in Section~\ref{sec:application-classification}. One example of such a theory is classical Dijkgraff-Witten theory. This theory, which in dimension $d$ is parametrized by a finite group $G$ and a characteristic class $\omega \in H^d(BG; \C^\times)$, assigns data to oriented manifolds equipped with principal $G$-bundles. It assigns trivial 1-dimensional vector spaces to each $(d-1)$-manifold and to a closed oriented $d$-manifold $M$ with principal $G$-bundle $P$ it assigns
\begin{equation*}
	\langle [M], \omega(P) \rangle
\end{equation*}	
	the $\omega$-characteristic number of $P$.
	
	\item An invertible Spin theory based on the Arf invariant appears in Gunningham's work \cite{Gunningham:2012aa} on Spin Hurewicz numbers. 
	
	\item Invertible field theories govern and control anomalies in more general quantum field theories. See for example the work of Freed \cite{MR3330283}. 
	\item There are also recent real-world applications of invertible topological field theories to condensed matter physics. Specifically the low energy behavior of gapped systems experiencing \emph{short-range entanglement} are well-modeled by invertible topological field theories, see for example \cite{Freed:2014aa,Kapustin:2015aa,Freed:2016aa}.
	\item One approach to Quantum Chern-Simons theory describes it as an invertible 4-dimensional theory coupled together with a 3-dimensional boundary theory. See for instance \cite{MR2648901, walker-note-2006}
	\item Invertible field theories are also one of the key ingredients in the study of what are called `relative field theories' by Freed-Teleman \cite{MR3165462} and `twisted field theories' by Stolz-Teichner \cite{Stolz:2011aa}. 
	\item The author has recently shown that any extended TFT with category number $n \geq 2$ in which the value $\cZ(T^{d-1})$ of the $(d-1)$-torus in invertible, is automatically an invertible TFT \cite{Schommer-Pries:2015aa}.
\end{itemize}

Invertible topological field theories are completely governed by the cohomology of $||\Bord_{d;n}^{(X, \xi)}||$, and as such they could be regarded as significantly simpler than general TFTs. Despite this fact, invertible topological field theories demonstrate a rich mathematical structure, which is revealed through our classification result. For example let us return briefly to the classical $1$-categorical notion of topological field theory, valued in the category of vector spaces. Such theories  associate a vector space $\cZ(M)$ to each closed $(d-1)$-manifold $M$. In order for $\cZ$ to be invertible, each of these vector spaces must be one-dimensional (if this is the case, then $\cZ$ automatically assigns invertible linear maps to every $d$-dimensional bordism as well).

\begin{example} \label{ex:2dnon-local}
	As an illuminating concrete case, consider oriented topological field theories in dimension $d=2$ (category number $n=1$). It is well-known such TFTs with values in the 1-category of vector spaces $\Vect_1$ are in bijection with commutative Frobenius algebras over $\C$ \cite{MR1359651,MR1414088,MR2037238}. An invertible topological field theory will be a commutative Frobenius algebra whose underlying vector space is one-dimensional. Any one-dimensional $\C$-algebra is canonically identified with $\C$, as an algebra. 
	The comultiplication will be a map $\C \to \C \otimes \C \cong \C$, hence just a complex number. The counit is similarly multiplication by a complex number. To satisfy the Frobenius equations, however, these numbers must be inverses of each other. 

	Hence given any invertible complex number $\mu$ there exists a 1-dimensional commutative Frobenius algebra as specified in Figure~\ref{2DTQFT=CommFrobAlgFig}.
	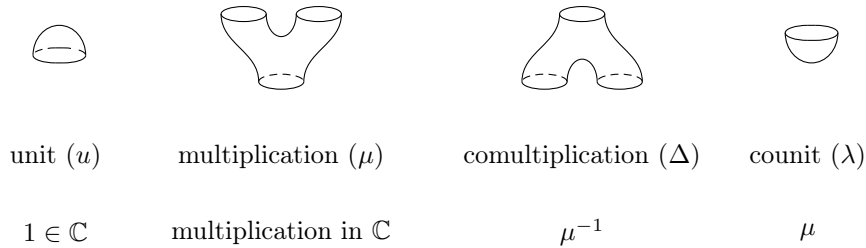
\begin{figure}[ht]
	\begin{center}
	 \begin{tikzpicture}
	 \node at (0, 1) {unit ($u$)};
	 \node at (3, 1) {multiplication ($\mu$)};
	\node at (7, 1) {comultiplication ($\Delta$)};
	\node at (10, 1) {counit ($\lambda$)};

	 \node at (0, 0) {$1 \in \C$};
	 \node at (3, 0) {multiplication in $\C $};
	\node at (7, 0) {$\mu^{-1}$};
	\node at (10, 0) {$\mu$};
	\begin{scope}[ xshift = 0cm, yshift = 0.25cm]
	\draw (.1, 2) -- (0,2);
	\draw (0,2) arc (270: 180: 0.3cm and 0.1cm);
	\draw [densely dashed] (0, 2.2) arc (90: 180: 0.3cm and 0.1cm); 
	\draw [densely dashed] (.1, 2.2) -- (0, 2.2);
	\draw (.1,2) arc (-90: 0: 0.3cm and 0.1cm);
	\draw [dashed] (0.1, 2.2) arc (90:0: 0.3cm and 0.1cm);
	\draw (-.3, 2.1) arc (180: 0: 0.35cm and 0.4cm);
	\end{scope}
	\begin{scope}[ xshift = 2.5cm, yshift = 0cm]
	\draw (0,2.9) ellipse (0.3cm and 0.1cm);
	\draw (1, 2.9) ellipse (0.3cm and 0.1cm);
	\draw (-0.3, 2.9) to [out = -90, in = 90] (0.2, 2) (0.8, 2) to [out = 90, in = -90] (1.3, 2.9);
	\draw (0.2, 2) arc (180: 360: 0.3cm and 0.1cm);
	\draw[densely dashed]  (0.2, 2) arc (180: 0: 0.3cm and 0.1cm);
	\draw (.3, 2.9) arc (180: 360: 0.2cm and 0.3cm);
	\end{scope}
	\begin{scope}[xshift = 4.5cm]
	\draw (2.5,2.9) ellipse (0.3cm and 0.1cm);
	\draw (1.7,2) arc (-180: 0: 0.3cm and 0.1cm);
	\draw [densely dashed] (2.3, 2) arc (0:180: 0.3cm and 0.1cm);
	\draw (1.7, 2) to [out = 90, in = -90] (2.2, 2.9) (2.8, 2.9) to [out = -90, in = 90] (3.3, 2);
	\draw (2.7,2) arc (180: 360: 0.3cm and 0.1cm);
	\draw [densely dashed] (3.3, 2) arc (0: 180: 0.3cm and 0.1cm); 
	\draw (2.3, 2) arc (180: 0: 0.2cm and 0.3cm);
	\end{scope}
	\begin{scope}[ xshift = 10cm, yshift = -0.25cm]
	\draw (.1, 3) arc (90: -90: 0.3cm and 0.1cm) -- (0, 2.8);
	\draw (0, 3) arc (90: 270: 0.3cm and 0.1cm) ;
	\draw (0, 3) -- (.1, 3);
	\draw (-.3, 2.9) arc (180: 360: 0.35cm and 0.4cm);
	\end{scope}
	\end{tikzpicture}
	\caption{A one-dimensional commutative Frobenius algebras}
	\label{2DTQFT=CommFrobAlgFig}
	\end{center}
	\end{figure}
\end{example}

This can be compared to the fully local $(d=2, n=2)$ case. For that we need a target 2-category. 
In \cite{kapranov19942} Kapranov and Voevodsky introduced a symmetric monoidal 2-category $\Vect_2$ of \emph{2-vector spaces}. It can be described as follows. The objects consist of natural numbers. The category of morphisms from $m$ to $n$ is the category of $m \times n$ matrices of vector spaces and matrices of linear maps. The horizontal composition is given by the usual matrix multiplication, but where one replaces the addition and multiplication of numbers with the direct sum and tensor product of vector spaces. Alternatively $\Vect_2$ can be regarded as a full sub-2-category of the Morita 2-category $\Alg$ of algebras, bimodules, and maps. It is the full subcategory on the objects of the form $\oplus_n \C$. 

$\Vect_2$ is a \emph{de-looping} of $\Vect_1$ in the sense that the 1-category of endomorphisms of the unit object of $\Vect_2$ is $\Vect_1$. The Picard sub-2-category of $\Vect_2$, is a delooping of the Picard subcatgory of $\Vect_1$, namely the connected delooping. It is a 2-category which up to isomorphism has one object, one 1-morphism, and $\C^\times$ many 2-morphisms. It is a 2-groupoid model of $K(\C^\times, 2)$.  

Kapranov and Voevodsky's construction can be repeated with $\Vect_2$ in place of $\Vect_1$. This yields a 3-category $\Vect_3$ of \emph{3-vector spaces} whose Picard subcategory models $K(\C^\times, 3)$. Repeating again yields a 4-category $\Vect_4$ of \emph{4-vector spaces} whose Picard subcategory models $K(\C^\times, 4)$, etc.  

\begin{example} \label{ex:2Dlocal}
	Given an invertible complex number $\lambda \in \C^\times$, there exists a fully local ($(d;n) = (2;2)$) invertible topological field theory valued in $\Vect_2$, known as the \emph{Euler field theory}. This filed theory is trivial on the first two layers of $\Bord_{2;2}^{SO(2)}$; it assigns to all 0- and 1-manifolds the respective unit object or identity morphism. Each 2-dimensional bordism $\Sigma$ will have a source $Y_0$ which is itself a 1-dimensional bordism. See the following illustrating example:
 \begin{center}
 \begin{tikzpicture}
\node (S) at (-1.5,2.5) {$\Sigma$};
\node (Y) at (-1.5,3.5) {$Y_0$};
\draw [->] (S) -- (-0.5, 2.5);
\draw [->] (Y.east) to [in=90] (0.5, 3.1);	 
\draw (0, 3) arc (90: -90: 0.3cm and 0.1cm) -- (0, 2);
\draw (1, 3) arc (90: 270: 0.3cm and 0.1cm) -- (1.1, 2.8) -- (1.1, 2) -- (0, 2);
\draw (1,3) -- (1, 2.8); \draw [densely dashed] (1, 2.8) -- (1, 2.2) -- (0,2.2);
\draw (.3, 2.9) arc (180: 360: 0.2cm and 0.3cm);
\draw (0, 3) -- (-.1, 3) -- (-.1, 2.2) -- (0, 2.2); 
 \end{tikzpicture}
 \end{center}
 The value of this field theory on $\Sigma$ is $\lambda^{e(\Sigma,Y)}$, where 
 \begin{equation*}
 	e(\Sigma,Y) = \chi(\Sigma) - \chi(Y)
 \end{equation*}
is the relative Euler characteristic. 
\end{example}

If we restrict a fully-local 2-dimensional TFT valued in $\Vect_2$ to the closed 1-manifolds and the 2-dimensional bordism between these, then we get a traditional 1-categorical TFT valued in vector spaces. Thus it makes sense to ask how the TFTs in Examples~\ref{ex:2dnon-local} and~\ref{ex:2Dlocal} compare? 
In Example~\ref{ex:2dnon-local} the value of the 2-dimensional TQFT on the closed surface $\Sigma_g$ of genus $g$ is $\mu^{1-g}$, while the Euler theory of Example~\ref{ex:2Dlocal} takes value
\begin{equation*}
	\lambda^{\chi(\Sigma_g)} = \lambda^{2-2g} = (\lambda^2)^{1-g}.
\end{equation*} 
In fact the Euler theory associated to $\lambda$ restricts to the commutative Frobenius algebra associated to $\mu = \lambda^2$. In particular the 2- and 1-dimensional part of the theory cannot distinguish between the Euler theories of $\lambda$ and $-\lambda$. The restriction map is at least 2-to-1. 

Using our main theorem, together with some computations of the cohomology of $||\Bord_{d;n}^{SO(d)}||$ which we carry out in Section~\ref{sec:application-classification}, we classify all the invertible field theories in a range of dimensions, as well as compute the associated restrictions maps. 
\begin{reptheorem}{thm:classify-invert-tqft}
For $1 \leq n \leq d \leq 4$ consider the symmetric monoidal functors 	
	\begin{equation*}
		Z: \Bord_{d;n}^{SO(d)} \to \Vect_n
	\end{equation*}
	landing in the Picard subcategory of $\Vect_n$, that is the invertible topological quantum field theories. Let $\TQFT_{d;n}^\textrm{invert}$ denote the set of natural isomorphism classes of such functors. These are 
are classified as follows:
	\begin{enumerate}
		\item When $d=1$ or $d=3$ (all allowed $n$) there is a unique such field theory up to natural isomorphism: the constant functor with value the unit object of $\Vect_n$.
		\item When $d=2$ and $n=1$ or $n=2$ such field theories are determined by a single invertible complex number. The restriction map
		\begin{equation*}
			\TQFT^\textrm{invert}_{2;2} \to \TQFT^\textrm{invert}_{2;1}
		\end{equation*}
		squares this number. 
		\item When $d=4$, then such field theories are determined by a pair of invertible complex numbers.
		The restriction maps
		\begin{equation*}
			\TQFT^\textrm{invert}_{4;3} \to \TQFT^\textrm{invert}_{4;2} \to \TQFT^\textrm{invert}_{4;1}
		\end{equation*}
		are isomorphisms (bijections). The restriction map $\TQFT^\textrm{invert}_{4;4} \to \TQFT^\textrm{invert}_{4;3}$ is given as follows:
		\begin{align*}
			\TQFT^\textrm{invert}_{4;4} &\to \TQFT^\textrm{invert}_{4;3} \\
			(\lambda_1,\lambda_2) & \mapsto ( \lambda_1^2, \frac{\lambda_2^3}{\lambda_1})
		\end{align*}
	\end{enumerate}
\end{reptheorem}

\begin{remark}
	This final restriction map is 6-to-1. If $\zeta$ is any sixth root of unity, then the fully local $(4;4)$-TQFTs corresponding to $(\lambda_1, \lambda_2)$ and to $(\zeta^3 \lambda_1, \zeta \lambda^2)$ have the same restriction to $(4;3)$-TQFTs. 
\end{remark}

The fully-local TFT associated to $(\lambda_1, \lambda_2)$ assigns to an oriented closed 4-manifold $W$ the value $\lambda_1^{e(W)} \lambda_2^{p_1(W)}$, where $e(W)$ is the Euler characteristic and $p_1(W)$ is the $1^\textrm{st}$-Pontryagin number. 

As a second application, we can use our classification result to answer an open problem posed by Gilmer and Masbaum in \cite{MR3100961}, which we now recall. 
In connection to studying anomalous 3-dimensional TFTs, Walker \cite{Walker-1991} considered a certain \emph{central extension} of the the 3-dimensional oriented bordism category. This is a new bordism category whose objects are `extended surfaces'  and whose morphisms are 3-dimensional `extended bordisms' (here `extended' is meant in Walker's terminology, not to be confused with our previous use and meaning of extended TFT). Briefly each surface $\Sigma$ is equipped with a choice of a bounding manifold $\tilde{\Sigma}$. That is $\partial{\tilde{\Sigma}} \cong \Sigma$ A morphism from $(\Sigma_0, \tilde{\Sigma}_0)$ to $(\Sigma_1, \tilde{\Sigma}_1)$ is a 3-dimensional bordism $M$ from $\Sigma_0$ to $\Sigma_1$ together with a choice of 4-manifold $W$ with $\partial W = \overline{\tilde{\Sigma_0}} \cup_{\Sigma_0} M \cup_{\Sigma_1} \tilde{\Sigma_1}$. Two such $W_0$ and $W_1$ are considered equivalent if $W_0 \cup_{\partial W_0} \overline{W_1}$ is null-cobordant. Thus for a given $M$ there are a $\Z$-torsor worth of equivalence classes of possible $W$'s (distinguished by their signature). 

If we fix a surface $\Sigma$, then for each diffeomorphism of $\Sigma$ we get a bordism from $\Sigma$ to itself. It is given by twisting the boundary parametrization of the cylinder bordism $\Sigma \times I$ by the given diffeomorphism. This bordism only depends on the diffeomorphism up to isotopy, and thus in in the bordism category we have a copy of the mapping class group $\Gamma(\Sigma)$ of the surfacce. If we lift $\Sigma$ to an extended surface and look at its automorphism group in the above category, then this fits into a central extension of groups:
\begin{equation*}
	1 \to \Z \to \tilde{\Gamma}(\Sigma) \to \Gamma(\Sigma) \to 1
\end{equation*}
For large genus $H^2(\Gamma; \Z) \cong \Z$ and Walker computed that this central extension corresponds to 4 times the generating extension, see also \cite{MR1329450}. 
In \cite{MR2096678} Gilmer identified an index 2 subcategory of Walker's category. This subcategory then induces the central extension of the mapping class group corresponding to twice the generator of $H^2(\Gamma; \Z)$. Gilmer and Masbaum ask whether it is possible to find an index four subcategory of Walker's category which would realize the fundamental central extension of the mapping class group \cite[Rmk.~7.5]{MR3100961}. 

Using our classification of topological field theories we can answer the Gilmer-Masbaum question negatively. See Section~\ref{sec:app-GM} for full details. 

\begin{theorem}\label{thm:GM-question}
	There is no $\Z$-central extension of the bordism category $\Cob_3^{SO(3)}$ which induces the fundamental central extension of the mapping class group (corresponding to a generator of $H^2(\Gamma_g; \Z) \cong \Z$, $g \geq 3$). In particular there is no index 4 symmetric monoidal subcategory of Walker's `extended bordism' category realizing this fundamental central extension.
\end{theorem}

\subsection{Results}

Our main theorem identifies the homotopy type of $||\Bord_{d;n}^{(X, \xi)}||$ as an $E_\infty$-space: 
\begin{reptheorem}{thm:stable_main}
	Fix numbers $d$ and $n$ and a fibration $\xi: X \to BO(d)$, as in Section~\ref{subsec:symbord}, and let $\Bord_{d;n}^{(X, \xi)}$ be the corresponding symmetric monoidal $(\infty,n)$-category of bordisms. 
	There is a weak equivalence of $E_\infty$-spaces between $||\Bord_{d;n}^{(X, \xi)}||$ and $\Omega^{\infty-n} MT\xi$, where $MT\xi$ is the Madsen-Tillmann spectrum $MT\xi =X^{-\xi^* \gamma_d}$.
\end{reptheorem}

The case $n=1$ is a well-known theorem of Galatius-Madsen-Tillmann-Weiss \cite{MR2506750}, which lead to a solution of the Mumford conjecture. This celebrated result has received much attention. The original argument of GMTW only shows this equivalence at the space level, not as $E_\infty$-spaces. For the case $d=2$, the equivalence as infinite loop spaces can be deduced when combined with \cite{MR1856399}. For general $d$, an identification as infinite loop spaces appears in \cite{Nguyen:2015aa}, using Segal's $\Gamma$-space approach for infinite loop spaces. 

Several variants of the $n=1$ case have appeared subsequently, each one improving, streamlining, and simplifying the original proof \cite{Ayala:2008aa, MR2653727, MR2764873}. A similar result was obtained in \cite{MR3243393} for a different multisimplicial space corresponding to an `$n$-uple' category. The key difference is that the bordism $(\infty,n)$-category we consider here satisfies an additional \emph{globularity} condition (condition (A2) in \cite[Def.~2.1.37]{lurie-tft}). This is the source of most of the work in Section~\ref{sec:real_bordn}.
A proof of the cobordism hypothesis would also establish the case $n=d$, see \cite[Sect~2.5]{lurie-tft} and \cite[Sect~7]{Freed:2012aa}.

We also obtain a non-symmetric monoidal variant extending results of Randal-Williams \cite{MR2764873}. Fix a finite dimensional manifold $M$, possibly with boundary. Then we consider the $(\infty,n)$-category $\Bord_{d;n}^{(X,\xi)}(M)$ consisting of bordisms embedded into $M \times \R^n$ and disjoint from the boundary.  We have:
\begin{theorem*}[Thm~\ref{thm:first_realization_thm}, Cor.~\ref{cor:first_realization_thm}]
	If $M$ is \emph{tame} (see  Def.~\ref{def:tame_manifold}) then we have a weak equivalence
	\begin{equation*}
		||\Bord_{d;n}^{(X, \xi)}(M)|| \simeq \Gamma(M, \partial M; Th^\textrm{f.w.}_M(\xi^*_M \gamma_d^\perp))
	\end{equation*}  
	where $Th^\textrm{f.w.}_M(\xi^*_M \gamma_d^\perp))$ is a bundle associated to the the tangent bundle of $M$ with fiber $Th(\xi^* \gamma_d^\perp)$ a Thom space over $X$ (see Section~\ref{sec:emb_thom}). The right-hand side denotes sections which restrict to the base-point section on $\partial M$. 
	
	In the case $M = D^p$, we get a weak equivalence 
	\begin{equation*}
		||\Bord_{d;n}^{(X, \xi)}(D^p)|| \simeq \Omega^p Th(\xi^* \gamma_d^\perp)
	\end{equation*}
	This is an equivalence of $E_p$-spaces.
\end{theorem*}

\noindent In the above `tame' is a technical condition which, for example, is satisfied by any manifold which has a finite handle decomposition.

Durring this work we have endeavored to incorporate as many improvements and simplifications to the proof of the GMTW theorem as possible. Some improvement which are unique to our treatment are the following:
\begin{itemize}
	\item The construction of the bordism category involves variants on a topological space, $\psi_d(M)$, of embedded submanifolds. As a set $\psi_d(M)$ consists of all those closed subsets $W \subseteq M$ which are smoothly embedded manifolds, disjoint from the boundary of $M$. The topology on $\psi_d(M)$ has been notoriously delicate to construct. For example in \cite[Sect.~2.1]{MR2653727} it takes slightly more than a page to define.  

	Our construction of the topology on the space of embedded manifolds is done via the theory of plots, see Section~\ref{sec:space_of_man}. This allows for a much shorter definition for which it is immediate to see that the various deformations carried out in Section~\ref{sec:real_bordn} are continuous. We show that our topology agrees with the topology constructed in \cite{MR2653727} in Theorem~\ref{thm:spaceofmanifoldcomparison} in an Appendix. Note, however, that the proof of our main theorem doesn't rely this comparison. 
	\item In the course of the proof of the GMTW theorem one is lead to compare the space $\psi_d(D^p \times \R^n)$ and the $p$-fold loop space $\Omega^p\psi_d(\R^{p+1})$. A choice of Segal's `scanning map' gives a comparison map:
	\begin{equation*}
		\psi_d(D^p \times \R^n) \to \Omega^p\psi_d(\R^{p+1}).
	\end{equation*}
	Moreover both $\psi_d(D^p \times \R^n)$ and $\Omega^p\psi_d(\R^{p+1})$ are naturally $E_p$-spaces, the latter with its usual $E_p$-space structure, and former because $\psi_d(-)$ is covariant for closed codimension-one embeddings. However, while the scanning map is a weak equivalence it is \emph{not} compatible 
with the $E_p$-algebra structure. It is not an $E_p$-algebra homomorphism\footnote{This is perhaps one reason that the GMTW theorem was originally only proven at the space level and not as infinite loop spaces.}. 	

	In section~\ref{sec:Ep-scanning} we show how to overcome this to get the desired $E_p$-equivalence. We introduce a larger space $\Psi_d(D^p \times \R^n)$ where a point includes a submanifold and what we call a \emph{scanning function}. This is additional data used to construct an alternative scanning map. We end up with a zig-zag
	\begin{equation*}
		\psi_d(D^p \times \R^n) \stackrel{\sim}{\leftarrow} \Psi_d(D^p \times \R^n) \stackrel{\sim}{\to} \Omega^p \psi_d(\R^{p+n})
	\end{equation*}
of $E_p$-algebra homomorphisms which are weak equivalences. See Thm.~\ref{thm:E_pSpaces}. 
		
	\item  Our proof most closely resembles the one in \cite{MR2653727} (but see also  \cite{Ayala:2008aa} and \cite{MR3243393}). In these proofs the authors rely on a technical result of Graeme Segal \cite[A.1]{MR516216} about \'etale maps of simiplicial spaces. The simplical space corresponding to the bordism category has parameters from the space $\R$ of real numbers. In order to apply Segal's lemma, the authors must replaces their simplicial space with a new simplicial space in which the standard topology on $\R$ is replaced with the discrete topology. This drastically changes the underlying $(\infty,1)$-category, and so one must prove that this nevertheless doesn't effect the homotopy type of the geometric realization.
	
	 We give a different argument (see Section~\ref{subsec:arrow3}) which is more elementary and avoids using Segal's lemma. We avoid changing the topology of the simplicial spaces involved, which gives a more direct comparison, decreasing the number of zig-zags of weak equivalences.
	  This yields a somewhat streamlined comparison. 
\end{itemize}

\subsection{Overview}

In section~\ref{sec:space-of-manifolds} we introduce the \emph{method of plots} which allows us to define many interesting topological spaces. Briefly for a set we can specify a collection of set-theoretic maps, called plots, from test topological spaces to the set. This specifies a topology, the finest making the plots continuous. This is used to define a topology on the set of closed subsets of a fixed topological space, and also on the set $\psi_d(M)$ of submanifolds of $M$. 

In section~\ref{sec:scanning} we review Segal's method of scanning, and the relationship between $\psi_d(\R^n)$ and Thom spaces. Then in section~\ref{sec:Ep-scanning} we modify the scanning map to show that $\psi_d(D^p \times \R^n) \simeq \Omega^p\psi_d(\R^{p+n})$ as $E_p$-algebras. 

In section~\ref{sec:Bordncat} we review $n$-fold Segal spaces (the model of $(\infty,n)$-categories which we employ) and write down precisely the $n$-fold simplicial space $\Bord_{d;n}^{(X, \xi)}$ which is the bordism $(\infty,n)$-category. In
section~\ref{sec:real_bordn} we prove our main theorem, which identifies the weak homotopy type of the geometric realization $||\Bord_{d;n}^{(X, \xi)}||$.

Section~\ref{subsec:example_applications} is devoted to applications. We focus on the oriented case ($X = BSO(d)$). We compute the cohomology of certain connected covers of $MTSO(d)$ in low dimensions. We then use this to prove Theorem~\ref{thm:classify-invert-tqft}, classifying invertible TFTs in low dimensions, and Theorem~\ref{thm:GM-question}, which answers an open question posed by Gilmer and Masbaum \cite[Rmk.~7.5]{MR3100961}.

We also include three appendices. Appendix~\ref{app:comparison} gives a comparison between our topology on the space $\psi_d(M)$ of embedded manifolds and the topology constructed by Galatius--Randal-Williams in \cite[Sect.~2.1]{MR2653727}. We show in Theorem~\ref{thm:spaceofmanifoldcomparison} that these two topologies coincide. In Appendix~\ref{app:embedded_classifying}, show that with our topology $\psi_d(D^\infty)$ is a disjoint union of classifying spaces $BDiff(W)$ taken over diffeomorphism classes of compact closed $d$-manifolds $W$. This justifies regarding it as a moduli space of $d$-manifolds. Finally in Appendix~\ref{app:lowhomotopy} we comute a few low-dimensional homotopy groups of $MTSO(d)$ for $d \leq 4$. These homotopy groups and more were computed in \cite{MR3356279}, but since knowledge of these groups is used in section~\ref{subsec:example_applications} this appendix is provided for the sake of the reader. 

\subsection{Acknowledgements}

We would like to give special thanks Dan Freed and Peter Teichner for their continued support and interest in this work, as well as Oscar-Randal Williams, S\o ren Galatius, and David Ayala for numerous conversations. I would also like to thank Stephan Stolz, Andr\'e Henriques, and Claudia Scheimbauer for useful discussions regarding this work.  

\section{The Space of Embedded Manifolds} \label{sec:space-of-manifolds}

\subsection{Topological Spaces via Plots}

Topological spaces come to us in many ways. Typically we begin with simple familiar spaces, such as vector spaces or other simple manifolds, and we get new spaces by either gluing together those that we know (e.g. CW-complexes, manifolds) or by passing to subspaces (e.g. the Hawaiian earrings, cantor sets, the topologist's sine curve, etc.). Here we will describe another method, the method of plots. 

The idea is to start with a collection of `test objects', which are known topological spaces, together with a collection of set-theoretic maps from these into a given set $X$. These maps, which we call \emph{plots}, then induce a maximal topology on $X$ in which they are continuous.

\begin{definition}\label{def:plots}
	Let $X$ be a set. A \emph{collection of plots} for $X$ is a collection of pairs $\cJ = \{ (J, p) \}$ consisting of topological spaces $J$ and set-theoretic maps $p: J \to X$, which we call \emph{plots}. The collection $\cJ$ determines a topology $\tau_\cJ$ on $X$, the \emph{plot topology}. A subset $U \subseteq X$ is open in $\tau_\cJ$ if and only if $p^{-1}(U) \subseteq J$ is open for all plots; $\tau_\cJ$ is the finest topology on $X$ making all the plots continuous.  
\end{definition}

Generally we will impose further properties on our collections of plots. For example in \cite{MR0300277} a \emph{quasi-topological space} is defined to be a set with a collection of plots such that the spaces $J$ range over the compact Hausdorff spaces, and such that the collection of plots:
\begin{itemize}
	\item contains the constant maps $J \to \{x\} \to X$;
	\item is closed under precomposition with continuous maps, i.e. if $f: J' \to J$ is continuous and $p: J \to X$ is a plot,  then $p f$ is a plot;
	\item if $J$ is the disjoint union of two closed sets $J = J_1 \cup J_2$, then $p: J \to X$ is a plot if and only if $p|_{J_1}$ and $p|_{J_2}$ are plots; and
	\item if $f:J \to J'$ is surjective, then $p: J' \to X$ is a plot if and only if $pf$ is a plot. 
\end{itemize}


\noindent We will primarily be concerned with the situation in which the source of each plot is a smooth manifold.  
Given a topological space $Y$, there is a canonical collection of plots $\cP_Y$ given by taking the plots to be all continuous maps from $U$ into $Y$, with $U$ a smooth manifold. The identity map of sets gives a canonical continuous map $(Y, \tau_{\cP_Y}) \to Y$. This is a homeomorphism if and only if $Y$ is $\Delta$-generated. 
 See \cite{dugger_dgs, MR0300277} for general properties of $\Delta$-generated spaces. 

Similarly a \emph{diffeology}\cite{MR3025051}  on $X$ is a collection of plots $\cD$ where the $J$ range over open subsets of $\R^n$ (for all $n$), such that the collection of plots:
\begin{itemize}
	\item contains the constant maps $J \to \{x\} \to X$;
	\item is closed under precomposition with smooth maps: if $f:V \to U$ is smooth and $p:U \to X$ is a plot, then $pf: V \to X$ is a plot; 
	\item If $U = \cup_i U_i$ is a union of open sets, then $p:U \to X$ is a plot if and only if each $U_i \to X$ is a plot. 
\end{itemize}
Spaces equipped with a diffeology are called \emph{diffeological spaces} and are one of many possible notions of `generalized smooth space'. The resulting topology $\tau_\cD$ was studied in \cite{Christensen:2013aa}. 

\begin{example}[{\cite[example 3.2]{Christensen:2013aa}}]
	Given a smooth manifold $M$, the standard diffeology on $M$ is the collection of plots consisting of all smooth maps $U \to M$. The resulting plot topology is the standard topology on $M$. 
\end{example}

\begin{example}
	Let $M$ and $N$ be smooth manifolds of the same dimension such that $M \subseteq N$. Let $Emb(M,N)$ be the set of (open) smooth embeddings. Then we equip $Emb(M,N)$ with a diffeology as follows: a map $p:U \to Emb(M,N)$ is a plot if the  adjoint map $\tilde{p}:U \times M \to N$ is a smooth map such that for each $u \in U$, 
	\begin{equation*}
		\tilde{p}(u, -): M \to N
	\end{equation*}
	is an (open) embedding. 
\end{example}

If a space is defined by a collection of plots, then it is easy to detect continuous maps out of it; in some cases we can detect some continuous maps into it. 

\begin{lemma}
	Let $(X, \cP)$ and $(X', \cP')$ be a sets with collections of plots $\cP$ indexed on the same spaces $J$, and let $Z$ be a topological space. 
	\begin{enumerate}
		\item A map $f: X \to Z$ is continuous for the plot topology on $X$ if and only if for each plot $p: J \to X$, the composite $fp: J \to Z$ is continuous. 
		\item The plots $p: J \to X$ are continuous for the plot topology; if a set-theoretic map $f: X \to X'$ sends plots to plots (i.e. for each plot $(p: J \to X) \in \cP$ we have $(fp: J \to X') \in \cP'$), then it is continuous. \qed
	\end{enumerate}
\end{lemma}

\noindent Note that in general there will also be continuous maps from $X$ to $X'$ which do not send plots to plots.

\subsection{The space of closed subsets}\label{sec:space_of_closed_sets}

Let $Y$ be a topological space and let $\cl(Y)$ denote the set of closed subsets of $Y$. We will define a collection of plots (and hence a topology) on the set $\cl(Y)$. The source of our plots will be always be  the real line $\R$.

Given $f: \R \to \cl(Y)$ a set-theoretical map consider the \emph{graph}
\begin{equation*}
	\Gamma(f) = \{ (r,y) \in \R \times Y \; | \; y \in f(r) \subseteq Y \}.
\end{equation*}
We will declare that a map $p: \R \to \cl(Y)$ is a plot if the graph $\Gamma(p)$ is a closed subset of $\R \times Y$. Regard $\cl(Y)$ as a topological space with the plot topology for this collection of plots. 

\begin{lemma}\label{lem:mapsarepoltsforclosedsubsets}
	If $Y$ be a locally compact Hausdorff space then every continuous map $p: \R \to \cl(Y)$ is a plot. 
\end{lemma}

\begin{proof}
	First we consider some open subsets of $\cl(Y)$. Let $K \subseteq Y$ be a compact subset and consider
	\begin{equation*}
		M(K) = \{ A \in \cl(Y) \; |\; A \cap K = \emptyset \}. 
	\end{equation*}
	Let $p: \R \to \cl(Y)$ be a plot, thus $\Gamma(p) \subseteq \R \times Y$ is closed. Since $Y$ is Hausdorff $K$ is closed and hence $\Gamma(p) \cap \R \times K$ is a closed subset of $\R \times K$. Since $K$ is compact, the projection $\R \times K \to \R$ is proper, and hence the image of  $\Gamma(p) \cap \R \times K$ in $\R$, which is precisely the complement of $p^{-1}(M(K))$, remains closed. It follows that $M(K)$ is open in $\cl(Y)$. 
	
	Now suppose that $g: \R \to \cl(Y)$ is continuous. We wish to show that $\Gamma(g)$ is closed. Let $(r,y) \in \R \times Y$ be a limit point of $\Gamma(g)$. 	We wish to show that $(r,y) \in \Gamma(g)$. Suppose the contrary, that  $(r,y) \not\in \Gamma(g)$. This means that $y \not\in g(r) \subseteq Y$. Since $g(r)$ is closed and $Y$ is locally compact Hausdorff, we can separate $g(r)$ and $y$ by a compact neighborhood. Specifically there exists a compact subset $K \subseteq Y$ such that $K \cap g(r) = \emptyset$ and an open subset $V \subseteq Y$ such that $y \in V \subseteq K$. Note that by construction $r \in g^{-1}(M(K))$. 

Next for each $n \in \N$ we may consider the open subset $B_{\frac{1}{n}}(r) \times V$, which is an open neighborhood of $(r,y)$. Since $(r,y)$ was assumed to be a limit point of $\Gamma(g)$, there exists $(r_n, y_n) \in \Gamma(g)$ with $(r_n, y_n) \in B_{\frac{1}{n}}(r) \times V$. Thus, since $y_n \in V \subseteq K$, we have $r_n \in (g^{-1}(M(K)))^c$. By construction $r_n \to r$ converges, and since $g$ is continuous, 	$(g^{-1}(M(K)))^c$ is closed, and each $r_n \in (g^{-1}(M(K)))^c$, it follows that $r \in (g^{-1}(M(K)))^c$ as well, a contradiction. 
\end{proof}

\noindent By checking on plots we can see that the following maps are continuous:

\begin{lemma}
	Let $K \subset Y$ be a closed subset, then the map:
	\begin{align*}
		(-) \cap K: \cl(Y) &\to \cl(Y) \\
		Z & \mapsto Z \cap K
	\end{align*}
	is continuous. \qed
\end{lemma}

\begin{lemma}\label{lem:proper_cts_on_closed}
	Let $f:Y \to Y'$ be a proper map.  Then the map
	\begin{align*}
		f_*: \cl(Y) &\to \cl(Y') \\
		Z & \mapsto f(Z)
	\end{align*}
	is continuous. 
\end{lemma}

\begin{proof}
	Since $f:Y \to Y'$ is proper we have that $id \times f: \R \times Y \to \R \times Y'$ is a closed map, and hence $f_* \circ p$ is a plot for any plot $p:\R \to \cl(Y)$. 
\end{proof}

\noindent The topology on $\cl(Y)$ is far from Hausdorff. 

\begin{lemma}
	Let $Z_1, Z_2 \in \cl(Y)$. If $Z_1 \subseteq Z_2$, then $Z_2 \in \overline{\{ Z_1 \}}$ is contained in the closure of $Z_2$. 
\end{lemma}

\begin{proof}
	Consider the plot $p:\R \to \cl(Y)$ defined as
	\begin{equation*}
		p(t) = \begin{cases}
			Z_1 & \text{if } t <0 \\
			Z_2 & \text{if } t \geq 0
		\end{cases}
	\end{equation*}	
	We want to show that $Z_2$ is a limit point of $Z_1$. Let $M\subset \cl(Y)$ be an open subset containing $Z_2$. Then $p^{-1}(M)$ is an open subset of $\R$ containing $0 \in \R$. Hence it also contains a small neighborhood $(a, b) \subseteq p^{-1}(M)$ with $a < 0 < b$. In particular $p(a/2) = Z_1 \in M$. This $Z_1$ is contained in every open neighborhood of $Z_2$, and hence $Z_2 \in \overline{\{ Z_1 \}}$. 
\end{proof}

In particular the empty set $\emptyset \in \cl(Y)$ is a point on $\cl(Y)$ which is dense in $\cl(Y)$. It is a \emph{generic point}. 

\begin{example}[closed set classifier] \label{ex:closed_classifier}
	Consider the one point space $pt$. Then $\cl(pt) = \{ \emptyset, \{pt\}\}$ consists of two points. By Lemma~\ref{lem:mapsarepoltsforclosedsubsets} a continous map $\R \to \cl(pt)$ is the same as a closed subset of $\R \cong \R \times pt$. 
	That closed subset is the inverse image of $\{pt\}$, and hence $\{pt\}$ is a closed subset of $\cl(pt)$. Likewise, $\{\emptyset\}$ is open. As we just observed the closure of $\emptyset$ is all of $\cl(pt)$, and this completely determines the topology of $\cl(pt)$. It is the \emph{closed subset classifier}. For any topological space $Y$, a continuous map $f:Y \to \cl(pt)$ is precisely the data of a closed subset $A \subseteq Y$, and $A = f^{-1}(\{pt\})$ under this correspondence. 
\end{example}

\begin{corollary}\label{cor:empty_open}
	If $Y$ is compact, then the set $\{ \emptyset \} \subseteq \cl(Y)$ is open. 
\end{corollary}

\begin{proof}
	If $Y$ is compact, then $p:Y \to pt$ is proper. The set $\{ \emptyset \} \subseteq \cl(Y)$ is the inverse image of $\{ \emptyset \} \subseteq \cl(pt)$ under $p_*$ and so by Lemma~\ref{lem:proper_cts_on_closed} it suffices to prove the corollary in the case $Y=pt$, but this was done in Example~\ref{ex:closed_classifier}. 	
\end{proof}

\subsection{The space of embedded manifolds} \label{sec:space_of_man}

Let $M$ be a smooth manifold (possibly with boundary and corners). We will refer to $M$ as the \emph{ambient manifold}. Let $\psi_d(M)$ denote the set of subsets $W \subseteq M$ which are smooth $d$-dimensional submanifolds of $M$ without boundary which are also disjoint from $\partial M$ and which are closed as subsets of $M$. We will make this a topological space by using plots. 


As in the previous section, given a set theoretic map $f: X \to \psi_d(M)$ we let the \emph{graph} denote the set
\begin{equation*}
	\Gamma(f) = \{ (x, w) \in X \times M \; | \; w \in f(x) \subseteq M \}.
\end{equation*}

\begin{definition}
	Let $U$ be a smooth manifold
	A set theoretic map $p: U \to \psi_d(M)$ is a \emph{plot} (or a \emph{smooth map}) if $\Gamma(p) \subseteq U \times M$ is a smooth submanifold and the projection $\pi: \Gamma(p) \to U$ is a submersion. 
\end{definition}

\noindent One may verify that this collection of plots satisfies the axioms of a diffeology, but we will not need this fact. We will regard $\psi_d(M)$ as a topological space with the plot topology. 

By checking on plots we immediately see that $\psi_d(-)$ has two functorial properties: 
\begin{itemize}
	\item If $M \subseteq M'$ is an open embedding, then we have a pullback functor
	\begin{equation*}
		\psi_d(M') \to \psi_d(M)
	\end{equation*}
	sending $W \subset M'$ to $W \cap M \subseteq M$. 
	\item If $M \subseteq M'$ is a closed embedding, then we have a pushforward functor $\psi_d(M) \to \psi_d(M')$ which simply regards $W \subseteq M$ as a subset of $M'$.  
\end{itemize}
In this way $\psi_d(-)$ can be thought of as both a presheaf and a co-presheaf. Both of these properties will be important. 

More generally we have:
\begin{lemma}
	Let $M$ and $M'$ be smooth manifolds of the same dimension. Let $f_y:M \subseteq M'$, $y \in Y$ be a smooth family of smooth open embeddings parametrized by a manifold $Y$. Then the map $Y \times \psi_d(M') \to \psi_d(M)$ defined by sending $(y,W)$ to $f_y^{-1}(W)$ is continuous. Similarly if $g_y:M \subseteq N$ is a smooth family of smooth closed embeddings, than $Y \times \psi_d(M) \to \psi_d(N)$ defined by sending $(y, W)$ to $f_y(W)$ is continuous. 
\end{lemma}


\begin{proof}
	We can check this on plots. We will consider the case of open embeddings, the case for closed embeddings is analogous.  A plot $U \to Y \times \psi_d(M')$ consists of two pieces of data. 
	First we have a smooth map $U \to Y$ and hence a smooth family of maps $f: U \times M \to M'$ such that for each fixed $u \in U$ the map $f(u,-): M \to M'$ is an open embedding. 
	Second we have a smooth embedded submanifold $W \subseteq U \times M'$ 
	of dimension $d + \dim U$, such that the projection to $U$ is a submersion. 
	
From this we consider the smooth map 
\begin{align*}
	F:U \times M &\to U \times M' \\
	(u,m) &\mapsto (u, f(u,m))
\end{align*}
This is a open smooth embedding and hence $F^{-1}(W) \subseteq U \times M$ is a smoothly embedded submanifold. The projection map $F^{-1}(W) \to U$ is still a submersion and this defines a plot of $\psi_d(M)$. Since the map in consideration sends plots to plots it is continuous.  
\end{proof}

\subsection{Embedded manifolds with tangential/normal structure}\label{sect:embded_with_tangent}

We will be interested in spaces of manifolds which are not just embedded into a fixed ambient manifold $M$, but also equipped with topological structures on the tangent/normal bundle. Fix a dimension $d$. Let $Gr_d(TM)$ denote the natural fiber bundle over $M$ whose fiber over $m$ is the Grassmannian of $d$-planes $Gr_d(T_mM)$ in the tangent space of $M$ at $m$. If $W \subseteq M$ is an embedded submanifold of dimension $d$, then we have a canonical Gauss map
\begin{equation*}
	\tau_W: W \to  Gr_d(TM)
\end{equation*}
taking $w$ in $W$ to the tangent space $T_wW \subseteq T_w M$.

Suppose that $\rho:Y \to Gr_d(TM)$ is a fibration. Then we define a $(Y, \rho)$-structure $\theta$ on $W$ to be a lift:
\begin{center}
\begin{tikzpicture}
		\node (LB) at (0, 0) {$W$};
		\node (RT) at (2, 1.5) {$Y$};
		\node (RB) at (2, 0) {$Gr_d(TM)$};
		\draw [->, dashed] (LB) -- node [above left] {$\theta$} (RT);
		\draw [->] (RT) -- node [right] {$\rho$} (RB);
		\draw [->] (LB) -- node [below] {$\tau_W$} (RB);
\end{tikzpicture}
\end{center}
of the Gauss map, i.e. a dashed arrow making the triangle commute.

We will also define $(Y, \rho)$-structures for smooth families. Suppose that $U$ is a smooth manifold of dimension $k$ and that $W \subseteq U \times M$ is a smooth manifold of dimension $k +d$ such that the projection $\pi: W \to U$ is a submersion. In this case $\ker(d\pi)$ is a smooth vector bundle over $W$ of dimension $d$ which is naturally embedded into $TM$. We have an induced Gauss map:
\begin{align*}
	\tau_{W/U}: W &\to Gr_d(TM)  \\
	w &\mapsto \ker(d\pi)_w
\end{align*}
An $(Y, \rho)$-structure $\theta$ on the $U$-family $W$ is a lift: 
\begin{center}
\begin{tikzpicture}
		\node (LB) at (0, 0) {$W$};
		\node (RT) at (2, 1.5) {$Y$};
		\node (RB) at (2, 0) {$Gr_d(TM)$};
		\draw [->, dashed] (LB) -- node [above left] {$\theta$} (RT);
		\draw [->] (RT) -- node [right] {$\rho$} (RB);
		\draw [->] (LB) -- node [below] {$\tau_{W/U}$} (RB);
\end{tikzpicture}
\end{center}

Let $\psi_d(M; Y, \rho)$ denote the set of $d$-dimension smooth submanifolds of $M$ equipped with $(Y, \rho)$-structures. The above maps define the plots for $\psi_d(M; Y, \rho)$, making it into a diffeological space as in the previous section. We regard it as a topological space with the plot topology.  

Now in general we will want to regard $\psi_d(M; Y, \rho)$ as a functor of $M$, but this is not possible for arbitrary $(Y,\rho)$. They must also be defined functorially in $M$. The following is one way to achieve this.

Throughout we fix dimensions $d$, as before, and $m$ which is the dimension of the ambient manifold. Let $X$ be a space with a $GL_m$-action, and let $\xi: X \to Gr_d(\R^m) = GL_m / GL_d \times GL_{m-d}$ be a map which is $GL_m$-equivariant and a fiber bundle. 
Then for each choice of ambient manifold $M$ we may consider the associated bundle:
\begin{equation*}
	Fr(TM) \times_{GL_m} X \to Gr_d(TM) \to M
\end{equation*} 
which is a fiber bundle over $Gr_d(TM)$. Here $Fr(TM)$ is the frame bundle of $TM$. We set $X_M = Fr(TM) \times_{GL_m} X$ and $\xi_M: X_M \to Gr_d(TM)$ to be the induced map.

Thus given such an $(X, \xi)$, we obtain a for each $M$ a structure $(X_M, \xi_M)$ for $d$-manifolds embedding in $M$. Hence we can consider the space $\psi_d(M; X_M, \xi_M)$ of manifolds embedded in $M$ equipped with an $(X_M, \xi_M)$-structure. To simplify notation we will write this as:
\begin{equation*}
	\psi_d^{(X, \xi)}(M) = \psi_d(M; X_M, \xi_M).
\end{equation*}
We retain the previous functoriality for the spaces of embedded manifolds:   
\begin{itemize}
	\item If $M \subseteq M'$ is an open embedding (or more generally a submersion), then we have a pullback functor:
	\begin{equation*}
		\psi_d^{(X, \xi)}(M') \to \psi_d^{(X, \xi)}(M).
	\end{equation*}
	\item If $M \subseteq M'$ is a closed embedding, then we have a pushforward functor:
	\begin{equation*}
		\psi_d^{(X, \xi)}(M) \to \psi_d^{(X, \xi)}(M').
	\end{equation*}
\end{itemize}

\begin{example}[tangential structures]
	Suppose that $B$ is a space equipped with a $d$-dimensional vector bundle $E$. Let $\gamma_d$ be the canonical $d$-plane bundle over $Gr_d(\R^m)$. Then we let
	\begin{equation*}
		X = Fr(\gamma_d) \times_{GL_d} Fr(E) = (GL_m \times Fr(E))/ GL_d \times GL_{m-d}
	\end{equation*}
	with its natural map $\xi:X \to Gr_d(\R^m)$. Then the corresponding structure on a manifold $W \subseteq M$ consists of a map $f:W \to B$ and a vector bundle isomorphism $\tau_W \cong f^*(E)$. 
	
	Some special cases are orientations ($B = BSO(d)$), spin structures ($B = BSpin(d)$), tangental framings ($B = pt$ with trivial bundle), $G$-principle bundles ($B = BO(d) \times BG$ with the bundle induced from $BO(d)$), etc. Note that these structures are defined for all $m$. 
\end{example}

\begin{example}[normal structures]
	Suppose that $B$ is a space equipped with an $(m-d)$-dimensional vector bundle $E$. Let $\gamma^\perp_d$ be the canonical $(d-m)$-plane bundle on $Gr_d(\R^m)$. Then similarly to above we let
	\begin{equation*}
		X = Fr(\gamma_d^\perp) \times_{GL_{m-d}} Fr(E) = (GL_m \times Fr(E))/ GL_d \times GL_{m-d}
	\end{equation*}
	with its natural map $\xi:X \to Gr_d(\R^m)$. Then the corresponding structure on a manifold $W \subseteq M$ consists of a map $f:W \to B$ and a vector bundle isomorphism $\nu_W \cong f^*(E)$, where $\nu_W$ is the normal bundle of the embedding $W \subset M$.  
\end{example}

\section{Scanning} \label{sec:scanning}

\subsection{Segal's method of scanning}\label{sec:scan}

As in previous sections we fix a dimension $m$ for our ambient manifolds, and a dimension $d$ for our embedded manifolds. 
Let $(X, \xi)$ be a space with a $GL_m$-action and equivariant fiber 
bundle $\xi: X \to Gr_d(\R^m)$, as in Section~\ref{sect:embded_with_tangent}. 
As in that section this induces a bundle $\xi_M: X_M \to Gr_d(TM)$ for any $m$-manifold $M$, and hence we have a space $\psi_d^{(X, \xi)}(M)$ of manifolds embedded in $M$ equipped with $(X_M, \xi_M)$-structures. 
In this section we will review what is known about this space. 

First we consider the closely related space $\psi_d(\R^m; X, \xi)$ of $d$-manifolds manifolds embedded into $\R^m$ equipped with an $(X, \xi)$-structure. This space has a natural $GL_m$-action and hence for any $m$-manifold manifold $M$ we may form the associated bundle:
\begin{equation*}
	Fr(TM) \times_{GL_m} \psi_d(\R^m; X, \xi) \to M.
\end{equation*}
The fiber over $m \in M$ is identified with $\psi_d^{(X,\xi)}(T_mM)$ the space of manifolds embedded in $T_mM$ equipped with $(X, \xi)$-structures. We will denote this bundle $\psi_d^{(X, \xi), \text{fib}}(TM)$. 

Each fiber is equipped 
with a canonical basepoint corresponding to the empty manifold embedded in $T_mM$, and this gives rise to a canonical section of $\psi_d^{(X, \xi), \text{fib}}(TM)$, which we call the \emph{zero section}.
We let $\Gamma(M, \partial M; \psi_d^{(X, \xi), \text{fib}}(TM))$ denote the space of sections of $\psi_d^{(X, \xi), \text{fib}}(TM)$ which restrict to the zero section on $\partial M$. 

\begin{remark}
	The space of sections of the bundle $\psi_d^{(X, \xi), \text{fib}}(TM)$ is the first derivative of the functor $\psi_d^{(X,\xi)}(-)$ in the Goodwillie-Weiss manifold calculus.  
\end{remark}

Segal's method of `Scanning' will allow us to compare $\psi_d^{(X, \xi)}(M)$ and the space of sections $\Gamma(M, \partial M; \psi_d^{(X, \xi), \text{fib}}(TM))$, and by a result of Oscar Randal-Williams \cite{MR2764873} this comparison map is a weak equivalence when $M$ is open (has no compact components).  

The situation is slightly easier when $M$ is without boundary, and we treat that case first. 
\begin{definition}
	Suppose that $M$ is a manifold without boundary. Then a \emph{scanning exponential} for $M$ is a smooth map
	\begin{equation*}
		e: TM \to M
	\end{equation*}
	whose restriction to the zero section is the identity map and such that the restriction
	\begin{equation*}
		e|_{T_mM} = e_m: T_mM \to M
	\end{equation*}
	embeds $T_mM$ as an open neighborhood of $m\in M$. 
\end{definition}

A choice of scanning exponential induces the \emph{scanning map}
\begin{align*}
	\psi_d^{(X, \xi)}(M) \times M &\to \psi_d^{(X, \xi), \text{fib}}(TM) \\
	(W, m) & \mapsto e_m^{-1}(W) \subseteq T_mM
\end{align*}
which we regard as a map 
\begin{equation*}
	s_e:\psi_d^{(X, \xi)}(M) \to \Gamma(M;\psi_d^{(X, \xi), \text{fib}}(TM)).
\end{equation*}

When $M$ has boundary the set up is more complicated. First if $m \in \partial M$ it is not possible to embed $T_mM$ as an open subset of $M$ with the origin centered at $m$. Moreover we want the scanning map to give rise to a section which restricts to the zero section on $\partial M$. These issues can be resolved by adopting the following definition.

\begin{definition}\label{def:scanning exponential}
	Suppose that $M$ is a manifold, possibly with boundary. Then a \emph{scanning exponential} for $M$ is a smooth map
	\begin{equation*}
		e: TM \to M
	\end{equation*}
	satisfying the following requirements:
	\begin{enumerate}
		\item the restriction of $e$ to the zero section is the identity map on $M$;
		\item for each $m\in M \setminus \partial M$ on the interior, $e$ embeds $T_mM$ as an open neighborhood of $m$; 
		\item For each open neighborhood of the boundary $U \supseteq \partial M$ there exists a open set $V$, with $\partial M \subseteq V \subseteq U$ such that $e(T_mM) \subseteq U$ for each $m \in V$. 
	\end{enumerate}
\end{definition}
In other words we require our scanning exponential to degenerate near the boundary of $M$. The embedding of $T_mM$ becomes smaller and smaller as $m$ approaches the boundary. 


Given a scanning exponential, $e$, the induced \emph{scanning map} is defined by the assignment
 \begin{align*}
 	\psi_d^{(X, \xi)}(M) \times M &\to \psi_d^{(X, \xi), \text{fib}}(TM) \\
 	(W, m) & \mapsto \begin{cases}
 		e_m^{-1}(W) \subseteq T_mM & \text{if } m \notin \partial M \\
		\emptyset \subseteq T_mM & \text{if } m \in \partial M. 
 	\end{cases}
 \end{align*}
 
\begin{lemma}\label{lem:scan_cts}
	The above defined scanning map is continous. 
\end{lemma} 
 
\begin{proof}
	The domain is a space whose topology can be defined using plots, and hence it is enough to show that plots are mapped to continous maps. Moreover we already know the map is continuous on $\psi_d^{(X, \xi)}(M) \times (M \setminus \partial M)$.

A plot parametrized by the smooth manifold $Z$ consists of two parts. First there is a smooth map $f:Z \to M$. In addition we have a submanifold $W \subseteq Z \times M$ such that the projection $\pi: W \to Z$ is a submersion and $W$ is equipped with a $\pi$-fiberwise $(X,\xi)$-structure.  

The subspaces $Z \times \partial M$ and $W$ are disjoint closed subsets and so there exists an open neighborhood $U$ of $Z \times \partial M$ disjoint from $W$. Now it follows from 
Definition~\ref{def:scanning exponential} property (3)
that there exists an open neighborhood $V \subseteq Z$ of $f^{-1}(\partial M)$ such that for any $z \in V$, the image $e(T_{f(z)}M)$ of scanning exponential is contained in $U$. 

In particular this means that for an $z \in (V \setminus f^{-1}(\partial M))$, we have that $W_z = \pi^{-1}(z)$ is disjoint from $e(T_{f(z)}M)$, and hence the scanning map (composed with the given plot) restricts to the zero section on $(V \setminus f^{-1}(\partial M))$. Since the above scanning map simple extends by the zero section on $f^{-1}(\partial M)$, it follows that this is continous. 
	
Since this is true for all plots, it follows that the above scanning map itself is continuous. 
\end{proof}

\begin{theorem}[\cite{MR2764873}]\label{thm:h-prin}
	If $M$ has no compact components, then for each choice of scanning exponential, the scanning map induces a weak homotopy equivalence of spaces: 
	\begin{equation*}
		\psi_d^{(X,\xi)}(M) \to \Gamma(M, \partial M; \psi_d^{(X, \xi), \text{fib}}(TM)).
	\end{equation*} \qed
\end{theorem}

\noindent An important special case of the above is when $M = D^p \times \R^k$ ($k>0$) in which case we have a weak homotopy equivalence of spaces:
\begin{equation*}
	\psi_d^{(X,\xi)}(D^p \times \R^k) \simeq \Omega^p \psi_d(\R^{p+k}; X, \xi).
\end{equation*} 
 Note, however that this weak equivalence is not necessarily compatible with the natural $E_p$-algebra structures present on both spaces, see Section~\ref{sec:Ep-scanning}.

\subsection{Fiberwise Thom spaces}

Fix $E \to X$ a vector bundle. If $x \in X$ let $E_x$ denote the fiber of $E$ at $x$. The \emph{fiberwise one-point compactification} $E^\infty$ of $E$ is defined as follows. The underlying set is $E^\infty = E \sqcup X$, the set $E$ with an additional copy of $X$.  We will denote points $x \in X$ in the additional copy of $X$  by $\infty_x$, and it should be thought of as the `point at infinity in the fiber $E_x$'. We topologize $E^\infty$ as follows: a subset $V \subseteq E \sqcup X$ is open if $V \cap E$ is open in $E$ and if in addition for each $\infty_x \in V$, the intersection $V \cap E_x$ is compact. When $X = pt$ is the one point space, then $E^\infty$ is the usual one point compactification.

Suppose that we are given a map $\pi: X \to M$ to another space $M$. We will want to view $X$ as defining a family of spaces parametrizes by $M$. The space associated to $m \in M$ is $X_m = \pi^{-1}(m)$. The vector bundle $E$ restricts to a vector bundle over $X_m$ for which we can form the Thom space. The fiberwise Thom space assembles these into a family of spaces over $M$. 

\begin{definition}\label{def:fiberwiseThom}
	In the notation above, the \emph{fiberwise Thom space} $Th^\textrm{f.w.}_M(E)$ is defined to be the quotient of $E^\infty$ be the relation $\infty_x \sim \infty_{x'}$ whenever $\pi(x) = \pi(x') \in M$. 
\end{definition}

\noindent When $M = pt$, then we recover the usual Thom space.

\subsection{The space of embedded manifolds and Thom spaces}\label{sec:emb_thom}

Let $(X, \xi)$ be a space with a $GL_m$-action and equivariant fiber 
bundle $\xi: X \to Gr_d(\R^m)$, as above and in Section~\ref{sect:embded_with_tangent}. Let $\gamma_d$ be the tautological $d$-plane bundle on $Gr_d(\R^m)$ and let $\gamma_d^\perp$ denote the  complementary bundle (of dimension $m - d$). The fiber of $\gamma_d^\perp$ over the $d$-plane $L \subset \R^m$ is the quotient vector space $(\gamma^\perp_d)_L \cong \R^m/L$. 

The pullback $\xi^* \gamma_d^\perp$ is a vector bundle over $X$ and we may form the Thom space $Th(\xi^* \gamma_d^\perp)$. There is a pointed map 
\begin{equation*}
	L:Th(\xi^* \gamma_d^\perp) \to \psi_d(\R^m; X, \xi)
\end{equation*}
defined as follows. The base point $\{\infty\}$ of $Th(\xi^* \gamma_d^\perp)$ is mapped to the base point of $\psi_d(\R^m; X, \xi)$, the empty $d$-manifold embedded in $\R^m$. A point of $Th(\xi^* \gamma_d^\perp)$ which is not the base point consists of a point $x \in X$, and a vector $v \in \xi^*(\gamma_d^\perp)_x \cong (\gamma^\perp_d)_{\xi(x)} \cong \R^m / \xi(x)$.  This data specifies an affine subspace of $\R^n$
\begin{equation*}
	L_{(x,v)} = q^{-1}_x(v)
\end{equation*}
where $q_x: \R^m \to \R^m/\xi(x)$ is the quotient map by the subspace $\xi(x)$. This is the $d$-dimensional hyperplane of $\R^m$ which is parallel to $\xi(x)$, but offset by $v$. We regard it as a $d$-dimensional embedded submanifold of $\R^n$. The gauss map for this embedded submanifold $L_{(x,v)}$ is the constant map to $Gr_d(\R^m)$ taking value $\xi(x)$. We equip $L_{(x,v)}$ with an $(X, \xi)$ structure consisting of the constant map to $X$ with value $x$. The map $Th(\xi^* \gamma_d^\perp) \to \psi_d(\R^m; X, \xi)$ is given be sending $(x,v)$ to the submanifold $L_{(x,v)}$ with this $(X, \xi)$-structure. A simple inspection on plots shows that this map is continuous. In fact it is a weak equivalence. 

\begin{theorem}[{\cite[Lm.~3.8.1]{Ayala:2008aa}\cite[Th.~3.22]{MR2653727} }]\label{thm:space_to_thom_space}
		The map  
		\begin{equation*}
			L:Th(\xi^* \gamma_d^\perp) \to \psi_d(\R^m; X, \xi)
		\end{equation*}
		is a $GL_m$-equivariant weak homotopy equivalence. \qed
\end{theorem}

For each $m$-manifold $M$, the map $L$ induces a map of associated bundles, which by the above theorem is also a weak equivalence:
\begin{equation*}
	L: Th^\textrm{f.w.}_M(\xi^*_M \gamma_d^\perp) \to \psi_d^{(X, \xi), \text{fib}}(TM)
\end{equation*}
where $\xi_M: X_M \to Gr_d(TM)$ is the associated map (recall $X_M = Fr(TM) \times_{GL_m} X$), $\gamma_d^\perp$ is the complementary bundle on $ Gr_d(TM)$ to the tautological $d$-plane bundle $\gamma_d$, and $\xi^*\gamma_d^\perp$ is the pullback bundle to $X_M$.

Combining this with Theorem~\ref{thm:h-prin} we have:
\begin{corollary}[\cite{MR2764873}]
	If $M$ has no compact components, then we have weak homotopy equivalences:
	\begin{equation*}
		\psi_d^{(X,\xi)}(M) \stackrel{\sim}{\longrightarrow} \Gamma(M, \partial M; \psi_d^{(X, \xi), \text{fib}}(TM)) \stackrel{\sim}{\longleftarrow} \Gamma(M, \partial M; Th^\textrm{f.w.}_M(\xi^*_M \gamma_d^\perp))
	\end{equation*}
	In particular, if  $M = D^p \times \R^k$ ($k>0$), then we have weak homotopy equivalences:
 	\begin{equation*}
 		\psi_d^{(X,\xi)}(D^p \times \R^k) \stackrel{\sim}{\longrightarrow} \Omega^p \psi_d(\R^{p+k}; X, \xi)
		  \stackrel{\sim}{\longleftarrow} \Omega^p Th(\xi^*\gamma_d^\perp)
 	\end{equation*}
	where $Th(\xi^*\gamma_d^\perp)$ is the Thom space of the bundle $\xi^* \gamma_d^\perp$ over $X$. 
\end{corollary}

\section{$E_p$-operads and algebras} \label{sec:Ep-scanning}

The \emph{$E_p$-operad} is the operad of little $p$-cubes \cite{MR0420610}. The $k^\textrm{th}$ space $E_p(k)$ of this operad is the space of embeddings
\begin{equation*}
	\coprod_k D^p \to D^p
\end{equation*}
where $D^p = \{ (x_i) \in \R^n \; | \; |x_i| \leq 1 \}$ is the unit cube, and such that restricted to each component the embedding is \emph{rectilinear}. This means that it is given by the formula
\begin{equation*}
	(x_i) \mapsto (a_i x_i + b_i)
\end{equation*}
for real constants $a_i$ and $ b_i$, with $a_i >0$. Thus the set of embeddings can be viewed as a subset of $\R^{2pk}$, and we view it as a topological space using the subspace topology. 

An $E_p$-algebra (in $\Top$) is a (pointed) space equipped with an action of the $E_p$-operad. This means that we have a space $X$ and for each $k$ we have a $\Sigma_k$-equivariant composition:
\begin{equation*}
	E_p(k) \times X^k \to X.
\end{equation*}
See \cite{MR0420610} for details. 

If $\cF$ is a functor from $p$-manifolds to spaces which is co-variant for closed embeddings and a contravaraint sheaf for open embeddings, then under mild conditions the value on the unit $p$-cube, $\cF(D^p)$, is naturally an $E_p$-algebra. We will consider two important examples.

\begin{example}[$p$-fold loop spaces] \label{ex:loop_space}
	Fix a pointed topological space $(Z,*)$ then the relative mapping space functor $\cF(M) = \Map((M, \partial M), (Z, *))$ sends $M$ to the space of maps which restrict to the constant base-point map on $\partial M$.  We have $\cF(D^p) = \Omega^P Z$, the $p$-fold loop space of $Z$.
	
If $M \to M'$ is a closed embedding (of manifolds of the same dimension), then we get a map 
	\begin{equation*}
		\cF(M) \to \cF(M')
	\end{equation*}
	by extending $M \to Z$ to $M' \to Z$ by the constant map to the basepoint. This defines a continuous covariant functor for the category of manifolds and closed embeddings. The space $\Omega^p Z$ is naturally an $E_p$-algebra. 
\end{example}

\begin{example}[Spaces of Embedded Manifolds] \label{ex:emb_man}
	Fix a dimension $m$ and a tangential structure $\xi: X \to Gr_{d}(\R^{p+m})$.
	For any $m$-dimensional manifold $M$, we have a continous functor on $p$-manifolds $\psi_d^{(X,\xi)}(- \times M)$. Again the space of embedded submanifolds  $\psi_d(D^p \times M)$ is an $E_p$-algebra. 
\end{example}

Taking the product with the standard interval $D^1 = [-1,1]$ gives us a way to regard $n$-cubes as $(n+1)$-cubes, and this induces a homomorphism from the $E_n$ operad to the $E_{n+1}$ operad. The colimit is the \emph{$E_\infty$ operad}. It consists of componentwise rectilinear embeddings of infinite dimensional cubes with are trivial in all but finitely many variables. Infinite loop spaces are the prototypical example of $E_\infty$ algebras. 

We saw in Section~\ref{sec:scan} that for any choice of scanning exponential, the induced scanning map
\begin{equation*}
	\psi_d^{(X, \xi)}(D^p \times \R^m) \to \Omega^p\psi_d(\R^{p+m}; X, \xi)
\end{equation*}
is a weak equivalence (provided $m >0$). As we saw in Examples~\ref{ex:loop_space} and~\ref{ex:emb_man}, both spaces are $E_p$-algebras and it is natural to speculate that the they are weakly equivalent as $E_p$-algebras. To the author's knowledge this has not been shown in the literature for finite $p$.

One difficulty is that there is no choice of scanning exponential which is compatible with the action of the $E_p$-operad. To remedy this we will enlarge $\psi_d^{(X, \xi)}(D^p \times \R^m)$ with additional data which will determine a new scanning map. 

\subsection{Scanning functions and an $E_p$-equivalence}

\begin{definition}
	A \emph{scanning function} on $D^p$ is a $p$-tuple $(\varepsilon^{(i)})$ of smooth functions $\varepsilon^{(i)}: D^p \to \R_{\geq 0}$ such that $\varepsilon^{(i)}|_{\partial D^p}$ agrees with the zero function to all order. That is the $\infty$-jet $(j^\infty \varepsilon^{(i)})|_{\partial D^p}$ restricted to $\partial D^p$ agrees with the $\infty$-jet of the $p$-tuple of constant zero functions. 
			
	Fix an embedded manifold $W \subseteq D^p \times \R^m$ which is disjoint from $\partial (D^p \times \R^m)$. We will say that a scanning function $\varepsilon$ is \emph{compatible} with $W$ if $\varepsilon^{(i)}(x) >0$ for all $x \in W$ and all $1 \leq i \leq p$.

Let $\Psi^{(X,\xi)}_d(D^p \times \R^m)$ be  defined as the space of all pairs $(W, \varepsilon)$ consisting of an embedded manifold $W \in \psi^{(X,\xi)}_d(D^p \times \R^m)$ and a compatible scanning function $\varepsilon$. 
\end{definition}

\begin{remark}
	There are several variations one can imagine for the notion of scanning function. The one we are using has the advantage that both (1) we can extend the domain of any scanning function $\varepsilon^{(i)}$ to all of $\R^p$ by extending by the $p$-tuple of zero functions outside of $D^p$, and (2) there exist scanning functions such $\varepsilon^{i}(x) > 0$ whenever $x$ is on the interior of $D^p$. Such scanning functions are compatible with all embedded manifolds $W$ disjoint from the boundary. 	
\end{remark}

\begin{example} \label{ex:box_scanning_function}
	Let $(x_1, \dots, x_p)$ be standard coordinates on 	$\R^p$. Define a function 
	\begin{equation*}
		g(x_1, \dots, x_p) = \prod_{i=1}^p (1 - x_i^2)
	\end{equation*}
	There is a scanning function $\varepsilon = \{ \varepsilon^{(i)}$ with $\varepsilon^{(1)}(x) = \varepsilon^{(2)}(x) = \cdots = \varepsilon^{(p)}(x)$ equal to the following function:
	\begin{equation*}
		\varepsilon^{(i)}(x_1) = e^{- \frac{1}{g(x)}}.
	\end{equation*}
	 This scanning function is non-zero on the interior of $D^p$ and hence is compatible with all closed embedded manifolds disjoint from $\partial D^p \times \R^m$. 
\end{example}

The space $\Psi^{(X,\xi)}_d(D^p \times \R^m)$ is naturally an $E_p$-algebra, which we can see as follows. The action on the space of manifolds is as it is on $\psi^{(X,\xi)}_d(D^p \times \R^m)$. We need only describe what happens on the scanning functions. 
Suppose that $\coprod_k D^p \to D^p$ is an embedding such that each disk is embedded rectilinearly. Suppose also that we are given $k$-many scanning functions $(\varepsilon_j^{(i)})$ on $D^p$ ($1 \leq j \leq k$ and $1 \leq i \leq p$). Then we get a new scanning function $(\varepsilon^{(i)})$ on $D^p$ as follows. Outside of the image of $\coprod_k D^p$, each $\varepsilon^{(i)} \equiv 0$ is identically zero for all $1 \leq i \leq p$. Inside the $j^\text{th}$ embedded $D^p$ the scanning function agrees with $(\varepsilon_j^{(i)})$ composed with the inverse of the rectilinear embedding. 

The forgetful map $u:\Psi^{(X,\xi)}_d(D^p \times \R^m) \to \psi^{(X,\xi)}_d(D^p \times \R^m)$ is a map of $E_p$-algebras by construction. 

\begin{lemma}\label{lem:forget_equiv}
	The forgetful map $u:\Psi^{(X,\xi)}_d(D^p \times \R^m) \to \psi^{(X,\xi)}_d(D^p \times \R^m)$ is an acyclic Serre fibration and hence a weak equivalence of $E_p$-algebras.
\end{lemma}

\begin{proof}
We will show that for any commutative square, as below, we can solve the indicated lifting problem:	
	\begin{center}
	\begin{tikzpicture}
			\node (LT) at (0, 1.5) {$ \partial D^k $};
			\node (LB) at (0, 0) {$D^k $};
			\node (RT) at (3, 1.5) {$ \Psi^{(X,\xi)}_d(D^p \times \R^m)$};
			\node (RB) at (3, 0) {$\psi^{(X,\xi)}_d(D^p \times \R^m) $};
			\draw [->] (LT) -- node [left] {$ $} (LB);
			\draw [->] (LT) -- node [above left] {$ $} (RT);
			\draw [->] (RT) -- node [right] {$ u$} (RB);
			\draw [->] (LB) -- node [below] {$ $} (RB);
			\draw [->, dashed] (LB) -- (RT);
	\end{tikzpicture}
	\end{center}
The data of such a lift consists of an assignment for each $x \in D^k$ of a compatible scanning function $\varepsilon_x$ such that $\varepsilon_x$ agrees with the specified lift on the boundary. 

We will view $D^k = \partial D^k \times [0,1] / \sim$ where $(x, 1) \sim (x', 1)$ for any $x, x'$. Let $\overline \varepsilon$ be the family of scanning functions on $\partial D^k$ specified by the initial lift. Let $\tilde{\varepsilon}$ be any scanning function which is compatible with all embedded manifolds (i.e. $\tilde{\varepsilon}^{(i)}(y) > 0$ for any interior point $y \in D^p$ and for all $1 \leq i \leq p$), for example the scanning function in Example~\ref{ex:box_scanning_function}. 
Then the desired family of scanning functions is given by:
\begin{equation*}
	\varepsilon_{(x, t)} = t \cdot \tilde{\varepsilon} + (1-t) \cdot \overline \varepsilon_x.
\end{equation*}
For all $t>0$, $\varepsilon_{(x, t)}$ is compatible with all embedded manifolds, and hence this does define a lift. 
\end{proof}

\subsection{A scanning map for $\Psi_d^{(X, \xi)}(D^p \times \R^m)$}

We will now describe a map 
\begin{equation*}
	s:\Psi_d^{(X, \xi)}(D^p \times \R^m) \to \Omega^p \psi_d(\R^{p + m}; X, \xi)
\end{equation*}
which is a variation on the scanning map in Section~\ref{sec:scan}. Given a point $z = (z_1, \dots, z_p) \in D^p$ and a $p$-tuple of positive real numbers $\varepsilon = (\varepsilon_1, \dots, \varepsilon_p)$ we have an embedding:
\begin{equation*}
	\phi_{x, \varepsilon}: \R^{p + m} \hookrightarrow \R^{p + m} 
\end{equation*}
In the $i^\textrm{th}$-coordianate $\phi_{x, \varepsilon}$ is given by 
\begin{equation*}
	(x_1, \dots, x_p, x_{p+1}, \dots, x_{p+m}) \mapsto \begin{cases}
		z_i + \varepsilon_i \arctan(x_i) & \text{if } 1 \leq i \leq p \\
		x_i & \text{if } p+1 \leq i \leq p+m
	\end{cases}.
\end{equation*}
In other words $\phi_{x, \varepsilon}$ is a diffeomorphism between $\R^{p + m}$ and the open box centered at $x$ with sides of length $2 \varepsilon_i$ on the $i^\textrm{th}$-coordinate direction (and infinite in the $\R^m$-directions). 

Given an embedded manifold $W \subseteq D^p \times \R^m$, we regard it as a manifold embedded in $\R^{p \times m}$ using the standard closed embedding of $D^p \subseteq \R^p$ as the unit cube. Then $s$ is adjoint to the map:
\begin{align*}
	\tilde{s}:\Psi_d^{(X, \xi)}(D^p \times \R^m) \times D^p &\to  \psi_d(\R^{p + m}; X, \xi) \\
	(((W, \theta), \varepsilon, x)) & \mapsto \begin{cases}
		\phi_{x, \varepsilon(x)}^{-1}(W, \theta) & \text{if } \varepsilon^{(i)}(x) > 0 \text{ for all} 1 \leq i \leq p \\
		\emptyset & \text{otherwise}
	\end{cases}
\end{align*}

\begin{lemma}
	The map $\tilde{s}$ and its adjoint $s: \Psi_d^{(X, \xi)}(D^p \times \R^m) \to \Omega^p \psi_d(\R^{p + m}; X, \xi)$ are continuous maps. 
\end{lemma}

\begin{proof}
	The proof is the same as for Lemma~\ref{lem:scan_cts}.
\end{proof}

\begin{lemma}
	The map $s: \Psi_d^{(X, \xi)}(D^p \times \R^m) \to \Omega^p \psi_d(\R^{p + m}; X, \xi)$ is a weak homotopy equivalence. 
\end{lemma}

\begin{proof}
Let $\varepsilon_0$ be the scanning function from Example~\ref{ex:box_scanning_function}. This scanning function is non-zero on the interior of $D^p$ and hence defines a section of the forgetful map 
\begin{equation*}
	u:\Psi_d^{(X, \xi)}(D^p \times \R^m) \to \psi_d^{(X, \xi)}(D^p \times \R^m)
\end{equation*}
where under this section an embedded manifold $W$ is sent to $(W, \varepsilon_0)$.

The composite 
\begin{equation*}
	\psi_d^{(X, \xi)}(D^p \times \R^m) \to \Psi_d^{(X, \xi)}(D^p \times \R^m) \stackrel{s}{\to} \Omega^p \psi_d(\R^{p + m}; X, \xi)
\end{equation*}
is adjoint to the map
\begin{align*}
	\psi_d^{(X, \xi)}(D^p \times \R^m) \times D^p &\to \psi_d(\R^{p + m}; X, \xi) \\
	((W, \theta), x) \mapsto \begin{cases}
		\phi_{x, \varepsilon_0(x)}^{-1}(W, \theta) & x \not \in \partial D^p \times \R^m \\
		\emptyset & x \in \partial D^p \times \R^m
	\end{cases}
\end{align*}
But this map is the scanning map from Section~\ref{sec:scan} associated to the scanning exponential
\begin{align*}
	e: T(D^p \times \R^m) = \R^{p+m} \times (D^p \times \R^m) &\to D^p \times \R^m \\
	(y, x) & \mapsto \begin{cases}
		\phi_{x, \varepsilon_0(x)}(y) & \text{if } x \not \in \partial D^p \times \R^m \\
		x & \text{if } x  \in \partial D^p \times \R^m \
	\end{cases} 
\end{align*}

The section of the forgetful map is a weak equivalence by Lemma~\ref{lem:forget_equiv}, and the composite map, being the scanning map from Section~\ref{sec:scan}, is a weak equivalence by Theorem~\ref{thm:h-prin}. It follows from the two-out-of-three property that $s$ is a weak equivalence. 
\end{proof}

\begin{theorem}\label{thm:E_pSpaces}
		We have natural weak equivalences of $E_p$-algebras:
		\begin{center}
		\begin{tikzpicture}
			\node (LB) at (0, 0) {$\psi^{(X,\xi)}_d(D^p \times \R^m)  $};
				\node (LT) at (3, 1.5) {$ \Psi^{(X,\xi)}_d(D^p \times \R^m) $};
				\node (RB) at (6, 0) {$\Omega^p \psi_d(\R^{p + m}; X, \xi) $};
				\node (RT) at (9, 1.5) {$\Omega^p Th(\xi^* \gamma_d^\perp) $};
				\draw [-Computer Modern Rightarrow] (LT) -- node [below right] {$u $} (LB) node [midway, above, sloped]  {$\sim$};
				\draw [-Computer Modern Rightarrow] (LT) -- node [below left] {$s $} (RB) node [midway, above, sloped]  {$\sim$};
				\draw [-Computer Modern Rightarrow] (RT) -- node [below right] {$ L$} (RB) node [midway, above, sloped]  {$\sim$};
		\end{tikzpicture}
		\end{center}

\end{theorem}

\begin{proof}
	The left-most map is a weak equivalence of $E_p$-algebras by Lemma~\ref{lem:forget_equiv}. The right most map is a weak equivalence of $E_p$-algebras by Theorem~\ref{thm:space_to_thom_space}.  The middle map was shown to be a weak equivalence in the previous lemma, and so all that remains is to show that it is an $E_p$-algebra map. 
	
The map $s$ is an $E_p$-algebra map by design. For suppose we are given a rectilinear embedding $\sqcup_k D^p \to D^p$ and a $k$-tuple of elements of $\{(W_k, \theta_k, \varepsilon_k) \}$ of elements of $\Psi^{(X,\xi)}_d(D^p \times \R^m)$. The $E_p$-composition is a new element $(W, \theta, \varepsilon)$ in $\Psi^{(X,\xi)}_d(D^p \times \R^m)$. The scanning function $\varepsilon$ is the constant function zero outside the images of the $k$-little disks embedded in $D^p$. Hence $(W, \theta, \varepsilon)$ is mapped via $s$ to a $p$-fold loop in $\psi_d(\R^{p + m}; X, \xi)$ which is the constant base-point valued loop outside the images of the $k$-little disks embedded in $D^p$. Inside the images, the scanning functions $\varepsilon_k$ are shifted and scaled in precisely the same way and the corresponding $p$-fold loops. Hence $s$ is an $E_p$-homomorphism. 	
\end{proof}

\section{The Bordism $n$-category}\label{sec:Bordncat}

\subsection{$n$-Fold Segal spaces}

\subsubsection{Segal Spaces} \label{sec:Segal_space_subsection}

Segal spaces are a homotopical weakening of the notion of nerve of a category. The category of simplicial sets is often used as the model of space the context of Segal spaces, but here we will use a variant using actual topological space. Specifically, $\Top$ will mean the category of  $\Delta$-generated topological spaces. With the weak homotopy equivalences this forms a combinatorial Cartesian simplicial Quillen model category with fibrations the Serre fibrations \cite{dugger_dgs}. All homotopy pull-backs will refer to this model structure.

The {\em spine} $S_n$ of the simplex $\Delta[n]$ is a sub-simplicial set consisting of the union of all the consecutive 1-simplices. There is the natural inclusion of simplicial sets
\begin{equation*}
	s_n: S_n = \Delta^{\{0,1\}} \cup^{\Delta^{\{1\}}} \Delta^{\{1,2\}} \cup^{\Delta^{\{2\}}} \cdots  \cup^{\Delta^{\{n-1\}}}  \Delta^{\{n-1,n\}} \to \Delta[n],
\end{equation*}
which corepresents the $n^\textrm{th}$ {\em Segal map}:
\begin{equation*}
	s_n: Z_n \to Z(S_n) = Z_1 \times_{Z_0}  Z_1 \times_{Z_0} \cdots \times_{Z_0}  Z_1.
\end{equation*}
Here $Z$ is a simplical space (or any simplicial object in a complete category). 

Recall that a simplicial set is isomorphic to the nerve of a category if and only if each Segal map is a bijection for $n \geq 1$. Moreover the full subcategory of simplical sets satisfying this property is equivalent to the category of small categories and functors. 

Also recall that given a pull-back diagram of spaces we can form both the fiber product $X \times_Y Z$, and the homotopy fiber product $X \times_Y^{h} Z$. There is a map
\begin{equation*}
	X \times_Y Z \to X \times_Y^{h} Z, 
\end{equation*}
which is well defined up to homotopy. So for example in a simplicial space $X$ (i.e. a functor $X : \bDelta^{op} \to \Top$) the Segal maps induce composite maps:
\begin{equation*}
	X_n \to X_1 \times_{X_0} \cdots \times_{X_0} X_1 \to X_1 \times^h_{X_0} \cdots \times^h_{X_0} X_1.
\end{equation*}

\begin{definition}
A \emph{Segal Space} is a simplicial space $X: \bDelta^\op \to \Top$ such that:
\begin{itemize}
\item \textbf{Segal Condition.} For each $n>0$ the Segal map induces a weak homotopy equivalence
\begin{equation*}
	s_n: X_n \xrightarrow{\simeq} \underbrace{ X_1 \times^{h}_{X_0} X_1 \times^{h}_{X_0} \dotsb \times^{h}_{X_0} X_1 \times^{h}_{X_0} X_1 }_{n \text{ factors}}.
\end{equation*}
\end{itemize}
\end{definition}
\noindent The Segal condition guarantees that we have a notion of composition which is coherent up to higher homotopy.


\begin{lemma}
	Segal spaces enjoy the following closure properties:
	\begin{enumerate}
		\item If $X$ is a Segal space and $Y$ is a simplicial space which is levelwise weakly equivalent to $X$ (meaning there is a finite zig-zag of levelwise weak equivalences of simplicial spaces connecting $X$ and $Y$), then $Y$ is also a Segal space. 
		\item If $X$, $Y$, and $Z$ are Segal spaces and $X \to Y$, $Z \to Y$ are any maps, then $X \times_Y^h Z$ is a Segal spaces, where the later denotes the levelwise homotopy fiber product of simplicial spaces. \qed
	\end{enumerate}	
\end{lemma}

\begin{definition}
	A map $X \to Y$ of Segal spaces is a \emph{weak equivalence} if it is a levelwise weak equivalence, equivalently if $X_0 \to Y_0$ and $X_1 \to Y_1$ are weak equivalences of spaces.  
\end{definition}


There is a good theory of $(\infty,1)$-categories based off of Segal spaces, but this requires considering Segal spaces which satisfy a further axiom. This additional axiom, called \emph{completeness} (or \emph{univalence}).

Let $K$ be the simplicial set given by the pushout square:
\begin{center}
\begin{tikzpicture}
		\node (LT) at (0, 1.5) {$\Delta^{\{0,2\}} \sqcup \Delta^{\{1,3\}}$};
		\node (LB) at (0, 0) {$\Delta^0 \sqcup \Delta^0$};
		\node (RT) at (4, 1.5) {$\Delta^{[3]}$};
		\node (RB) at (4, 0) {$K$};
		\draw [->] (LT) -- node [left] {$ $} (LB);
		\draw [->] (LT) -- node [above] {$ $} (RT);
		\draw [->] (RT) -- node [right] {$ $} (RB);
		\draw [->] (LB) -- node [below] {$ $} (RB);
		\node at (3.5, 0.5) {$\lrcorner$};
\end{tikzpicture}
\end{center}
The space of maps from $K$ into a simplcial space $X$ is given by a fiber product
\begin{equation*}
	\Maps(K, X) = (X_0 \times X_0) \times_{X_1 \times X_1} X_3.
\end{equation*}
The {derived} space of maps is given by $\Maps^h(K, X) = (X_0 \times X_0) \times_{X_1 \times X_1}^h X_3$. 

\begin{definition}\label{def:univalent_segal_space}
	A Segal space $X$ is called \emph{complete} (also called \emph{univalent}) if the canonical map (induced by $K \to pt$)
	\begin{equation*}
		X_0 \to \Maps^h(K, X)
	\end{equation*}
	is a weak homotopy equivalence.  
\end{definition}

\noindent As before if $X$ is sufficiently fibrant then we may use $\Maps$ instead of $\Maps^h$. We will not need to consider complete Segal spaces, but we include the definition since it is crucial for a general theory of $(\infty,1)$-categories based on Segal spaces.

\subsubsection{$n$-Fold Segal Spaces}\label{sec:n_fold_segal}

An $n$-fold simplicial space is a functor $(\bDelta^\op)^n \to \Top$, and these form a category $\sTop_n$ (we will drop the subscript in the case $n=1$). We will denote the objects of $(\bDelta)^n$ as products:
\begin{equation*}
	\Delta^{[k_1]} \boxtimes \Delta^{[k_2]} \boxtimes \cdots \boxtimes \Delta^{[k_n]} 
\end{equation*} 
or more briefly $[k_1] \boxtimes [k_2] \boxtimes \cdots \boxtimes [k_n]$. This notation extends to an assembly functor:
\begin{equation*}
	(-) \boxtimes \cdots \boxtimes (-): (\sTop)^{\times n} \to \sTop_n
\end{equation*}
which is the unique functor sending $([k_1], \dots, [k_n])$ to $[k_1] \boxtimes \cdots \boxtimes [k_n]$, and which commutes with colimits separately in each variable. The value of an $n$-fold simplicial space $X$ on $[k_1] \boxtimes \cdots \boxtimes [k_n]$ will be denoted $X_{k_1, \dots, k_n}$.

An $n$-fold simplicial space $X$ will be called \emph{essentially constant} if the canonical map
\begin{equation*}
	X_{0, \dots, 0} \to X_{k_1, \dots, k_n}
\end{equation*}
is a weak equivalence for all $[k_1] \boxtimes \cdots \boxtimes [k_n] \in (\bDelta)^n$. By convention a 0-fold simplicial space is simply a space and is always regarded as essentially constant. 

By adjunction we can equivalently regard an $n$-fold simplicial space as a simplicial object in $(n-1)$-fold simplicial spaces. This can be done in each coordinate, but we will make to following convention which avoids ambiguity. If $X$ is an $n$-fold simplicial space then $X_i$ will denote the $(n-1)$-fold simplicial space determined by:
\begin{equation*}
	X_i: [k_1] \boxtimes \cdots \boxtimes [k_{n-1}] \mapsto X_{i, k_1, \dots, k_{n-1}}. 
\end{equation*}

\begin{definition}
An \emph{$n$-fold Segal Space} is an $n$-fold simplicial space $X$ (i.e. a functor $X : \bDelta^{\op} \to \Fun((\bDelta^\op)^{\times n-1}, \Top)$) such that
\begin{itemize}
\item \textbf{Local.} $X_n$ is is an $(n-1)$-fold Segal space for each $n\geq 0$;
\item \textbf{Globularity.} The Segal space $X_0$ is essentially constant;
\item \textbf{Segal Condition.} For each $n>0$ the Segal map induces a levelwise weak homotopy equivalence
\begin{equation*}
	s_n: X_n \xrightarrow{\simeq} \underbrace{ X_1 \times^{h}_{X_0} X_1 \times^{h}_{X_0} \dotsb \times^{h}_{X_0} X_1 \times^{h}_{X_0} X_1 }_{n \text{ factors}}.
\end{equation*}
Here these homotopy fiber products of $(n-1)$-fold Segal spaces are taken levelwise. 
\end{itemize}
\end{definition}

\noindent If we replace homotopy fiber products in the the Segal condition with ordinary fiber products, then we will say that $X$ satisfies the \emph{strict Segal condition}.  

%

\begin{lemma}
	$n$-Fold Segal spaces 
	enjoy the following closure properties:
	\begin{enumerate}
		\item  If $X$ is an $n$-fold 
		Segal space and $Y$ is an $n$-fold simplicial space which is levelwise weakly equivalent to $X$ (meaning there is a finite zig-zag of levelwise weak equivalences of simplicial spaces connecting $X$ and $Y$), then $Y$ is also an $n$-fold Segal space.
		\item If $X$, $Y$, and $Z$ are $n$-fold Segal spaces 
		and $X \to Y$, $Z \to Y$ are any maps, then $X \times_Y^h Z$ is an $n$-fold Segal space 
		where the later denotes the levelwise homotopy fiber product of simplicial spaces. \qed
	\end{enumerate}
\end{lemma}

\subsection{Notation for Bordism $n$-categories}

The goal of this section is to carefully define the higher bordism categories as an $n$-fold Segal space. There are quite a few variations on the bordism category that we will need to consider simultaneously. For example we will want to vary the category number of our bordism $n$-category; our bordisms will be equipped with embeddings into an ambient manifold, which we will want to vary; we will want to consider unstable bordism categories which are not symmetric monoidal (i.e. $E_\infty$), but which retain an $E_p$-monoidal structure; and our bordism will be equipped with tangential structures as in section \ref{sect:embded_with_tangent}.

To keep track of all of these variations we will need to develop a consistent notation. There will be several variables and it is the goal of this section to define and explain the meaning of all these variables, the parameters used to specify the higher bordism categories. 
Let us begin: 
\begin{itemize}
	\item The category number of our higher bordism category will be denoted $n$. Specifically this means that we will be considering an $(\infty,n)$-category of bordisms.  
	\item The maximal dimension of the bordisms in our higher bordism category will be denoted by $d$. Hence the minimal dimension of the bordisms involved will be $(d-n)$.
	\item We will have an ambient manifold $M$ (of dimenions $\dim M = m$) into which our bordisms will be embedded. More specifically they will be embedded into the product of $M$ and a Euclidean space of appropriate dimension. The manifold $M$ is allowed to be non-compact and to have boundary. 
	If $M$ is non-compact then the embedded submanifolds can be `deformed off to infinity' in the non-compact directions of $M$, and if $M$ has boundary, then we require our embedded manifolds to always be disjoint from this boundary, as in section \ref{sec:space_of_man}.
	
	The $E_p$ monoidal structure arises when $M = D^p \times N$ is a product with the $p$-disk. 
	
	\item Our bordisms will be equipped with tangential structures, such as framings, orientations, spin structures, etc. The type of tangential structure is specified by a $GL_{m+n}$-equivariant
	fibration $\xi: X \to Gr_{d}(\R^{m+n})$, and the corresponding tangential/normal structure is called an $(X, \xi)$-structure. See Section~\ref{sect:embded_with_tangent} for details. 
\end{itemize}

\noindent These conventions are neatly summarized in the Table~\ref{tab:Notation}:
\begin{table}[h]
  \centering
 \begin{tabular}{| c | l |}
 	\hline
 	variable & meaning \\ \hline
 	$n$ & category number \\
 	 $d$ & maximal dimension of our bordisms \\
 	 $M$ & ambient manifold \\
 	 $(X, \xi)$ & fibration defining tangential structures  \\
 	 \hline
 \end{tabular}
  \caption{Summary of the parameters specifying the higher bordism category.}
  \label{tab:Notation}
\end{table}

\noindent Our principal object of study will be the $n$-fold multisimplicial space 
\begin{equation*}
	\Bord_{d;n}^{(X, \xi)}(M)
\end{equation*}
This is the $(\infty, n)$-category  of $d$-dimensional $(X, \xi)$-bordisms with embeddings into $M$. It is a particular $n$-fold fold Segal space, which we will define in complete precision in the next section, but we can think of as an $(\infty, n)$-category, where philosophically it has:
	\begin{itemize}
		\item objects which are $(d-n)$-manifolds embedded into $M$;
		\item 1-morphisms which are $(d-n+1)$-dimensional bordisms embedded into $M \times I$;
		\item 2-morphisms which are $(d-n+2)$-dimensional bordisms between bordisms embedded into $M \times I^2$;
		\item ...
		\item $n$-morphisms which are $d$-dimensional bordisms between bordisms between ... embedded into $M \times I^n$.
	\end{itemize}
	Moreover everything is equipped with an $(X,\xi)$-structure.

\subsection{Bordism n-categories} \label{sec:bordncat}

Now we turn to the precise definition of the $n$-fold multisimplicial space $\Bord_{d;n}^{(X, \xi)}(M)$ as a functor:
\begin{equation*}
	\Bord_{d;n}^{(X, \xi)}(M): (\bDelta^\op)^n \to \Top.
\end{equation*}
The objects of $(\bDelta^\op)^n$ will be denoted $[\Bm]$ where $\Bm = (m_1, \dots, m_n)$ is an $n$-tuple of natural numbers.
Thus to define $\Bord_{d;n}^{(X, \xi)}(M)$ we must specify a collection of spaces $\Bord_{d;n}^{(X, \xi)}(M)_{[\Bm]}$ together with face and degeneracy maps. 

To aid in this we will need some further notation. Let
\begin{equation*}
	\R^{[k]} = \{ (t_i)_{i=0}^{i=k} \; | \; t_i \leq t_{i+1} \} \subseteq \R^{k+1}
\end{equation*}
denote the space of order preserving maps from the poset $[k] \in \bDelta$ to $(\R, <)$. An element $\Bt \in \R^{[k]}$ consists of a $(k+1)$-tuple of real numbers $\Bt=(t_0, t_1, \dots, t_k)$  satisfying $t_i \leq t_{i+1}$ for $0 \leq i < k$. 

A point in the space $\Bord_{d;n}^{(X, \xi)}(M)_{[\Bm]}$ includes an element $\Bt^i \in \R^{[m_i]}$ for each $1 \leq i \leq n$. These numbers specify various hyperplanes in $\R^n$, and the space $\Bord_{d;n}^{(X, \xi)}(M)_{[ \Bm]}$ is built as a subspace of $\psi_d^{(X, \xi)}(M \times  \R^n)$ of submanifolds which satisfy certain cylindricality conditions with respect to these hyperplanes. 

\begin{definition}\label{def:cylindrical}
	Let $M_i$, $i = 1,2$ be manifolds of dimension $k_i$. Let $U \subseteq M_1$ be an open set. Then $W \subseteq M_1 \times M_1$ in $\psi_d(M_1 \times M_2)$ is \emph{cylindrical over $U$} if there exists a manifold $W_0 \in \psi_{d - k_1}(M_2)$ such that
	\begin{equation*}
		W \cap U \times M_2 = U \times W_0 \in \psi_d(U \times M_2)
	\end{equation*}
	as elements of $\psi_d(U \times M_2)$. If $Z \subseteq M_1$ is any subset, then we will say that $W$ is \emph{cylindrical near $Z \subseteq M_1$} if there exists an open neighborhood $U$ of $Z$ so that $W$ is cylindrical over $U$. 
\end{definition}

\noindent The above definition supposes a splitting of the ambient manifold $M = M_1 \times M_2$. In the case of the bordism category, the relevant ambient manifold is $M \times \R^n$ which may be split in many different ways. Indeed, to define $\Bord_{d;n}^{(X, \xi)}(M)$ we will have to use the above definitions for several different splittings. We will always be careful to specify which manifold $M_1$ the subspace $Z$ is contained in, thereby implicitly specifying the splitting $M \times \R^n \cong M_1 \times M_2$.  

We are now ready to define the bordism category. 

\begin{definition}\label{def:Bordismncat}
	Fix natural numbers $n,d$, an ambient manifold $M$, and a $GL_{m+n}$-equivariant fibration $(X, \xi)$, as in Table~\ref{tab:Notation}. Then the functor
	\begin{equation*}
		\Bord_{d;n}^{(X, \xi)}(M):  (\bDelta^\op)^n \to \Top
	\end{equation*}
	is defined by assigning to $[ \Bm] \in  (\bDelta^\op)^n$ the space consisting of tuples
	$( (\Bt^i)_{i=1}^n, (W, \theta))$ where $\Bt^i \in \R^{[m_i]}$  for each $1 \leq i \leq n$ and $(W, \theta) \in \psi^{(X, \xi)}_d( M  \times \R^n)$ is an embedded submanifold of $M \times \R^n$ with $(X,\xi)$-structure. These are required to satisfy the following condition:
	\begin{itemize}
		\item \textbf{Globular.}  For all $ 1 \leq i \leq n$, and $0 \leq j \leq m_i$, $W$ is cylindrical near 
			\mbox{$\{t^i_j\} \times \R^{\{i+1, \dots, n\}} \subseteq \R^{\{i, i+1, \dots, n\}}$}; 
	\end{itemize}
\end{definition}

The topology on the space $\Bord_{d;n}^{(X, \xi)}(M)_{[ \Bm]}$ is defined, just as before, by specifying a collection of smooth plots. Such a plot, parametrized by a smooth manifold $U$, consists of a smooth function
\begin{equation*}
	 (\Bt^i)_{i=1}^n: U \to  \prod_{i = 1}^n \R^{[m_i]}
\end{equation*}
and a smooth plot $p: U \to \psi^{(X,\xi)}_d(M \times \R^n)$ such that the globular conditions are satisfied for each $u \in U$. 

\pgfdeclareimage[height=5cm, width=8cm]{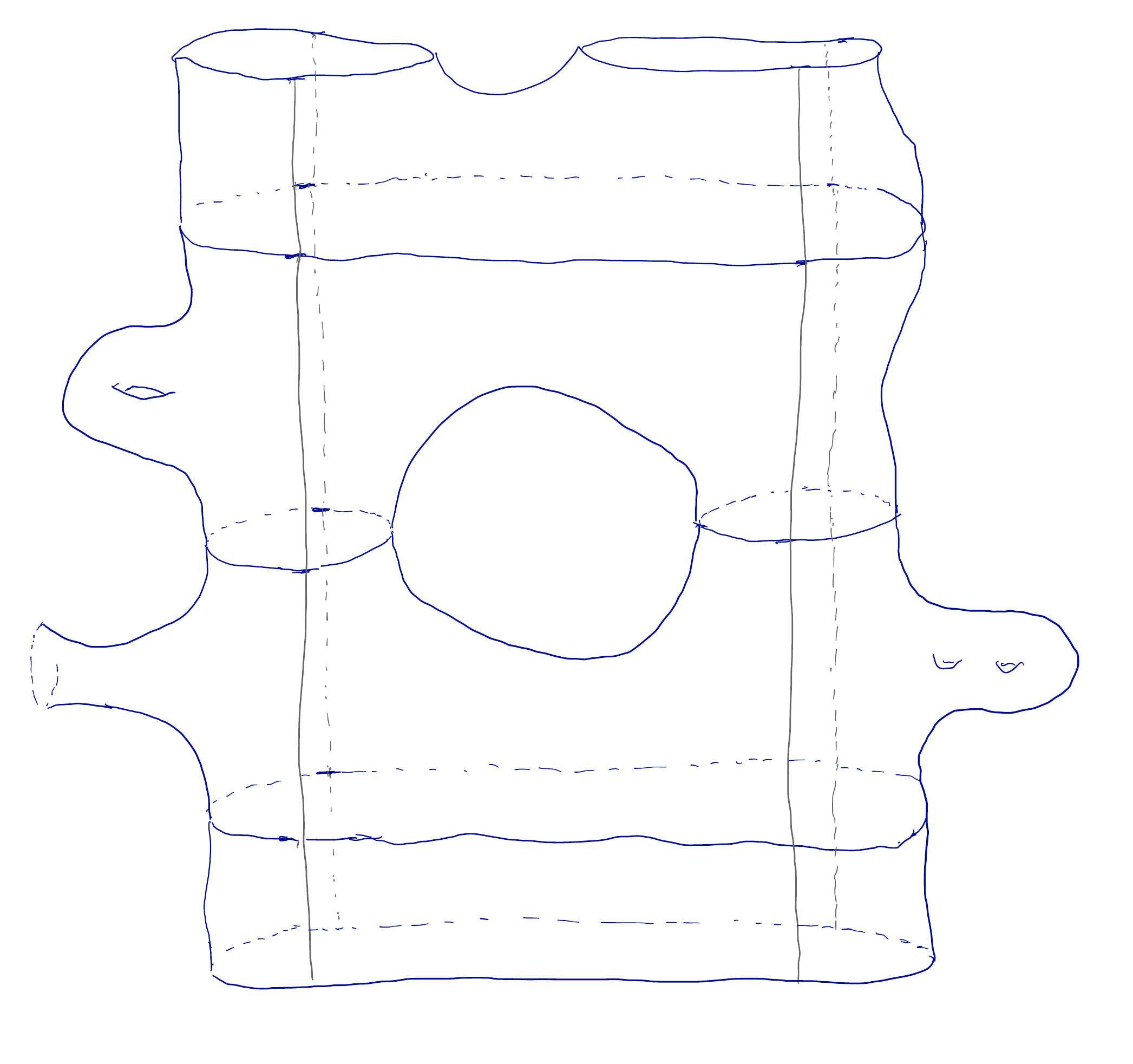}{Bord2-pic}

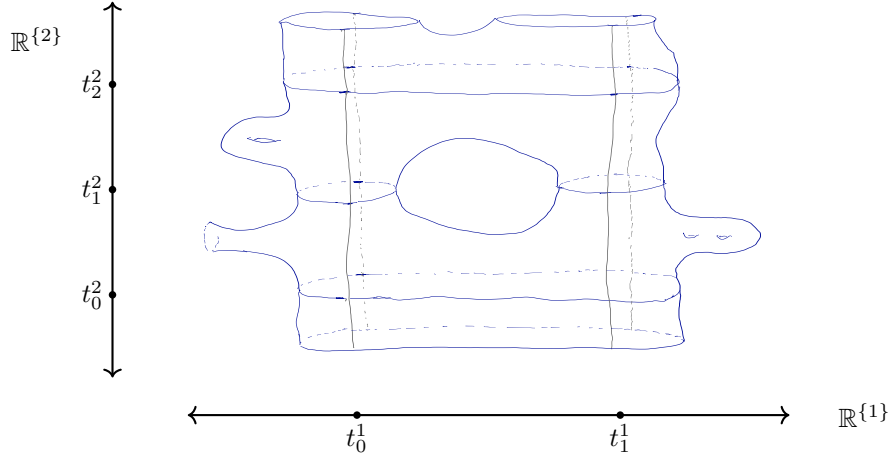
\begin{figure}[ht]
	\centering
	\begin{tikzpicture}
			\node (A) at (0, 0) {\pgfuseimage{Bord2-pic}}; 
			\draw [thick, <->] (-5, -2.5) -- (-5,2.5);
			\draw [thick, <->] (-4, -3) -- (4,-3);
			\node [fill = black, inner sep = 1pt, circle] (B0) at (-1.75,-3) {};
			\node at (B0) [below] {$t^1_0$};
			
			\node [fill = black, inner sep = 1pt, circle] (B1) at (1.75,-3) {};
			\node at (B1) [below] {$t^1_1$};
			
			\node [fill = black, inner sep = 1pt, circle] (C0) at (-5,-1.4) {};
			\node at (C0) [left] {$t^2_0$};
			
			\node [fill = black, inner sep = 1pt, circle] (C1) at (-5,0) {};
			\node at (C1) [left] {$t^2_1$};
			
			\node [fill = black, inner sep = 1pt, circle] (C2) at (-5,1.4) {};
			\node at (C2) [left] {$t^2_2$};
			\node at (5,-3) {$\R^{\{1\}}$};
			\node at (-6, 2) {$\R^{\{2\}}$};
	\end{tikzpicture}	
	\caption{A point in the $[1,2]$-space of $\Bord_2(D^1)$.  }
	\label{fig:label}
\end{figure}

\section{Realizations of Bordism $n$-categories}\label{sec:real_bordn}

\subsection{The main theorem}

Fix natural numbers $n,d$, an ambient manifold $M$, and a $GL_{m+n}$-equivariant fibration $(X, \xi)$, as in Table~\ref{tab:Notation}. These parameters specify a bordism $n$-category, realized as a multisimplicial space:
\begin{equation*}
	\Bord_{d;n}^{(X, \xi)}(M):  (\bDelta^\op)^n \to \Top.
\end{equation*}
The goal of this section is to identify the geometric realization of this multisimplicial space. We will be able to do this under the assumption that the ambient manifold $M$ is \emph{tame} in the sense defined below. This rules out certain pathological $M$ like the surface of infinite genus. 

\begin{definition}\label{def:tame_manifold}
	A manifold $M$ is \emph{tame} if there exists a compact subspace $K \subseteq M$ and a continuous  1-parameter family of embeddings $\varphi_t: M \to M$ starting with the identity and ending with an embedding whose image is contained within $K$.
\end{definition}

\noindent We will prove:

\begin{theorem} \label{thm:first_realization_thm}
	If $M$ is \emph{tame}, then for each $i \leq n$, there is a natural levelwise weak homotopy equivalence of $(n-i)$-fold simplicial spaces:
		\begin{equation*}
		B^i \Bord_{d;n}^{(X, \xi)}(M) \stackrel{\simeq}{\to} \Bord_{d;n-i}^{(X, \xi)}(M \times \R^i)
	\end{equation*}
	where the classifying space functor is applied to the final $i$-many simplicial directions $\{ n-i +1, n-i +2, \dots, n\}$. 
\end{theorem}

\noindent Of course Theorem~\ref{thm:first_realization_thm} follows immediately from the special case $i=1$, and we will focus on that case. When $n =1$ we obtain the topological category considered by Randal-Williams~\cite{MR2764873}, which is a generalization of the categories considered by Galatius-Madsen-Tillmann-Weiss~\cite{MR2506750} and Ayala~\cite{Ayala:2008aa}. In this version the manifolds are embedded into $M$ instead of $\R^\infty$. 

\begin{corollary}\label{cor:first_realization_thm}
	If $n \geq 1$, the classifying space $B^n \Bord_{d;n}^{(X, \xi)}(D^p)$ is weakly equivalent  to $\Omega^p Th(\xi^* \gamma^\perp_d)$ as an $E_p$-algebra.  
\end{corollary}

\begin{proof}
	By Theorem~\ref{thm:first_realization_thm}  $B^n \Bord_{d;n}^{(X, \xi)}(D^p)$ is weakly equivalent to $ \Bord_{d;0}^{(X, \xi)}(D^p \times \R^n) = \psi_d^{(X, \xi)}(D^p \times \R^n)$ as ab $E_p$-algebra and the later is weakly equivalent to  $\Omega^p Th(\xi^* \gamma^\perp_d)$ as an $E_p$-algebra by Theorem~\ref{thm:E_pSpaces}.
\end{proof}

\subsection{Overview of proof of Theorem~\ref{thm:first_realization_thm} and Variations on the Bordism $n$-Category}

We will prove Theorem~\ref{thm:first_realization_thm} in the special case $i=1$. The general case follows from this by induction. Thus we wish to relate $\Bord_{d;n}^{(X, \xi)}(M)$ and $\Bord_{d;n-1}^{(X, \xi)}(M \times \R)$, and to do so we need a map comparing them. The latter object is only a $(n-1)$-fold multisimplicial space, but we can regard it as a $n$-fold multisimplicial space which is constant in the final simplicial direction. Regarded in this way there is a natural map: 
\begin{equation*}
	\Bord_{d;n}^{(X, \xi)}(M) \to \Bord_{d;n-1}^{(X, \xi)}(M \times \R)
\end{equation*}
of $n$-fold multisimplicial spaces. A point in the left-hand space consists of a tuple $( (\Bt^i)_{i=1}^n, (W, \theta))$, where $W \subseteq M \times \R^{\{1, \dots, n\}}$ is an embedded manifold. Similarly a point in the right-hand space consists of a smaller tuple $( (\Bt^i)_{i=1}^{n-1}, (W, \theta))$, where $W \subseteq M \times \R^{\{n\}} \times \R^{\{1, \dots, n-1\}}$ is an embedded manifold. Using the obvious identifications of these ambient spaces (into which the $W$ are embedded) the above map is given simply by forgetting the final tuple of coordinates $(\Bt^n)$. 

Upon taking classifying spaces (which will always mean the fat geometric realization in the final ($n^\text{th}$) simplicial coordinate) we get maps
\begin{equation*}
	B(\Bord_{d;n}^{(X, \xi)}(M) \to B(\Bord_{d;n-1}^{(X, \xi)}(M\times \R))  \stackrel{\simeq}{\to} \Bord_{d;n-1}^{(X, \xi)}(M\times \R)
\end{equation*}
 and we will call the composite $u: B(\Bord_{d;n}^{(X, \xi)}(M)) \to \Bord_{d;n-1}^{(X, \xi)}(M \times \R)$. Theorem~\ref{thm:first_realization_thm} is established once we can show that $u$ is a weak equivalence of $(n-1)$-fold multisimplicial spaces.

Although we will show that this map is a weak equivalence, our proof will not be direct. Instead we will introduce six additional variations of the bordism higher category which arrange into the large commutative diagram in Figure~\ref{fig:daigramofBordcats2}. We will describe these variations momentarily.
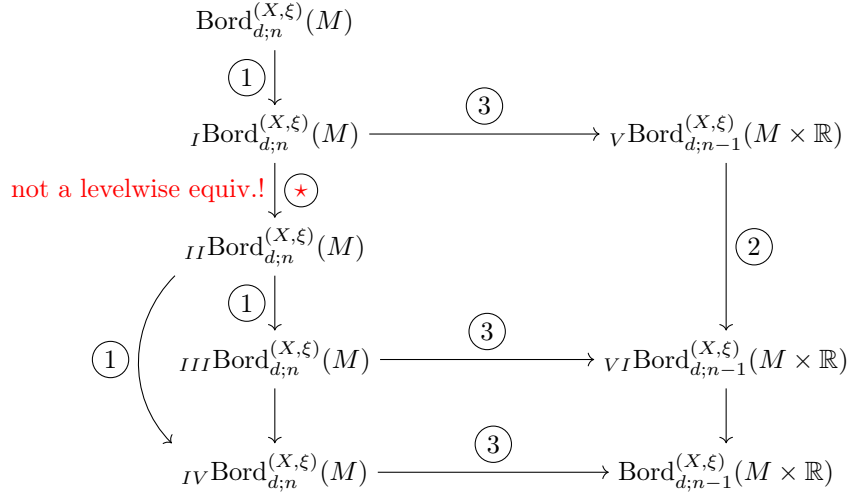
\begin{figure}[thpb]
	\begin{center}
	\begin{center}
	\begin{tikzpicture}
			\node (LT) at (0, 6) {$\Bord_{d;n}^{(X, \xi)}(M)$};
			\node (LM1) at (0, 4.5) {${}_{I}\Bord_{d;n}^{(X, \xi)}(M)$};
			\node (LM2) at (0, 3) {${}_{II}\Bord_{d;n}^{(X, \xi)}(M)$};
			\node (LM3) at (0, 1.5) {${}_{III}\Bord_{d;n}^{(X, \xi)}(M)$};
			\node (LB) at (0, 0) {${}_{IV}\Bord_{d;n}^{(X, \xi)}(M)$};
			
			\node (RM1) at (6, 4.5) {${}_{V}\Bord_{d;n-1}^{(X, \xi)}(M \times \R)$};
			\node (RM3) at (6, 1.5) {${}_{VI}\Bord_{d;n-1}^{(X, \xi)}(M \times \R)$};	
			\node (RB) at (6, 0) {$\Bord_{d;n-1}^{(X, \xi)}(M \times \R)$};
			
			\draw [->] (LT) -- node [left] {$\circled{1}$} (LM1);
			\draw [->] (LM1) -- node [left] {\color{red}{not a levelwise equiv.!}} node [right] {$\circled{\color{red}{{$\star$}}}$} (LM2);
			\draw [->] (LM2) -- node [left] {$\circled{1}$} (LM3);
			\draw [->] (LM3) -- node [left] {$ $} (LB);
			\draw [->] (LM2.south west) to [out = 225, in = 135] node [left] {$\circled{1}$} (LB.north west);
			
			\draw [->] (LM1) -- node [above] {$\circled{3}$} (RM1);
			\draw [->] (LM3) -- node [above] {$\circled{3}$} (RM3);
			\draw [->] (LB) -- node [above] {$\circled{3}$} (RB);

			\draw [->] (RM1) -- node [right] {$\circled{2}$} (RM3);
			\draw [->] (RM3) -- node [right] {$ $} (RB);
	\end{tikzpicture}
	\end{center}
	\end{center}
	\caption{Many variations of the bordism higher category.}
	\label{fig:daigramofBordcats2}
\end{figure}

To obtain our desired result we will then show that each of these maps becomes a weak equivalence upon passing to geometric realizations. The most difficult map with respect to this measure is the one labeled by \circled{\color{red}{$\star$}}, which is not a levelwise weak equivalence. 

 Theorem~\ref{thm:first_realization_thm} will be proven in three stages. First we will show that all the arrows labeled by \circled{1} are levelwise weak equivalences of multisimplicial spaces.
 This follows easily by applying deformations to the spaces of bordisms of the sort considered in the next Section~\ref{sec:examples_of_deformations}. A slightly different set of deformations will similarly allow us to also show that the arrow labeled by  \circled{2} is  a levelwise weak equivalence.
The final step, to show that each of the horizontal arrows labeled by \circled{3} are weak equivalences after passing to geometric realization, is more difficult and requires a new argument. This argument is based on the observation that each of the sources of these maps are (levelwise) the nerves of certain topological posets. The desired result is proven in Lemma~\ref{lem:3_equiv} below. This is the key place where the tameness of $M$ appears. 

Now let us describe the variants that appear above in Figure~\ref{fig:daigramofBordcats2}. As usual we have fixed natural numbers $n,d$, an ambient manifold $M$, and an equivariant fibration $(X, \xi)$, as in Table~\ref{tab:Notation}. We will consider five functors
\begin{align*}
	\Bord_{d;n}^{(X, \xi)}(M): (\bDelta^\op)^n &\to \Top \\
	{}_{I}\Bord_{d;n}^{(X, \xi)}(M):  (\bDelta^\op)^n &\to \Top \\
	{}_{II}\Bord_{d;n}^{(X, \xi)}(M): (\bDelta^\op)^n &\to \Top \\
	{}_{II}\Bord_{d;n}^{(X, \xi)}(M):  (\bDelta^\op)^n &\to \Top \\
	{}_{IV}\Bord_{d;n}^{(X, \xi)}(M):  (\bDelta^\op)^n &\to \Top 
\end{align*}
and three functors
\begin{align*}
	{}_{V}\Bord_{d;n-1}^{(X, \xi)}(M \times \R):  (\bDelta^\op)^{n-1} &\to \Top \\
	{}_{VI}\Bord_{d;n-1}^{(X, \xi)}(M \times \R):  (\bDelta^\op)^{n-1} &\to \Top \\
	\Bord_{d;n-1}^{(X, \xi)}(M \times \R):  (\bDelta^\op)^{n-1} &\to \Top 
\end{align*}
The first and last of these, $\Bord_{d;n}^{(X, \xi)}(M)$ and $\Bord_{d;n-1}^{(X, \xi)}(M \times \R)$, are described in Definition~\ref{def:Bordismncat}. We recall this definition now for the convenience of the reader. 

The multisimplicial space $\Bord_{d;n}^{(X, \xi)}(M)$ assigns to $[ \Bm] \in  (\bDelta^\op)^n$ the space consisting of tuples
	$((\Bt^i)_{i=1}^n, (W, \theta))$ where $\Bt^i \in \R^{[m_i]}$  for each $1 \leq i \leq n$ and $(W, \theta) \in \psi^{(X, \xi)}_d( M  \times \R^n)$ is a submanifold with $(X,\xi)$-structure $\theta$. These tuples are required to satisfy the Globularity conditions of Definition~\ref{def:Bordismncat}:
		\begin{itemize}
			\item \textbf{Globular.}  For all $ 1 \leq i \leq n$, and $0 \leq j \leq m_i$, $W$ is cylindrical near 
				\mbox{$\{t^i_j\} \times \R^{\{i+1, \dots, n\}} \subseteq \R^{\{i, i+1, \dots, n\}}$}; 
		\end{itemize}
See Defintion~\ref{def:cylindrical} for the meaning of \emph{cylindrical}.

The four additional functors 	${}_{I}\Bord_{d;n}^{(X, \xi)}(M)$, ${}_{II}\Bord_{d;n}^{(X, \xi)}(M)$, ${}_{III}\Bord_{d;n}^{(X, \xi)}(M)$, and ${}_{IV}\Bord_{d;n}^{(X, \xi)}(M)$ are defined in precisely the same way, except that the globularity condition is modified in each case:
\begin{itemize}
	\item \textbf{Globular} (equivalent to the original condition for $\Bord_{d;n}^{(X, \xi)}(M)$)
		\begin{itemize}
			\item  For all $ 1 \leq i < n$, and $0 \leq j \leq m_i$, $W$ is cylindrical near \\
				\mbox{$\{t^i_j\} \times \R^{\{i+1, \dots, n\}} \subseteq \R^{\{i, i+1, \dots, n\}}$}; 
			\item For all $0 \leq j \leq m_n$, $W$ is cylindrical near $\{t^n_j\} \subseteq \R^{\{n\}}$
		\end{itemize}
		\item \textbf{Globular}-$I$
			\begin{itemize}
				\item  For all $ 1 \leq i < n$, and $0 \leq j \leq m_i$, $W$ is cylindrical near \\
					\mbox{$\{t^i_j\} \times \R^{\{i+1, \dots, n\}} \subseteq \R^{\{i, i+1, \dots, n\}}$};
				\item For all $0 \leq j \leq m_n$, $t^n_j$ is a regular value of the projection $W \to \R^{\{n\}}$.
			\end{itemize}
			\item \textbf{Globular}-$II$
				\begin{itemize}
					\item  For all $ 1 \leq i < n$, and $0 \leq j \leq m_i$, $W$ is cylindrical near \\
						\mbox{$\{t^i_j\} \times \R^{\{i+1, \dots, n-1\}} \subseteq \R^{\{i, i+1, \dots, n-1\}}$};
					\item For all $ 1 \leq i < n$, and $0 \leq j \leq m_i$, $W$ is cylindrical near \\
						\mbox{$\{t^i_j\} \times \R^{\{i+1, \dots, n-1\}} \times (\infty, t_0^n] \subseteq \R^{\{i, i+1, \dots, n\}}$};
					\item For all $0 \leq j \leq m_n$, $t^n_j$ is a regular value of the projection $W \to \R^{\{n\}}$.
				\end{itemize}
			\item \textbf{Globular}-$III$
				\begin{itemize}
					\item  For all $ 1 \leq i < n$, and $0 \leq j \leq m_i$, $W$ is cylindrical near \\
						\mbox{$\{t^i_j\} \times \R^{\{i+1, \dots, n-1\}} \subseteq \R^{\{i, i+1, \dots, n-1\}}$};
					\item For all $ 1 \leq i < n$, and $0 \leq j \leq m_i$, there exists an $L \in \R$ such that $W$ is cylindrical near \\
						\mbox{$\{t^i_j\} \times \R^{\{i+1, \dots, n-1\}} \times (\infty, L] \subseteq \R^{\{i, i+1, \dots, n\}}$};
					\item For all $0 \leq j \leq m_n$, $t^n_j$ is a regular value of the projection $W \to \R^{\{n\}}$.
				\end{itemize}
			\item \textbf{Globular}-$IV$
				\begin{itemize}
					\item  For all $ 1 \leq i < n$, and $0 \leq j \leq m_i$, $W$ is cylindrical near \\
						\mbox{$\{t^i_j\} \times \R^{\{i+1, \dots, n-1\}} \subseteq \R^{\{i, i+1, \dots, n-1\}}$};
					\item For all $0 \leq j \leq m_n$, $t^n_j$ is a regular value of the projection $W \to \R^{\{n\}}$.
				\end{itemize}
\end{itemize}

	The multisimplicial spaces ${}_{V}\Bord_{d;n-1}^{(X, \xi)}(M\times \R)$, ${}_{VI}\Bord_{d;n-1}^{(X, \xi)}(M\times \R)$, and $\Bord_{d;n-1}^{(X, \xi)}(M\times \R)$ are also quite similar. They assigns spaces to each  $[ \Bm] \in  (\bDelta^\op)^{n-1}$ and these spaces consist of tuples $( (\Bt^i)_{i=1}^{n-1}, (W, \theta))$ where $\Bt^i \in \R^{[m_i]}$  for each $1 \leq i \leq n-1$ and $(W, \theta) \in \psi^{(X, \xi)}_d( M \times \R  \times \R^{n-1})$ is a submanifold with $(X,\xi)$-structure. These can be compared to the previous spaces using the natural identification $M \times \R  \times \R^{n-1} \cong M  \times \R^{n}$, which identifies the additional $\R$ factor with the additional $\R^{\{n\}}$ coordinate. 
	
These spaces are required to satisfy the the following globularity conditions:
\begin{itemize}
	\item \textbf{Globular}-$V$
		\begin{itemize}
			\item  For all $ 1 \leq i < n$, and $0 \leq j \leq m_i$, $W$ is cylindrical near \\
				\mbox{$\{t^i_j\} \times \R^{\{i+1, \dots, n\}} \subseteq \R^{\{i, i+1, \dots, n\}}$};
		\end{itemize}
	\item \textbf{Globular}-$VI$
		\begin{itemize}
			\item  For all $ 1 \leq i < n$, and $0 \leq j \leq m_i$, $W$ is cylindrical near \\
				\mbox{$\{t^i_j\} \times \R^{\{i+1, \dots, n-1\}} \subseteq \R^{\{i, i+1, \dots, n-1\}}$};
			\item For all $ 1 \leq i < n$, and $0 \leq j \leq m_i$, there exists an $L \in \R$ such that $W$ is cylindrical near \\
				\mbox{$\{t^i_j\} \times \R^{\{i+1, \dots, n-1\}} \times (\infty, L] \subseteq \R^{\{i, i+1, \dots, n\}}$};
		\end{itemize}	
	\item \textbf{Globular} (the original condition for $\Bord_{d;n-1}^{(X, \xi)}(M\times \R)$)
		\begin{itemize}
			\item  For all $ 1 \leq i < n$, and $0 \leq j \leq m_i$, $W$ is cylindrical near \\
				\mbox{$\{t^i_j\} \times \R^{\{i+1, \dots, n-1\}} \subseteq \R^{\{i, i+1, \dots, n-1\}}$};
		\end{itemize}	
\end{itemize}
These conditions are identical to Globular-$I$, Globular-$III$, and Globular-$IV$, respectively, except that the condition that $t_j^n$ be a regular value of the projection from $W$ to $\R^{\{n\}}$ has been dropped. 	

\subsection{Examples of continuous deformations} \label{sec:examples_of_deformations}

In what is coming, we will want to manipulate these spaces of embedded manifolds in various ways. Using the yoga of plots this will be very easy. We have already seen in Section~\ref{sec:space_of_man} that $\psi_d^{(X, \xi)}(-)$ is contravariant for open embeddings and covariant for closed embeddings. In some cases we have additional functoriality. For example:

\begin{lemma}\label{lem:conditional_inverse}
	Suppose that $c$ is a condition for embedded manifolds in $N$ defining the subspace $\psi_d(N; c)$. Let $i:U \times M \to N$ be smooth and set $i_u = i(u, -): M \to N$. Suppose that $i_u \pitchfork W$ for every $W \in \psi_d(N; c)$ and $u \in U$. Then 
	\begin{align*}
		i^*: U \times \psi_d(N; c) &\to \psi_d(M) \\
		(u,W) & \mapsto i_u^{-1}(W)
	\end{align*}
	is continuous. \qed
\end{lemma}

This latter kind of deformation can be used to `straighten' our embedded manifolds, as we will now show. Let $a \in \R$ and consider the following two spaces:
\begin{align*}
		\psi_d(\R \times M; (a, \pitchfork)) &= \{ W \in \psi_d(\R \times M) \; | \; a \text{ is a regular value of the projection } W \to \R \} \\
		\psi_d(\R \times M; (a, \perp)) & = \{ W \in \psi_d(\R \times M) \; | \; W \text{ is cylindrical near } a \}.
\end{align*}
The condition that $a$ is a regular value of the projection $W \to \R$ is the same as requiring the transversality condition $W \pitchfork \{a\} \times M$. This is clearly satisfied for manifolds which are cylindrical near $\{a\}$ and so $\psi_d(\R \times M; (a, \perp, \epsilon)) \subseteq \psi_d(\R \times M; (a, \pitchfork))$ includes as a subset.

\begin{lemma}\label{lem:straightening_transverse}
	The inclusion map $i: \psi_d(\R \times M; (a, \perp)) \to \psi_d(\R \times M; (a, \pitchfork))$ is a homotopy equivalence. 
\end{lemma}
 
\begin{proof}
	We will construct the homotopy equivalence using the previous lemma. Our treatment is based on \cite[Lem.~3.4]{MR2653727}.
	We will first construct our potential inverse homotopy equivalence. Choose once and for all a smooth function $\lambda: \R \to \R$ with $\lambda(s) = 0$ for $|s| \leq 1$, $\lambda(s) = s$ for $|s| \geq 2$, and $\lambda'(s) > 0$ for $|s| > 1$. Also fix $\epsilon > 0$, and set
	\begin{align*}
		q_\epsilon: \R \times M &\to \R \times M \\
			(s, m) & \mapsto \left( \epsilon \lambda(\frac{s -a}{\epsilon}) + a , m   \right).
	\end{align*}
	This is a smooth map, and the requirement that $a$ is a regular value for $W \in \psi_d(\R \times M; (a, \pitchfork))$ ensures that $W \pitchfork q_\epsilon$. Thus 
	\begin{align*}
q_\epsilon^*: \psi_d(\R \times M; (a, \pitchfork)) &\to \psi_d(\R \times M; (a, \perp)) \\
			W & \mapsto q_\epsilon^{-1}(W).
	\end{align*}
	is continuous. Note that the image of $q_\epsilon$ is contained in those $W$ which are cylindrical over the interval $(a - \epsilon, a + \epsilon)$.
  
Next we will show that the composites $i \circ q_\epsilon^*$ and $q_\epsilon^* \circ i$ are homotopic to identity maps. For $t \in [0,1]$ we set
	\begin{align*}
		\varphi_t: \R \times M &\to \R \times M \\
			(s, m) & \mapsto \left( (1-t)s + t \epsilon \lambda(\frac{s -a}{\epsilon}) , m   \right).
	\end{align*}
	This gives a smooth map $\varphi: [0,1] \times \R \times M \to \R \times M$. Again the fact that $a$ is a regular value for each $W \in \psi_d(\R \times M; (a, \pitchfork))$ ensures that the conditions of Lemma~\ref{lem:conditional_inverse} are met (in fact for $t < 1$, the map $\varphi_t$ is a diffeomorphism). Thus we have a continuous homotopy
	\begin{equation*}
		\varphi^*: I \times \psi_d(\R \times M; (a, \pitchfork)) \to \psi_d(\R \times M; (a, \pitchfork)). 
	\end{equation*}
We have $\varphi_0 = id_{\R \times M}$ and $\varphi_1 = q_\epsilon$, and so $\varphi^*$ gives the desired homotopy between $i \circ q_\epsilon^*$ and the identity on $\psi_d(\R \times M; (a, \pitchfork))$.

Finally we note that $\varphi_t^*$ preserves the property of being cyclindrical near $\{a\}$, and hence $\varphi^*$ also restricts to give a homotopy between $q_\epsilon \circ i$ and the identity on $\psi_d(\R \times M; (a, \perp))$.	
\end{proof}

 \noindent Similar deformations will occur later. 

In some situations we can also pull-back and deform by functions which are not smooth. We will now give an example of this phenomenon. Let $a \in \R$ and consider the following two spaces:
\begin{align*}
	\psi_d(\R \times M; (a, \perp)) & = \{ W \in \psi_d(\R \times M) \; | \; W \text{ is cylindrical near } a \} \\
	\psi_d(\R \times M; ([a, \infty), \perp)) & = \{ W \in \psi_d(\R \times M) \; | \; W \text{ is cylindrical near } [a, \infty) \}. 
\end{align*}
The first we have already considered. The later space consists of those $W \in \psi_d(\R \times M)$ which are not only cylindrical near $a \in \R$ but also over the inverval $(a, +\infty) \subseteq \R$. 

Let $\alpha: \R \to \R$ be the following function:
\begin{equation*}
	\alpha(s) = \begin{cases}
		a & s \geq a \\
		s & s \leq a
	\end{cases}
\end{equation*}
and let $p= (\alpha, id_M): \R \times M \to \R \times M$. Then $\alpha$ is continuous but not smooth. Nevertheless it induces a continuous map
 	\begin{align*}
		p^*: \psi_d(\R \times M; (a, \perp)) &\to \psi_d(\R \times M; ([a, \infty), \perp)) \\
 			W & \mapsto p^{-1}(W).
 	\end{align*}
This is because for each smooth plot $\gamma: X \to \psi_d(\R \times M; (a, \perp))$, which consists of manifolds which are cylindrical near $a$, the result of applying $p^*$ is again a smooth plot. Thus, even though $p$ itself is not smooth, it induces a continuous map between these spaces of smoothly embedded manifolds. 

\begin{lemma}\label{lem:extend_cylindrical}
	The map $p^*$ extends to a strong deformation retraction of $\psi_d(\R \times M; (a, \perp))$ onto $\psi_d(\R \times M; ([a, \infty), \perp))$. 
\end{lemma}

\begin{proof}
	The desired homotopy is given by replacing $\alpha: \R \to \R$ in the definition of $p$ by the family of maps $\alpha_t$ for $t \in [0,1]$:
	\begin{equation*}
		\alpha_t(s) = \begin{cases}
			(1-t) s + ta & s \geq a \\
			s & s \leq a
		\end{cases}
	\end{equation*}
	We have $\alpha_0 = id$ and $\alpha_1 = \alpha$, our original map. 	This induces a continuous homotopy for the same reasons that $p$ induces a continuous map, and direct inspection shows that it restricts to the constant homotopy on $\psi_d(\R \times M; ([a, \infty), \perp))$.
\end{proof}

\subsection{Showing that the arrows (1) and (2) are levelwise equivalences}
	
\begin{lemma}
	The three arrows labeled with \circled{1} in Figure~\ref{fig:daigramofBordcats2}:
	\begin{align}
		\label{eqn:1}
		\Bord_{d;n}^{(X, \xi)}(M) &\to {}_{I}\Bord_{d;n}^{(X, \xi)}(M) \\
		\label{eqn:2}
		{}_{II}\Bord_{d;n}^{(X, \xi)}(M) & \to {}_{III}\Bord_{d;n}^{(X, \xi)}(M)  \\
		\label{eqn:3}
		{}_{II}\Bord_{d;n}^{(X, \xi)}(M) & \to {}_{IV}\Bord_{d;n}^{(X, \xi)}(M) 
	\end{align}
	are levelwise weak equivalences. 
\end{lemma}

\begin{proof}
Fix  $[\Bm] \in (\bDelta^\op)^n$. 
	Consider first the map (\ref{eqn:1}). The difference between these two spaces of manifolds is that in $\Bord_{d;n}^{(X, \xi)}(M)$ the embedded manifold is required to be cylindrical near $\{t^n_j\} \subseteq \R^{\{n\}}$ while in  ${}_{I}\Bord_{d;n}^{(X, \xi)}(M)$  the manifold is only required to be transverse to the hyperplane $M \times \R^{n-1} \times \{t^n_j\}$. We need to `straighten' the manifold near each $t^n_j$ hyperplane. 
	
	This is exactly the situation that we considered in Lemma~\ref{lem:straightening_transverse} and (\ref{eqn:1}) can be shown to be an equivalence by precisely the same argument, applied at each $t^n_j$. The only care that must be taken is that, in the notation of the proof of Lemma~\ref{lem:straightening_transverse}, $\epsilon$ is sufficiently small (see the proof of Lemma~\ref{lem:straightening_transverse}). Taking 
	\begin{equation*}
		\epsilon < \frac{1}{3} \min_j \{ |t^n_{j-1} - t^n_j| \}
	\end{equation*}
is sufficient. 
	
For the other two arrows (\ref{eqn:2}) and (\ref{eqn:3}) the difference between the bordism categories is that the embedded manifold in ${}_{II}\Bord_{d;n}^{(X, \xi)}(M)$ is required to be cylindrical near $\{t^i_j\} \times \R^{\{i+1, \dots, n-1\}} \times (-\infty, t_0^n] \subseteq \R^{\{i, \dots, n\}}$ while in ${}_{III}\Bord_{d;n}^{(X, \xi)}(M)$ it is only required to be cylindrical near $\{t^i_j\} \times \R^{\{i+1, \dots, n-1\}} \times (-\infty, L]$ for some $L$ and in ${}_{IV}\Bord_{d;n}^{(X, \xi)}(M)$ there is no corresponding cylindricality condition. 

We will procede in two stages. First we use the same method as above to straighten the embedded manifold near $\{t_0^n\}$, again applying the same argument as in Lemma~\ref{lem:straightening_transverse}. The result is that we may assume that the embedded manifolds are cylindrical near $\{t_0^n\} \subseteq \R^{\{n\}}$. Since this deformation only occurs in the $\R^{\{n\}}$ coordinate it does not change the cylindricality near $\{t^i_j\} \times \R^{\{i+1, \dots, n-1\}}$. As a consequence we have that $W$ is now cylindrical near $\{t^i_j\} \times \R^{\{i+1, \dots, n-1\}} \times \{t_0^n\}$. 

For the next stage, we use an argument that is nearly identical to the proof of Lemma~\ref{lem:extend_cylindrical}. Effectively we will deform the embedded manifold $W$ to satisfy the conditions for ${}_{II}\Bord_{d;n}^{(X, \xi)}(M)$  by `sliding' the bordism to infinity below $t_0^n$ in the $\R^{\{n\}}$ coordinate. This is in fact a deformation retraction onto ${}_{II}\Bord_{d;n}^{(X, \xi)}(M)$. Specifically we will precompose the $\R^{\{n\}}$ coordinate by the family of maps:
\begin{equation*}
	\alpha_t(s) = \begin{cases}
			s & s \geq t_0^n \\
			(1-t)\cdot s + t \cdot t_0^n & s \leq t_0^n.
	\end{cases}
\end{equation*}
As in Lemma~\ref{lem:extend_cylindrical} this yields the desired deformation retraction. 
\end{proof}

\begin{lemma}
	The arrow labeled with \circled{2} in Figure~\ref{fig:daigramofBordcats2}:
	\begin{equation} \label{eqn:4}
		{}_{V}\Bord_{d;n-1}^{(X, \xi)}(M \times \R) \to {}_{VI}\Bord_{d;n-1}^{(X, \xi)}(M \times \R)
	\end{equation}
	is a levelwise weak equivalence. 
\end{lemma}

\begin{proof}
	The difference between ${}_{V}\Bord_{d;n-1}^{(X, \xi)}(M \times \R)$ and ${}_{VI}\Bord_{d;n-1}^{(X, \xi)}(M \times \R)$ is in the cylindricality condition satisfied by the embedded manifolds. In the former the manifold is cylindrical near \mbox{$\{t^i_j\} \times \R^{\{i+1, \dots, n\}} \subseteq \R^{\{i, i+1, \dots, n\}}$} while in the latter it is only required to be cylindrical near \mbox{$\{t^i_j\} \times \R^{\{i+1, \dots, n-1\}} \subseteq \R^{\{i, i+1, \dots, n-1\}}$} and near
	\mbox{$\{t^i_j\} \times \R^{\{i+1, \dots, n-1\}} \times (\infty, L] \subseteq \R^{\{i, i+1, \dots, n\}}$} for some $L$. 
	
We will show that this map of multisimplicial spaces is a levelwise homotopy equivalence by exhibiting it as part of a specific deformation retraction. In words the idea is to slide the embedded bordism in the additional $\R$ direction to extend the cylindricality condition from near $\{t^i_j\} \times \R^{\{i+1, \dots, n-1\}} \times (\infty, L]$ to one near all of $\{t^i_j\} \times \R^{\{i+1, \dots, n-1\}} \times \R$. In the course of this deformation some of the manifold $W$ may `disappear at $\infty$'. 

Mathematically this will be accomplished by precomposing our manifold by a family of self-embeddings of $M \times \R  \times \R^{n-1}$. One complication is that $L$ is not fixed,  and thus we must choose a family of embeddings which will be compatible with all possible $L$. 

Thus we fix $i$ and proceed as follows. First we fix a smooth bump function $\rho: \R \to [0,1]$ satisfying:
	\begin{equation*}
		\rho(s) = \begin{cases}
			0 & s \leq 0 \\
			1 & s \geq 1
		\end{cases}
	\end{equation*}
In addition we will need a family of embeddings from $\R$ into $\R$ parametrized by a parameter $a$. For concreteness we will use:
	\begin{equation*}
		f_a(s) = \frac{s-a - \sqrt{(s-a)^2 + 4}}{2} + a
	\end{equation*}
The important features of this family of functions are that
\begin{enumerate}
	\item for each $a$ it is a diffeomorphism onto its image $(-\infty, a) \subseteq \R$,
	\item for $s << a$, $f_a(s)$ is asymptotic to the identity function, and
	\item for $s >> a$, $f_a$ is 	asymptotic to the constant function with value $a$.
\end{enumerate}
In particular the limit of $f_a$ as $a \to + \infty$ exists and is the identity function. 

Using this we can now construct a family of self-embeddings $\varphi_t$ of $M \times \R \times \R^{n-1}$ parametrized by $t \in [0,1]$. In fact this family only depends on and changes the additional $\R$-coordinate and the $i^\text{th}$ coordinate of $\R^{n-1}$; it is the identity on the remaining variables. On $\R \times \R^{\{i\}}$ it is given as follows (for $t \in [0,1)$):
	\begin{align*}
		\alpha_t: \R \times \R^{\{i\}} & \to \R \times  \R^{\{i\}} \\
		(s, y) & \mapsto (\sum_j \rho(\frac{1}{1-t} | y - t^i_j|) s + (1 - \rho(\frac{1}{1-t} | y - t^i_j|)) f_{\cot \pi t}(s), y  )
	\end{align*}
When $t=0$ we have that $\alpha_0 = id$ is the identity map. For positive $t$ $\alpha_t$  leaves the $\R^{\{i\}}$ coordinate, $y$, unchanged and applies a diffeomorphism to the additional $\R$ direction. This diffeomorphism depends on both the time variable $t$ and on the $\R^{\{i\}}$ coordinate $y$. When $y$ is sufficiently far away from the $t^i_j$ values, then the diffeomorphism is simply the identity morphism of $\R$. When $y$ is near to $t^i_j$, then the diffeomorphism is essentially the function $f_a$ with $t=0$ corresponding to $a = +\infty$ and $t=1$ corresponding to $a = -\infty$. Moreover as $t$ increases the condition of being `near to $t^i_j$' becomes increasingly stringent, so that the diffeomorphism of $\R$ is the identity for more and more values of $y$. 
	
These conditions plus the fact that $W$ is cylindrical near \mbox{$\{t^i_j\} \times \R^{\{i+1, \dots, n-1\}} \subseteq \R^{\{i, i+1, \dots, n-1\}}$} and near
	\mbox{$\{t^i_j\} \times \R^{\{i+1, \dots, n-1\}} \times (\infty, L] \subseteq \R^{\{i, i+1, \dots, n\}}$} for some $L$ ensure that for any fixed $W$ eventually there exists a $t_W < 1$ after which $W$ remains fixed $W = \varphi_t^{-1}(W)$ for $t \geq t_W$. Thus this deformation extends to a deformation well-defined even at $t=1$. Since this deformation also preserves the subspace ${}_{V}\Bord_{d;n-1}^{(X, \xi)}(M \times \R)$ it gives the desired deformation retraction. 
\end{proof}

\subsection{Showing that the arrows labeled (3) are equivalences after geometric realization} \label{subsec:arrow3}	

To complete the proof of Theorem~\ref{thm:first_realization_thm} we must show that the three arrows in Figure~\ref{fig:daigramofBordcats2} labeled with \circled{3} are weak homotopy equivalences after geometric realization, which is the statement of the following lemma.  

\begin{lemma}\label{lem:3_equiv}
	The  three maps induced by Figure~\ref{fig:daigramofBordcats2}:
	\begin{align} 
		\label{eqn:5}
		B({}_{I}\Bord_{d;n}^{(X, \xi)}(M)) \to {}_{V}\Bord_{d; n-1}^{(X, \xi)}(M\times \R) \\
		\label{eqn:6}
		B({}_{III}\Bord_{d;n}^{(X, \xi)}(M)) \to {}_{VI}\Bord_{d; n-1}^{(X, \xi)}(M\times \R) \\
		\label{eqn:7}
		B({}_{IV}\Bord_{d;n}^{(X, \xi)}(M)) \to \Bord_{d; n-1}^{(X, \xi)}(M\times \R)
	\end{align}
	are levelwise weak equivalences.
\end{lemma}

\noindent Recall that here the (fat) geometric realization has been preformed on the extra simplicial direction present on the source of the maps in  Figure~\ref{fig:daigramofBordcats2}. 

The argument in each case (\ref{eqn:5}), (\ref{eqn:6}), and (\ref{eqn:7}) is the same and for simplicity we will focus on the final case (\ref{eqn:7}). A key observation is that if we fix $[\Bm] \in (\bDelta^\op)^{n-1}$, and consider the induced simplicial space:
\begin{equation*}
	{}_{IV}\Bord_{d;n}^{(X, \xi)}(M)_{[\Bm], \bullet}
\end{equation*}
This simplicial space is the nerve of a topological poset. Specifically it is the nerve of a topological poset which is a subposet of 
\begin{equation*}
	( \R, \leq) \times \Bord_{d; n-1}^{(X, \xi)}(M\times \R)_{[\Bm]}
\end{equation*}
where the partial order is induced from the standard order on $\R$. The topological poset ${}_{IV}\Bord_{d;n}^{(X, \xi)}(M)_{[\Bm]}$ consists of those pairs $(\lambda,  ( (\Bt^i)_{i=1}^{n-1}, (W, \theta)))$  which satisfy the condition that  $\lambda$
is a regular value of the projection of $W$ onto the additional $\R$-direction. 

The strategy underlying the proof of Lemma~\ref{lem:3_equiv} (the final step in the proof Theorem~\ref{thm:first_realization_thm}) relies on exploiting the description in terms of topological posets. We begin with a easy lemma.

\subsubsection{A lemma about topological posets}
	

	Let $P_0$ be a topological space and let $T$ be a totally ordered set with the order topology. In our  examples we have $T = \R$ with the standard topology. We may regard $P_0$ as a topological poset in which no elements are comparable. In this case the nerve of $P_0$ is a constant simplicial space. Let $(P, \leq) \subseteq P_0 \times T$ be a sub-topological poset, which means it is a subset endowed with the induced pre-order and the subspace topology. Let $u: P \to P_0$ be the projection, which we will regard as a map from the nerve of $P$ to the constant simplicial space $P_0$. The fat geometric realization (a.k.a. classifying space) functor will be denoted $||-||$. 

	\begin{lemma}\label{lem:PosetRealization}
		Let $u:(P, \leq) \to P_0$ be as in the situation above. Assume that $u$ admits a section $s: P_0 \to (P, \leq)$
		 as a map of topological posets. Then after fat geometric realization
		 \begin{equation*}
		 	||u||: ||P|| \to ||P_0|| \cong P_0 \times ||pt|| \simeq P_0
		 \end{equation*}
		 the map $||u||$ is a weak homotopy equivalence. 
	\end{lemma}

	\begin{proof}
		The composite $||u||  \circ ||s|| = || us || = id_{||P_0||}$, and thus it suffices to show that that $||s|| \circ ||u|| = ||su||$ is homotopic to the identity map on $||P||$. 
		The fat geometric realization of the nerve sends functors between topological posets to maps and natural transformations to homotopies between these maps. 
		We will construct a zig-zag of natural transformations between the endofunctors $su$ and the identity on $(P, \leq)$.
	
	First define closed subsets of $P$ as follows:
		\begin{align*}
			P_{\leq} &= \{ w \in P \; | \; w \leq su(w) \} \\
			P_{\geq} &= \{ w \in P \; | \; w \geq su(w) \} 
		\end{align*}
		These sets are well-define because  by construction for each fixed $w_0 \in P_0$, the fiber $u^{-1}(w_0)$ is a (possibly empty) totally ordered set. Next define continuous functors $(P, \leq) \to (P, \leq)$ as follows:
	\begin{equation*}
		F_{\leq}(w) = \begin{cases}
			su(w) & w \in P_{\leq} \\
			w & w \in P_{\geq} 
		\end{cases}
	\end{equation*}	
	\begin{equation*}
		F_{\geq}(w) = \begin{cases}
			w & w \in P_{\leq} \\
			su(w) & w \in P_{\geq} 
			\end{cases}
	\end{equation*}	
	The composite $F_\geq F_\leq = su$ and we have (unique) natural transformations
	\begin{equation*}
		su = F_\geq F_\leq \to F_\leq  \leftarrow id_{P}.
	\end{equation*}
	Thus the result follows. 
	\end{proof}

\subsubsection{The proof of Lemma~\ref{lem:3_equiv}}

Fix $[ \Bm] \in  (\bDelta^\op)^{n-1}$. Our aim is to show that the map (\ref{eqn:7})
\begin{equation*}
 B( {}_{IV}\Bord_{d;n}^{(X, \xi)}(M)_{[ \Bm], \bullet}) \to \Bord_{d;n-1}^{(X, \xi)}(M \times \R)_{[ \Bm]}
\end{equation*}
is a weak homotopy equivalence. This would follow from Lemma~\ref{lem:PosetRealization} if there were a section of the map 
\begin{equation*}
	u: {}_{IV}\Bord_{d;n}^{(X, \xi)}(M)_{[ \Bm]}  \to \Bord_{d;n-1}^{(X, \xi)}(M\times \R)_{[ \Bm]},
\end{equation*}
where both ${}_{IV}\Bord_{d;n}^{(X, \xi)}(M)_{[ \Bm]}$ and $\Bord_{d;n-1}^{(X, \xi)}(M\times \R)_{[ \Bm]}$ are viewed as a topological posets. 
Unfortunately no such section presents itself. 

However for each $\lambda \in \R$ we may define the subspace $V_\lambda \subseteq \Bord_{d;n-1}^{(X, \xi)}(M\times \R)_{[ \Bm]}$ which consists of all those tuples $( (\Bt^i)_{i=1}^{n-1}, (W, \theta))$ such that $\lambda$ is a regular value of the projection of $W$ onto the extra $\R$ direction. Let $P^\lambda = u^{-1}(V_\lambda)$ and $u_{\lambda}: P^\lambda = u^{-1}(V_\lambda) \to V_\lambda$ be the restriction. Then $u_\lambda$ does admit a section. This section takes the tuple $( (\Bt^i)_{i=1}^{n-1}, (W, \theta))$ to the pair
\begin{equation*}
	\left(\lambda, ( (\Bt^i)_{i=1}^{n-1}, (W, \theta)) \right),
\end{equation*}
and hence \begin{equation*}
	B u_\lambda: B P^\lambda \to V_\lambda
\end{equation*}
is a weak equivalence (by Lemma~\ref{lem:PosetRealization}). The same holds for the restrictions to any finite number of intersections of the $V_\lambda$ (the section adds, say, the least of the $\lambda$'s). The union of the $V_\lambda$ is all of $\Bord_{d;n-1}^{(X, \xi)}(M\times \R)_{[ \Bm]}$. If the $V_\lambda$ formed an open cover then we would be done by the gluing lemma for weak homotopy equivalences \cite[Thm.~6.7.11]{MR2456045}. However the $V_\lambda$ are not open unless $n=1$ and $M$ is compact, assumptions we do not want to make.\footnote{The case $n=1$ and $M$ compact is enough to recover the original theorem of Galatius-Madsen-Tillmann-Weiss. If we were only interested in their original result, we could stop here. }

It is important to understand why the $V_\lambda$ fail to be an open cover. Given $W \subseteq M \times\R  \times \R^{n-1}$, let $p$ be the projection onto the first $\R$ factor. Define
\begin{equation*}
	c(W) = \{ w \in W \; | \; w \text{ is a critical point of the projection} p:W \to \R \},
\end{equation*}
the critical locus.  The subset $c(W)$ is a closed subset of $ M \times\R  \times \R^{n-1}$, but the projection $p(c(W)) \subseteq \R$ to the first $\R$ coordinate is not necessarily closed since this projection is not proper. If this set fails to be closed at $\lambda \in \R$, then we can form a 1-dimensional family $\gamma$ which simply translates the manifold $W$ in the $\R$-direction. This is a plot, but $\gamma^{-1}(V_\lambda)$ will fail to be open, and hence $V_\lambda$ is not open.

\begin{example}
	Let $W \subseteq \R^2$ be the curve $y = \frac{\sin x}{x}$. Then $W$ is an embedded manifold. The map $p$ is projection to the $x$-axis, and the image of the critical set $c(W)$ in the $x$-axis converges to the origin, but does not contain the origin. 
\end{example}

\noindent However we can consider a modified version of the poset ${}_{IV}\Bord_{d;n}^{(X, \xi)}(M )_{[ \Bm]}$.  We define a topological space $Q_0\subseteq \Bord_{d;n-1}^{(X, \xi)}(M\times \R)_{[ \Bm]} \times \cl(\R)$ as: 
\begin{equation*}
	Q_0 = \{ (\BW, A) \; | \;  A \neq \R, \;  p(c(W)) \subseteq A \subseteq \R \}.
\end{equation*}
Here $\BW = ( (\Bt^i)_{i=1}^n, (W, \theta)) \in \Bord_{d;n-1}^{(X, \xi)}(M\times \R)_{[ \Bm]}$, and $\cl(Y)$ denotes the previously introduced topological space whose points are closed subsets of $Y$ (see Section~\ref{sec:space_of_closed_sets}). Thus $Q_0$ enhances the data of $\Bord_{d;n-1}^{(X, \xi)}(M\times \R)_{[ \Bm]}$
with the additional choice of a proper closed subset $A \subsetneq \R$ which contains the critical values of the projection of $W$ into the first $\R$-coordinate.

We also define the topological poset $(Q, \leq)$, analogously to the poset $(P, \leq)$, to be the subset of $(\R, \leq) \times Q_0$ consisting of those $(a, \BW, A)$ such that the value $a \in \R$ is a regular value of the projection $p:W \to \R$ to the first $\R$-coordinate. 
 
For $\lambda \in \R$ we define subsets $U_\lambda \subseteq Q_0$: 
\begin{equation*}
	U_\lambda = \{ (\BW, A) \; | \; \lambda \not\in A\}
\end{equation*}
These subsets are pulled back from open subsets $M(\{\lambda\})$ of $\cl(\R)$ (see the proof of Lemma~\ref{lem:mapsarepoltsforclosedsubsets}) and hence are open in $Q_0$. Since the closed subsets $A$ are proper subsets ($A \neq \R$) each $(\BW, A)$ is contained in some $U_\lambda$,
 and hence they form an open cover of $Q_0$. Moreover the restriction of $(Q, \leq)$ to each $U_\lambda$ (and each finite intersections of these) admits a section just as before in the case of $(P, \leq)$ and $V_\lambda$. However now, since the $U_\lambda$ form an open cover of $Q_0$, the gluing lemma for weak homotopy equivalences \cite[Thm.~6.7.11]{MR2456045} and Lemma~\ref{lem:PosetRealization} both apply and show that the map
\begin{equation*}
	|| u_Q ||: ||Q|| \to Q_0
\end{equation*}
is a weak homotopy equivalence. 

Moreover, we have a commuting square:
\begin{center}
\begin{tikzpicture}
		\node (LT) at (0, 1.5) {$||{}_{IV}\Bord_{d;n}^{(X, \xi)}(M)_{[ \Bm]}||$};
		\node (LB) at (0, 0) {$\Bord_{d;n-1}^{(X, \xi)}(M\times \R)_{[ \Bm]}$};
		\node (RT) at (6, 1.5) {$||Q||$};
		\node (RB) at (6, 0) {$Q_0$};
		\draw [->] (LT) -- node [left] {$||u||$} (LB);
		\draw [<-] (LT) -- node [above] {$j$} (RT);
		\draw [->] (RT) -- node [left] {$||u_Q||$} node [right] {$\simeq$} (RB);
		\draw [<-] (LB) -- node [below] {$j$} (RB);
\end{tikzpicture}
\end{center}
where the horizontal maps forget the closed subset $A$. We will show, provided $M$ is tame, that there exists another commuting square:
\begin{center}
\begin{tikzpicture}
		\node (LT) at (0, 1.5) {$||{}_{IV}\Bord_{d;n}^{(X, \xi)}(M)_{[\Bm]}||$};
		\node (LB) at (0, 0) {$\Bord_{d;n-1}^{(X, \xi)}(M\times \R)_{[ \Bm]}$};
		\node (RT) at (6, 1.5) {$||Q||$};
		\node (RB) at (6, 0) {$Q_0$};
		\draw [->] (LT) -- node [left] {$||u_P||$} (LB);
		\draw [->] (LT) -- node [above] {$v_K$} (RT);
		\draw [->] (RT) -- node [left] {$||u_Q||$} node [right] {$\simeq$} (RB);
		\draw [->] (LB) -- node [below] {$v_K$} (RB);
\end{tikzpicture}
\end{center}
such that the horizontal composites $j \circ v_K$ are homotopic to the identity map of $||{}_{IV}\Bord_{d;n}^{(X, \xi)}(M)_{[ \Bm]}||$, respectively $\Bord_{d;n-1}^{(X, \xi)}(M\times \R)_{[ \Bm]}$. It then follows, by the fact that weak equivalences form a saturated class\footnote{\emph{Saturated} in this context means that every map which becomes an isomorphism in the homotopy category was already a weak equivalence. We will have shown that in the homotopy category, $[||u||]$ is a retract of the isomorphism $[||u_Q||]$, and hence also an isomorphism.} that $||u||$ is also a weak equivalence. 

To construct the maps $v_K$ we need to use the fact that $M$ is tame, which we recall means that there exists a compact subset $K \subseteq M$ and a smooth 1-parameter family of embeddings $\psi^M_t: M \to M$ with $\psi^M_0 = id_M$ and $\psi^M_1(M) \subseteq K \subseteq M$. In addition we choose smooth 1-parameter families of embeddings $ \psi^i_t: \R \to \R$,  $1 \leq i \leq n-1$ such that:
\begin{itemize}
	\item $\psi_0^i = id_{\R}$
	\item $\psi_1^i(\R) \subseteq [t_0^i - 1, t^i_{m_i}+1]$
	\item $\psi^i_t(s) =s$ for $t_0^i - \frac{1}{2} \leq s \leq t^i_{m_i}+ \frac{1}{2} $ 
\end{itemize}
These combine to give a 1-parameter family of embeddings
\begin{equation*}
	\varphi_t = (\psi^M_t, id, (\psi^i_t)_{i=1}^{n-1}): M \times \R  \times \R^{n-1} \to M \times \R  \times \R^{n-1}. 
\end{equation*}
Set 
\begin{equation*}
	L =  \prod_{i =1}^{n-1} [t_0^i - 1, t^i_{m_i}+1] \subseteq \R^{n-1}
\end{equation*}
Then $\varphi_1(M \times \R \times \R^{n-1}) \subseteq K \times \R \times L$.

Viewing $Q_0\subseteq \Bord_{d;n-1}^{(X, \xi)}(M\times \R)_{[ \Bm]} \times \cl(\R)$ as a subspace of the product, we may write the map $v_K = (\varphi_1^*, a_K)$  in two parts. The first part is simply the pullback along the embedding $\varphi_1$. The second map is more subtle and is defined by:
\begin{align*}
	a_K: \Bord_{d;n-1}^{(X, \xi)}(M\times \R)_{[ \Bm]} &\to \cl(\R) \\
	\BW = ( (\Bt^i)_{i=1}^n, (W, \theta)) & \mapsto p( c(W) \cap K \times \R \times L)
\end{align*}
takes a submanifold, considers its critical locus $c(W) \subseteq M \times \R \times \R^{n-1}$, intersects this with $K \times \R \times L$, and finally projects the result to the $\R$-coordinate. This is well-defined because the projection $p:K \times \R \times L \to \R$ is proper and hence sends closed sets to closed sets. Moreover Sard's theorem states that $c(W)$ has Lebesgue measure zero, and hence $a_K(\BW) \neq \R$ is necessarily a \emph{proper} closed subset. Since  
\begin{equation*}
	p(c(\varphi_1^{-1}(W))) \subseteq  p( c(W) \cap K \times \R \times L) = a_K(\BW)
\end{equation*}
it follows that  $v_K = (\varphi_1^*, a_K)$ does indeed land in $Q_0$. 

If $\lambda \in \R$ is a regular value of the projection $p: W \to \R$, then $\lambda$ is also a regular value of the projection $p: \varphi_t^{-1}(W) \to \R$ and hence the map $v_K$ also induces a map of topological posets:
\begin{equation*}
	v_K: {}_{IV}\Bord_{d;n}^{(X, \xi)}(M)_{[ \Bm]} \to Q.
\end{equation*}

Finally the composite $j \circ v_K$ coincides with the map $\varphi^*_1$, which by construction is homotopic to the identity by the homotopy $\varphi^*_t$. This completes the proof of Lemma~\ref{lem:3_equiv} and hence also Theorem~\ref{thm:first_realization_thm}. \qed

\subsection{Madsen-Tillmann spectra}\label{sec:MTspectra}

Let $\xi_p: X_p \to Gr_d(\R^p)$ be a sequence of $GL_p$-equivariant fibrations together with $GL_p$-equivariant connecting maps $f_p:X_p \to X_{p+1}$ making the following diagram commute
\begin{center}
\begin{tikzpicture}
		\node (LT) at (0, 1.5) {$ X_p $};
		\node (LB) at (0, 0) {$Gr_d(\R^p) $};
		\node (RT) at (3, 1.5) {$ X_{p+1}$};
		\node (RB) at (3, 0) {$ Gr_d(\R^{p+1})$};
		\draw [->] (LT) -- node [left] {$ \xi_p$} (LB);
		\draw [->] (LT) -- node [above ] {$ f_p$} (RT);
		\draw [->] (RT) -- node [right] {$ \xi_{p+1}$} (RB);
		\draw [->] (LB) -- node [below] {$ $} (RB);
\end{tikzpicture}
\end{center}
where $Gr_d(\R^p) \to Gr_d(\R^{p+1})$ is induced by the standard inclusion of $\R^p$ into $\R^{p+1}$. We have a canonical isomorphism 
\begin{equation*}
	f_p^* \xi_{p+1}^* \gamma^\perp_d \cong \xi_{p}^* \gamma^\perp_d \oplus \varepsilon
\end{equation*}
of vector bundles over $X_p$, where $\varepsilon$ denotes trivial bundle of rank one. Hence we have induced maps of Thom spaces:
\begin{equation*}
	\Sigma Th(\xi^*_p \gamma^\perp_d) \to Th(\xi^*_{p+1} \gamma^\perp_d).
\end{equation*}

\begin{definition}\label{def:Madsen-Tillmann-spectra}
	Let $\Bxi = \{ (X_p, \xi_p)\}$ denote a collection of $X_p$ with connecting maps, as above. Then the \emph{Madsen-Tillmann} spectrum is the Thom spectrum $M T\Bxi$ whose $p^\text{th}$ space is 
	\begin{equation*}
		(M T\Bxi)_p = Th(\xi_p^* \gamma_d^\perp)
	\end{equation*}
	and with the above defined connecting maps. 
\end{definition}

\begin{example}[orientations]
	We may take $X_p$ to be the Grassmanian of oriented $d$-planes in $\R^p$. In this case we write $M TSO(d)$ for the corresponding Madsen-Tillmann spectrum.
\end{example}

The Madsen-Tillmann spectrum $M T\Bxi$ is $(-d)$-connective. As a homology theory we have $M T\Bxi_k(Y)$ is represented by (c.f. \cite{179327}.)
\begin{enumerate}
	\item A closed $(d+k)$-dimensional manifold $M$ embedded into $\R^{p + k}$, with normal bundle $\nu_M$; 
	\item A map $g:M \to X_p$ for some $p$;
	\item An isomorphism $\nu_M \cong g^* \xi_p^* \gamma_d^\perp$; and 
	\item a continuous map $M \to Y$. 
\end{enumerate}
This data is taken up to cobordism in the obvious way together with stablizing the map $g$ along the connecting maps $X_p \to X_{p+1}$, and the embedding along the inclusion $\R^{p + k} \subseteq \R^{p + 1+ k}$. This permits us in many cases to calculate the negative homotopy groups of $M T\Bxi$. See Section~\ref{sec:lod-dim-homotopy-groups} and Appendix~\ref{app:lowhomotopy}.

\subsection{The symmetric monoidal bordism category} \label{subsec:symbord}

Let $\Bxi = \{ (X_p, \xi_p)\}$ be as in the previous section. The for each $d$ and $n$ we get a corresponding family of $n$-fold Segal spaces $\Bord_{d;n}^{(X_{p+n}, \xi_{p+n})}(D^p)$. The $p^\textrm{th}$ term in this sequence is a $E_p$-algebra and we have natural connecting maps:
\begin{equation*}
	\Bord_{d;n}^{(X_{p+n}, \xi_{p+n})}(D^p) \to \Bord_{d;n}^{(X_{p+ 1+n}, \xi_{p+1+n})}(D^{p+1})
\end{equation*}
which are $E_p$-algebra maps. 

The colimit, which we will denote $\Bord^{\Bxi}_{d;n}$, is an $E_\infty$ $n$-fold Segal space and hence is an example of a symmetric monoidal $(\infty,n)$-category. The following is a direct consequence of Theorem~\ref{thm:E_pSpaces} and Theorem~\ref{thm:first_realization_thm}:

\begin{theorem}\label{thm:stable_main}
	If $n\geq 1$ then there is a natural equivalence of $E_\infty$-algebras between the geometric realization $B^n\Bord^{\Bxi}_{d;n}$ and $\Omega^{\infty-n} MT \Bxi$. \qed
\end{theorem}

\section{Examples and applications}\label{subsec:example_applications}

In this final section we will give several applications of the classification of invertible topological field theories. We will classify certain simple oriented topological field theories with dimensions and category numbers $d, n \leq 4$. Our computations will also lead to a negative answer to an open question raised by Gilmer-Masbaum  \cite[Rmk.~7.5]{MR3100961}.

\subsection{Covers of Madsen-Tillmann spectra}

Let $\Bord_{d;n}^{SO(d)}$ denote the oriented $d$-dimensional symmetric monoidal $(\infty,n)$-category. As we have seen we have a natural identification of infinite loop spaces:
\begin{equation*}
	||\Bord_{d;n}^{SO(d)}|| \simeq \Omega^{\infty -n} MTSO(d) = \Omega^\infty \Sigma^n MTSO(d) \simeq \Omega^\infty \Sigma^n p_{\geq -n} MTSO(d)
\end{equation*}
where $p_{\geq k}E$ is the \emph{Postnikov cover} of the spectrum $E$. We have $\pi_i p_{\geq k}E = 0$ for $i < k$ and there is a map $p_{\geq k}E \to E$ inducing an isomorphism on $\pi_i$ for $i \geq k$. 

The categories $\Bord_{d;n}^{SO(d)}$ are related for different values of $n$ and $d$. For example the $n$-category $\Bord_{d; n}^{SO(d)}$ sits inside the $(n+1)$-category $\Bord_{d; n+1}^{SO(d)}$ as the $n$-category of endomorphisms of the empty $(d-n-1)$-manifold. Said differently, we can use the symmetric monoidal structure to view $\Bord_{d; n}^{SO(d)}$ as an $(n+1)$-category (with one object). To notate this we will add a `$B$' in front to indicate this sort of categorical delooping. Then there is an inclusion map 
\begin{equation*}
	B\Bord_{d; n}^{SO(d)} \to \Bord_{d; n+1}^{SO(d)}
\end{equation*}
Upon passing to geometric realizations this corresponds to the map 
\begin{align*}
	\Omega^\infty \Sigma^n p_{\geq -n} MTSO(d) \simeq & \\
	 \Omega^\infty \Sigma^n MTSO(d)  &\to \Omega^{\infty + 1} \Sigma^{n+1} MTSO(d) \\
	 & \simeq  \Omega^\infty \Sigma^n p_{\geq -n-1} MTSO(d)
\end{align*}
induced from $p_{\geq -n} MTSO(d) \to p_{\geq -n-1} MTSO(d)$. 
Similarly the $d$-dimensional $n$-category $\Bord_{d;n}^{SO(d)}$ also sits inside the $(d+1)$-dimensional $(n+1)$-category $\Bord_{d+1;n+1}^{SO(d+1)}$ as the objects through to the $n$-morphisms. Upon passing to geometric realizations we get the following map of infinite loop spaces:
\begin{equation*}
	\Omega^\infty \Sigma^n MTSO(d) \to \Omega^\infty \Sigma^{n+1} MTSO(d+1).
\end{equation*}

Interpreted in the above way and letting $d$ and $n$ range over $1,2,3,4$ we obtain a grid of higher categories and maps between them. 
\begin{center}
\begin{tikzpicture}
	\matrix[row sep=3mm,column sep=5mm]{
		\node (p11) {$B^3\Bord_{4;1}^{SO(4)}$}; &  
		\node (p12) {$B^2\Bord_{4;2}^{SO(4)}$}; & 
		\node (p13) {$B\Bord_{4;3}^{SO(4)}$}; & 
		\node (p14) {$\Bord_{4;4}^{SO(4)}$}; \\
		
		&
		\node (p21) {$B^2\Bord_{3;1}^{SO(3)}$}; &  
		\node (p22) {$B\Bord_{3;2}^{SO(3)}$}; & 
		\node (p23) {$\Bord_{3;3}^{SO(3)}$};  \\
		&&
 		\node (p31) {$B\Bord_{2;1}^{SO(2)}$}; &  
 		\node (p32) {$\Bord_{2;2}^{SO(2)}$};  \\
		&&&
		\node (p41) {$\Bord_{1;1}^{SO(1)}$}; \\
		};
		\draw[->] (p11) -- (p12);
		\draw[->] (p12) -- (p13);
		\draw[->] (p13) -- (p14);
		
		\draw[->] (p21) -- (p22);
		\draw[->] (p22) -- (p23);
		
		\draw[->] (p31) -- (p32);
		
		\draw[->] (p41) -- (p32);
		\draw[->] (p31) -- (p22);
		\draw[->] (p21) -- (p12);
		
		\draw[->] (p32) -- (p23);
		\draw[->] (p22) -- (p13);
		
		\draw[->] (p23) -- (p14);
\end{tikzpicture}
\end{center}
Passing to geometric realizations gives a corresponding grid of infinite loop spaces and maps as depicted in Figure~\ref{fig:grid-of-loop-spaces}. We will show the indicated maps are weak homotopy equivalences in Cor.~\ref{cor:grid-equivs} below.

\begin{figure}[h]
	\centering
		\begin{tikzpicture}
			\matrix[row sep=3mm,column sep=4mm]{
				\node (p11) {$\Omega^\infty p_{\geq 3} \Sigma^4 MTSO(4)$}; &  
				\node (p12) {$\Omega^\infty p_{\geq 2} \Sigma^4 MTSO(4)$}; & 
				\node (p13) {$\Omega^\infty p_{\geq 1} \Sigma^4 MTSO(4)$}; & 
				\node (p14) {$\Omega^\infty \Sigma^4 MTSO(4)$}; \\
		
				&
				\node (p21) {$\Omega^\infty p_{\geq 2} \Sigma^3 MTSO(3)$}; &  
				\node (p22) {$\Omega^\infty p_{\geq 1} \Sigma^3 MTSO(3)$}; & 
				\node (p23) {$\Omega^\infty \Sigma^3 MTSO(3)$};  \\
				&&
		 		\node (p31) {$\Omega^\infty p_{\geq 1} \Sigma^2 MTSO(2)$}; &  
		 		\node (p32) {$\Omega^\infty \Sigma^2 MTSO(2)$};  \\
				&&&
				\node (p41) {$\Omega^\infty \Sigma MTSO(1)$}; \\
				};
				\draw[->] (p11) -- node[above] {$\sim$} (p12);
				\draw[->] (p12) -- node[above] {$\sim$}(p13);
				\draw[->] (p13) -- (p14);
		
				\draw[->] (p21) -- node[above] {$\sim$}(p22);
				\draw[->] (p22) -- (p23);
		
				\draw[->] (p31) -- (p32);
		
				\draw[->] (p41) -- (p32);
				\draw[->] (p31) -- (p22);
				\draw[->] (p21) -- (p12);
		
				\draw[->] (p32) -- (p23);
				\draw[->] (p22) -- (p13);
		
				\draw[->] (p23) -- (p14);
		\end{tikzpicture}
	\caption{A grid of maps of infinite loop spaces.}
	\label{fig:grid-of-loop-spaces}
\end{figure}

\subsection{Low dimensional homotopy groups of Madsen-Tillmann Spectra} \label{sec:lod-dim-homotopy-groups}
Some of the maps in figure~\ref{fig:grid-of-loop-spaces} are equivalences of infinite loop spaces. Which ones are equivalences can be seen by computing the low dimensional homotopy groups of the corresponding spectra. 
The description at the end of Section~\ref{sec:MTspectra} identifies the homotopy groups of $MTSO(d)$ with the \emph{vector field cobordism groups} which have been computed in low degrees \cite{MR3356279}. For $k < d$ they agree with classical oriented bordism groups:
\begin{equation*}
	\pi_k \Sigma^d MTSO(d) \cong \Omega^\text{or}_{k}, \quad k < d.
\end{equation*}
When $k = d$, $d+1$, and $d+2$, these groups have also been computed \cite{MR3356279}. For the reader's benefit we compute the groups $\pi_d \Sigma^dMTSO(d)$ for $d \leq 4$ in Appendix~\ref{app:lowhomotopy}. This group is given as the quotient of the monoid of diffeomorphism classes of closed compact oriented manifolds $Y$ be the equivalence relation that $[Y] \simeq 0$ whenever there exists a compact oriented $(d+1)$-manifold $W$ with $\partial W \cong Y$ equipped with a non-vanishing vector field restricting to the inward pointing vector field on $Y$.  
For now it suffices to simply quote the result of \cite{MR3356279} which identifies this group:

\begin{theorem}[\cite{MR3356279}]\label{thm:homotopy-MT-spectra}
	We have:
	\begin{equation*}
		\pi_0 MTSO(d) \cong \begin{cases}
		\Z	\oplus \Omega_d^\text{or} & \text{if $d \equiv 0$ mod 4} \\
		\Z/2 \oplus \Omega_d^\text{or}	 & \text{if $d \equiv 1$ mod 4} \\
		\Z	\oplus \Omega_d^\text{or} & \text{if $d \equiv 2$ mod 4} \\
			 \Omega_d^\text{or} & \text{if $d \equiv 3$ mod 4} \\
		\end{cases}
	\end{equation*}
	If $q: \pi_0MTSO(d) \to \Omega^\text{or}_d$ is the natural quotient map then these splittings are given by: 
	\begin{itemize}
		\item $(\frac{1}{2}(\chi + \sigma), q)$ when $d\equiv 0$ mod 4;
		\item $(\frac{1}{2}\chi , q)$ when $d \equiv 2$ mod 4;
		\item $(k_\R, q)$ when $d \equiv 1$ mod 4;
	\end{itemize}
	 where $\chi$ and $\sigma$ are the Euler characteristic and signature, respectively, and  
	$k_\R$ is the mod $2$ reduction of the real form of Kervaire's semi-characteristic:
			\begin{equation*}
				k_\R(M) = \sum_{i = 0}^{(d-1)/4} \dim_\R H^{2i}(M; \R) \text{ mod 2}.
			\end{equation*}
\end{theorem} \qed

Here is a table summarizing the above statements about the homotopy groups of $\pi_k \Sigma^dMTSO(d)$ for $d$ up to $4$:
\begin{equation} \label{eqn:homotopy-table}
	\begin{tabular}{l |  llll l}
		$d \backslash k$ &  0 & 1 & 2 & 3 & 4\\ \hline
		1 &  $\Z$ & $\Z/2\Z$ & $\cdot$  & $\cdot$ &  $\cdot$ \\
		2 & $\Z$ & 0 & $\Z$ & $\cdot$ & $\cdot$ \\
		3 & $\Z $ & 0 & 0 & 0 & $\cdot$ \\
		4 & $\Z$ & 0 & 0 & 0 & $\Z \oplus \Z$ \\
	\end{tabular}
\end{equation}
in fact $\Sigma MTSO(1) \simeq \S^0$, and the first row corresponds to the stable stems. 

\begin{corollary}\label{cor:grid-equivs}
	The arrows in Figure~\ref{fig:grid-of-loop-spaces} which are indicated to be equivalences are in fact equivalences. 
\end{corollary}

\begin{corollary}\label{cor:fibersequence}
	For $d = 2,3,4$ there exists a fiber sequence of spectra
	\begin{equation*}
		p_{\geq 1}\Sigma^dMTSO(d) \to \Sigma^d MTSO(d) \to H \Z
	\end{equation*}
\end{corollary}

\noindent Thus in Figure~\ref{fig:grid-of-loop-spaces} there are seven distinct infinite loop spaces corresponding to $\Omega^\infty \Sigma^dMTSO(d)$ for $d=1, \dots, 4$ and $\Omega^\infty p_{\geq 1} \Sigma^d MTSO(d)$ for $d=2, 3, 4$ . The special case $d=1$ is well known: $\Omega^\infty \Sigma MTSO(1) \simeq Q(S^0)$ is the infinite loop space underlying the sphere spectrum. 

\subsection{Cohomology of (covers of) Madsen-Tillmann spectra} \label{sec:cohom-of-MT-spectra}

In this section we will review how to compute the infinite loop maps from the remaining six non-trivial infintie loop spaces to Eilenberg-MacLane spaces $K(A,n)$. The spectrum $\Sigma^dMTSO(d)$ is a Thom spectrum for the virtual vector bundle $\varepsilon^d - \gamma_d$ of virtual dimension zero over the space $BO(d)$. Hence both the spectrum $\Sigma^dMTSO(d)$ and $p_{\geq 1}\Sigma^d MTSO(d)$ are connective spectra.  If $E$ is any spectrum this implies that infinite loop maps from $\Omega^\infty\Sigma^dMTSO(d)$ and $\Omega^\infty p_{\geq 1}\Sigma^d MTSO(d)$ to $\Omega^\infty E$ are the same as maps of spectra from $\Sigma^dMTSO(d)$ and $p_{\geq 1}\Sigma^d MTSO(d)$ to $E$. 

When $E = H \Z$ we get the following results.

\begin{theorem}\label{thm:cohomology}
	The integral cohomology of the spectra $\Sigma^d MTSO(d)$ and $p_{\geq 1}\Sigma^dMTSO(d)$ for $d=2,3,4$ in degrees $k= 0, \dots 5$, together with generating elements, is listed in the following table:  
\begin{equation*}
	\begin{tabular}{|c | lllll l |}
		\hline
		* & 0 & 1 & 2 & 3 & 4 & 5\\
		\hline
		$H\Z^*(\Sigma^4MTSO(4))$ & $\Z$ & 0 & 0 & $\Z/2\Z$ & $\Z \oplus \Z$ & 0\\
		 & $u$ &  &  & $W_3 u$ & $eu, p_1u$&  \\ \hline
		$H\Z^*(\Sigma^3MTSO(3))$ & $\Z$ & 0 & 0 & $\Z/2\Z$ & $\Z$ & 0\\
		 & $u$ &  &  & $W_3u$ & $p_1 u$ & \\ \hline
		$H\Z^*(\Sigma^2MTSO(2))$ & $\Z$ & 0 & $\Z$ & 0 & $\Z$& 0 \\
		 & $u$ &  & $cu$ &  & $c^2u$&\\ \hline \hline
 		$H\Z^*(p_{\geq 1}\Sigma^4MTSO(4))$ & 0 & 0 & 0 & 0 & $\Z \oplus \Z$ & 0\\
 		 &  &  &  &  & $\psi, \sigma $&  \\ \hline
 		$H\Z^*(p_{\geq 1}\Sigma^3MTSO(3))$ & $0$ & 0 & 0 & 0 & $\Z$ & 0\\
 		 &  &  &  &  & $\rho$ & \\ \hline
 		$H\Z^*(p_{\geq 1}\Sigma^2MTSO(2))$ & 0 & 0 & $\Z$ & 0 & $\Z$& 0 \\
 		 &  &  & $\tau$ &  & $\rho$&\\ \hline 
	\end{tabular}
\end{equation*}	
\noindent	For $H\Z^*(\Sigma^d MTSO(d))$ these are isomorphisms as $H^*(BSO(d); \Z)$-modules, as explained below.
	
	Moreover, the following restriction maps preserve generators with the same names and have the indicated effect on the remaining generators:
	\begin{center}
	\begin{tikzpicture}
			\node (LT) at (0, 3) {$ H\Z^*(\Sigma^4MTSO(4))  $};
			\node (LC) at (0, 1.5) {$ H\Z^*(\Sigma^3MTSO(3)) $};
			\node (LB) at (0, 0) {$ H\Z^*(\Sigma^2MTSO(2))$};
			
			\node (RT) at (5, 3) { $H\Z^*(p_{\geq 1}\Sigma^4MTSO(4))$ };
			\node (RC) at (5, 1.5) {$ H\Z^*(p_{\geq 1}\Sigma^3MTSO(3))$};
			\node (RB) at (5, 0) {$ H\Z^*(p_{\geq 1}\Sigma^2MTSO(2))$};
			
			\draw [->] (LT) -- node [left] {$ $} (LC);
			\draw [->] (LC) -- node [left] {$ $} (LB);
			
			\draw [->] (RT) -- node [right] {$ $} (RC);
			\draw [->] (RC) -- node [right] {$ $} (RB);
			
			\draw [->] (LT) -- node [above left] {$ $} (RT);
			\draw [->] (LC) -- node [above left] {$ $} (RC);
			\draw [->] (LB) -- node [below] {$ $} (RB);
			
			\node [red] (p1) at (2, 2) {$p_1u$};
			\node [red] (c2) at (2, 0.5) {$-c^2u$};
			\node [red] (c1) at (2, -0.5) {$cu$};
			\node [red] (6rho) at (4.5, 2) {$6\rho$};
			\node [red] (2rho) at (7, 2) {$2\rho$};
			\node [red] (sigma) at (7, 3.5) {$\sigma$};
			\node [red] (2tau) at (4.5, -0.5) {$2\tau$};
			\node [red] (e) at (-2, 3.5) {$eu$};
			\node [red] (zero) at (-2, 2) {$0$};
			\node [red] (p1-2) at (-1, 4) {$p_1 u$};
			\node [red] (2psi-sig) at (4,3.5) {$2 \psi - \sigma$};
			\node [red] (3sigma) at (4,4) {$3 \sigma$};
			\node [red] (psi) at (7.5, 3.5) {$\psi$};
			\node [red] (1rho) at (7.5,2) {$\rho$};
			
			\draw [red, |->] (p1) -- (c2);
			\draw [red, |->] (p1) -- (6rho);
			\draw [red, |->] (sigma) -- (2rho);
			\draw [red, |->] (c1) -- (2tau);
			\draw [red, |->] (e) -- (zero);
			\draw [red, |->] (p1-2) -- (3sigma);
			\draw [red, |->] (e) -- (2psi-sig);
			\draw [red, |->] (psi) -- (1rho);

	\end{tikzpicture}
	\end{center}
\end{theorem}
\begin{proof}

We will prove Thm~\ref{thm:cohomology} over the course of this section. First we observe that since the spectra $\Sigma^d MTSO(d)$ are Thom spectra, for any $SO$-oriented cohomology theory $E$, such as $H \Z$ and $H \F_p$, we have a Thom isomorphism: 
\begin{equation*}
	E^*(\Sigma^dMTSO(d)) \cong E^*(BSO(d)) \cdot u
\end{equation*}
where $u$ is the $E$-theory Thom class of the virtual vector bundle $\varepsilon^d - \gamma_d$. Since the virtual dimension of $\varepsilon^d - \gamma_d$ is zero, the Thom class $u$ is degree zero and the Thom isomorphism in this case does not shift degree. 

The first three cases of Thm~\ref{thm:cohomology} now follow from the next proposition.

\begin{proposition}[{See for example \cite{MR652459}}]\label{prop:bso-cohomology}
	The integral cohomology of $BSO(d)$ for $d= 2,3,4$ is given as a graded ring by:
	\begin{align*}
		H^*(BSO(4); \Z) & \cong \Z[W_3, e, p_1]/ (2W_3) \\
		H^*(BSO(3); \Z) & \cong \Z[W_3, p_1]/ (2W_3) \\
		H^*(BSO(2); \Z) & \cong \Z[c]
	\end{align*}
	where $|c|= 2$, $|W_3| =3$, and $|p_1| = |e| = 4$. Under this isomorphism the natural restriction maps preserve generators with the same name, send $e$ to zero, and send $p_1$ to $-c^2$. 	
\qed
\end{proposition}

\noindent The boundary map induced by the short exact sequence $\Z \to \Z \to \Z/2$ gives rise the \emph{integral Bockstein}:
\begin{equation*}
	\beta: H^*(X; \Z/2\Z) \to H^{*+1}(X;\Z).
\end{equation*}
In the above description the class $W_3 = \beta(w_2)$ where $w_i \in H^i(BSO(d); \Z/2\Z)$ is the $i^\text{th}$ Steifel-Whitney class. 

From the computation of homotopy groups listed in the table in Eq.~(\ref{eqn:homotopy-table}) we observe that $p_{\geq 1}\Sigma^3MTSO(3)$ and $p_{\geq 1}\Sigma^4MTSO(4)$ are each 3-connected. By the Hurewicz theorem it follows that 
\begin{equation*}
	H\Z^k(p_{\geq 1}\Sigma^3MTSO(3)) = H\Z^k(p_{\geq 1}\Sigma^4MTSO(4)) = 0
\end{equation*} 
for $k \leq 3$, and that $H\Z_4(p_{\geq 1}\Sigma^4MTSO(4)) \cong \pi_4 p_{\geq 1}\Sigma^4MTSO(4) \cong \Z \oplus \Z$. The universal coefficient theorem then implies that $H\Z^4(p_{\geq 1}\Sigma^4MTSO(4)) \cong \Z \oplus \Z$, the $\Z$-linear dual of homology. The generators of these two $\Z$'s are classes $\psi$ and $\sigma$ which induce invariants of the vector field bordism groups $\pi_4(\Sigma^4 MTSO(4))$. By Theorem~\ref{thm:homotopy-MT-spectra} and the well-know identification of $\Omega_4\cong \Z$ by the signature, we have that the invariant corresponding to $\sigma$ is the signature $sign$, and to $\psi$ 
is $\frac{1}{2}(\chi + sign)$, half the sum of the signature and Euler characteristic. 

It follows, from the above discussion and from the Hirzebruch signature theorem that the natural map
\begin{equation*}
	\Z \oplus \Z \cong H\Z^4(\Sigma^4 MTSO(4)) \to H\Z^4(p_{\geq 1}\Sigma^4MTSO(4)) \cong \Z \oplus \Z
\end{equation*} 
sends the generator $p_1u$ to $3 \sigma$ and the generator $eu$ to $2 \psi - \sigma$, as claimed.

To access the remaining cohomology of the connected cover $p_{\geq 1}\Sigma^dMTSO(d)$ we will utilize the fiber sequence of spectra from Cor.~\ref{cor:fibersequence}:
\begin{equation*}
	p_{\geq 1}\Sigma^dMTSO(d) \to \Sigma^d MTSO(d) \to H \Z.
\end{equation*}
 It induces a long exact sequence in $H\Z$-cohomology. The computation of the integral cohomology of $H\Z$ is a classical exercise. The relevant groups are listed below. 
\begin{equation*}
		\begin{tabular}{c | ccccc cc }
			k                & 0      & 1 & 2 & 3 & 4 & 5 & 6    \\ \hline
			$(H\Z)^k(H\Z)$  &  $\Z$ & 0 &  0 & $\Z/2\Z$  &  0  & $\Z/6\Z$  & 0    \\
		\end{tabular}
	\end{equation*}
This immediately implies that  $H\Z^5(p_{\geq 1}\Sigma^dMTSO(d)) \cong 0$ (for $d = 2,3,4$), and yields a short exact sequence:
\begin{equation} \label{eqn:h2mtso(2)}
	0 \to \Z \to H\Z^2(p_{\geq 1}\Sigma^2MTSO(2)) \to \Z/2\Z \to 0.
\end{equation}
By table~\ref{eqn:homotopy-table} the first non-trivial homotopy group of $p_{\geq 1}\Sigma^2MTSO(2)$ is $\pi_2$ which is $\Z$. It follows, again from the Hurewicz theorem and the universal coefficent theorem, that $H\Z^2(p_{\geq 1}\Sigma^2MTSO(2)) \cong \Z$, generated by a class $\tau$. This determines the above short exact sequence and implies that the generator $c$ of $H\Z^2(\Sigma^2MTSO(2))$ is mapped to $2\tau$, twice the generator of $H\Z^2(p_{\geq 1}\Sigma^2MTSO(2))$.

The remaining portions of the long exact sequence in $H\Z$-cohomology fit into the following commutative diagram whose rows are short exact sequences:
\begin{center}
\begin{tikzpicture}
		\node (LLT) at (0, 3) {$ 0 $};
		\node (LT) at (1.5, 3) {$ \Z eu \oplus \Z p_1u $};
		\node (MT) at (4.5, 3) {$ \Z \psi \oplus \Z \sigma $};
		\node (RT) at (7.5, 3) {$ \Z/6\Z $};
		\node (RRT) at (9, 3) {$ 0 $};
		
		\node (LLC) at (0, 1.5) {$ 0 $};
		\node (LC) at (1.5, 1.5) {$ \Z p_1u $};
		\node (MC) at (4.5, 1.5) {$ H\Z^4(p_{\geq 1}\Sigma^3MTSO(3)) $};
		\node (RC) at (7.5, 1.5) {$  \Z/6\Z $};
		\node (RRC) at (9, 1.5) {$ 0 $};
		
		\node (LLB) at (0, 0) {$ 0 $};
		\node (LB) at (1.5, 0) {$ \Z (-c^2) $};
		\node (MB) at (4.5, 0) {$ H\Z^4(p_{\geq 1}\Sigma^2MTSO(2)) $};
		\node (RB) at (7.5, 0) {$  \Z/6\Z $};
		\node (RRB) at (9, 0) {$ 0 $};

		\draw [->] (LLT) -- node [above] {$  $} (LT);
		\draw [->] (LT) -- node [above] {$ \circled{3} $} (MT);
		\draw [->] (MT) -- node [above] {$  $} (RT);
		\draw [->] (RT) -- node [above] {$  $} (RRT);
		
		\draw [->] (LLC) -- node [above] {$  $} (LC);
		\draw [->] (LC) -- node [above] {$  $} (MC);
		\draw [->] (MC) -- node [above] {$  $} (RC);
		\draw [->] (RC) -- node [above] {$  $} (RRC);
		
		\draw [->] (LLB) -- node [above] {$  $} (LB);
		\draw [->] (LB) -- node [above] {$  $} (MB);
		\draw [->] (MB) -- node [above] {$  $} (RB);
		\draw [->] (RB) -- node [above] {$  $} (RRB);
		
		\draw [->] (LT) -- node [right] {$ \circled{2}  $} (LC);
		\draw [->] (MT) -- node [right] {$  $} (MC);
		\draw [->] (RT) -- node [right] {$ = $} (RC);
		
		\draw [->] (LC) -- node [right] {$ \cong $} (LB);
		\draw [->] (MC) -- node [right] {$ \circled{1} $} (MB);
		\draw [->] (RC) -- node [right] {$ = $} (RB);

\end{tikzpicture}
\end{center}
From the five-lemma the arrow marked with a $\circled{1}$ is an isomorphism. The vertical arrow marked with a $\circled{2}$ sends the generator $eu$ to zero and $p_1u$ to $p_1u$. As we have seen above, the arrow marked with $\circled{3}$ sends $p_1u$ to $3 \sigma$ and $eu$ to $2 \psi - \sigma$. 
Thus since $eu$ maps to zero in $H\Z^4(p_{\geq 1}\Sigma^3MTSO(3))$, it follows that $2 \psi$ and $\sigma$ have the same image in $H\Z^4(p_{\geq 1}\Sigma^3MTSO(3))$. Moreover the image of $p_1 u$ agrees with the image of $3 \sigma$ which then agrees with the image of $6 \psi$ in $H\Z^4(p_{\geq 1}\Sigma^3MTSO(3))$. In summary, the short exact sequence 
\begin{equation*} 
	0 \to \Z \to H\Z^4(p_{\geq 1}\Sigma^3MTSO(3)) \to \Z/6 \Z \to 0
\end{equation*}
is one in which the image of $p_1u$ is a multiple of 6 times another element $\rho$. The only possibility for this extension is that 
\begin{equation*}
	H\Z^4(p_{\geq 1}\Sigma^3MTSO(3)) \cong \Z
\end{equation*}
generated by the element $\rho$.  Furthermore the image of $\psi$ is also $\rho$ and the image of $\sigma$ is $2 \rho$. This concludes the proof of Theorem~\ref{thm:cohomology}.
 \end{proof}

\subsection{Application: The classification of invertible TFTs in low dimensions} \label{sec:application-classification}

We will now use the above cohomology calculations to prove Theorem~\ref{thm:classify-invert-tqft} and classify certain invertible topological field theories in dimensions less than or equal to four. Recall from the introduction the higher categories $\Vect_n$ of \emph{$n$-vector spaces}. These were obtained by iterating Kapranov and Voevodsky's construction of the 2-category $\Vect_2$ of 2-vector spaces \cite{kapranov19942}. A key property of this symmetric monoidal $n$-category is that the Picard subcategory is a model for $K(\C^\times, n)$.

\begin{theorem}\label{thm:classify-invert-tqft}
For $1 \leq n \leq d \leq 4$ consider the symmetric monoidal functors 	
	\begin{equation*}
		Z: \Bord_{d;n}^{SO(d)} \to \Vect_n
	\end{equation*}
	landing in the Picard subcategory of $\Vect_n$, that is the invertible topological quantum field theories. Let $\TQFT_{d;n}^\textrm{invert}$ denote the set of natural isomorphism classes of such functors. These are 
are classified as follows:
	\begin{enumerate}
		\item When $d=1$ or $d=3$ (all allowed $n$) there is a unique such field theory up to natural isomorphism: the constant functor with value the unit object of $\Vect_n$.
		\item When $d=2$ and $n=1$ or $n=2$ such field theories are determined by a single invertible complex number. The restriction map
		\begin{equation*}
			\TQFT^\textrm{invert}_{2;2} \to \TQFT^\textrm{invert}_{2;1}
		\end{equation*}
		squares this number. 
		\item When $d=4$, then such field theories are determined by a pair of invertible complex numbers.
		The restriction maps
		\begin{equation*}
			\TQFT^\textrm{invert}_{4;3} \to \TQFT^\textrm{invert}_{4;2} \to \TQFT^\textrm{invert}_{4;1}
		\end{equation*}
		are isomorphisms (bijections). The restriction map $\TQFT^\textrm{invert}_{4;4} \to \TQFT^\textrm{invert}_{4;3}$ is given as follows:
		\begin{align*}
			\TQFT^\textrm{invert}_{4;4} &\to \TQFT^\textrm{invert}_{4;3} \\
			(\lambda_1,\lambda_2) & \mapsto ( \lambda_1^2, \frac{\lambda_2^3}{\lambda_1})
		\end{align*}
	\end{enumerate}
\end{theorem}

\begin{remark}
	This final restriction map is 6-to-1. If $\zeta$ is any sixth root of unity, then the fully local $(4;4)$-TQFTs corresponding to $(\lambda_1, \lambda_2)$ and to $(\zeta^3 \lambda_1, \zeta \lambda^2)$ have the same restriction to $(4;3)$-TQFTs.
\end{remark}

\begin{proof}[Proof of Theorem~\ref{thm:classify-invert-tqft}]
	By our previous discussions we can identify $\TQFT^\textrm{invert}_{d;n}$ with the spectrum cohomology group
	\begin{equation*}
		\TQFT^\textrm{invert}_{d;n} \cong H^d(p_{\geq{d-n}}\Sigma^dMTSO(d); \C^\times).
	\end{equation*}
	These groups and the corresponding restriction maps are easily computed from Thm~\ref{thm:cohomology} and the long exact cohomology sequence coming from the short exact sequence
	\begin{equation*}
		\Z \to \C \to \C^\times.
	\end{equation*}
	For example in the case $d=n=4$ the relevant portion of the long exact sequence splits apart and gives a short exact sequence:
	\begin{align*}
		0 \to &H^4(\Sigma^dMTSO(d); \Z) \to && H^4(\Sigma^dMTSO(d);\C) \to H^4(\Sigma^dMTSO(d);\C^\times) \to 0 \\
		& \cong \Z\langle p_1\rangle \oplus \Z\langle e\rangle && \cong\C\langle p_1\rangle \oplus \C\langle e\rangle
	\end{align*}
	From which it follows that $H^4(\Sigma^dMTSO(d);\C^\times) \cong \C^\times \oplus \C^\times$. The computations of the other groups and the effect of the restriction maps follow straightforwardly from Thm~\ref{thm:cohomology}. We leave the details to the reader. 
\end{proof}


\subsection{Application: A solution to a question of Gilmer and Masbaum} \label{sec:app-GM}

We now turn to our second application, which answers a question posed by Gilmer and Masbaum in \cite{MR3100961}. Gilmer and Masbaum's question concerns a certain modification of the 3-dimensional bordism category used in the process of anomaly cancelation in 3-dimensional topological field theories. Let's first recall the  context surrounding Gilmer and Masbaum's question.

The Reshetikhin-Turaev construction aims to produce a 3-dimensional topological field theory from a modular tensor category.  Modular tensor categories can be constructed from representations of quantum groups, representations of loop groups, and by other means. However there is often a problem in that the resulting 3-dimensional oriented field theory is \emph{anomalous}. It only respects the gluing law of cobordisms up to a projective factor. Said differently, rather than a representation of the bordism category, the Reshetikhin-Turaev construction produces a \emph{projective} representation. See for example \cite{MR1797619} and \cite{MR3617439} for in-depth treatments of the Reshetikhin-Turaev construction.

To get a more satisfactory situation, what is commonly done is that the bordism category is replaced with a what can be regarded as a `central extension' of the bordism category. This is another symmetric monoidal category which comes with a forgetful functor:
\begin{equation*}
	p:\widetilde{\Bord}_{3;1}^{SO(3)} \to \Bord_{3;1}^{SO(3)}.
\end{equation*} 
The Reshetikhin-Turaev construction then produces an honest, non-projective representation:
\begin{equation*}
	\cZ: \widetilde{\Bord}_{3;1}^{SO(3)} \to \Vect
\end{equation*}
which lifts the previous projective version of the Reshetikhin-Turaev TQFT. 

Now the precise choice of the central extension $\widetilde{\Bord}_{3;1}^{SO(3)}$ differs somewhat from author to author. However they all share the common feature that for a fixed closed oriented 3-manifold $M$ the fiber of the projection $p$ over $M$, has the structure of a torsor over an abelian group $A$. In \cite{MR3617439} the abelian group is $A = k^*$, the units in the underling ground field of the target category of the TFT. This would be $\C^\times$ in the case we are considering here. 

However in this section we will instead limit ourselves to the case $A = \Z$, and we will be interested in divisibility phenomena. One early example of an $\Z$-central extension due to Atiyah \cite{MR1046621} uses what he called `2-framings'. These tangential structures are equivalent to `$p_1$-structures', which can be defined as follows. Fix once an for all a fibration
\begin{equation*}
	BSO(3)\langle p_1 \rangle  \to BSO(3) \stackrel{p_1}{\to} K(\Z,4).
\end{equation*} 
realizing $BSO(3)\langle p_1 \rangle$ as the homotopy fiber of a map $p_1$ representing the integral first Pontryagin class. A $p_1$-structure $\theta$ for a 3-manifold $M$ is a lift of the classifying map of the  tangent bundle of $M$ to $BSO(3)\langle p_1 \rangle$. The obstruction to finding such a lift is the 4-dimensional class $p_1(M) = 0$, which vanishes for all 3-manifolds $M$. So a lift always exists and for closed oriented $M$ the homotopy classes of such lifts form a $\Z$-torsor. In this case we set
\begin{equation*}
	\widetilde{\Bord}_{3;1}^{SO(3)} = {\Bord}_{3;1}^{BSO(3)\langle p_1 \rangle}.
\end{equation*}

A different central extension, not corresponding to a tangential structure, was considered by Walker. We described this in the introduction, but will describe it again for the reader's benefit. In Walker's category the objects are `extended surfaces' and the morphisms are 3-dimensional `extended bordisms' (this is Walker's terminology, not to be confused with extended (higher categorical) TFTs). In fact Walker implicitly describes two equivalent versions of this bordism category:
\begin{enumerate}
	\item `extended surfaces' are surfaces $\Sigma$ is equipped with a choice of a bounding manifold $\tilde{\Sigma}$. That is $\partial{\tilde{\Sigma}} \cong \Sigma$. A morphism from $(\Sigma_0, \tilde{\Sigma}_0)$ to $(\Sigma_1, \tilde{\Sigma}_1)$ is a 3-dimensional bordism $M$ from $\Sigma_0$ to $\Sigma_1$ together with a choice of 4-manifold $W$ with $\partial W = \overline{\tilde{\Sigma_0}} \cup_{\Sigma_0} M \cup_{\Sigma_1} \tilde{\Sigma_1}$. Two such $W_0$ and $W_1$ are considered equivalent if $W_0 \cup_{\partial W_0} \overline{W_1}$ is null-cobordant. Thus for a given $M$ there are a $\Z$-torsor worth of equivalence classes of possible $W$'s (distinguished by their signature).
	\item `extended surfaces' are surfaces $\Sigma$ equipped with a Lagrangian $L \subseteq H^1(\Sigma; \R)$ in the first cohomology of $\Sigma$. Morphisms are pairs consisting of a 3-dimension bordisms $M$ and an integer. The composition composes the bordisms in the obvious way. The integers are added with a correction term that depends on Wall’s non-additivity function (or equivalently on the Maslov index of the involved Lagrangians). 
\end{enumerate} 
We refer the reader to Walker's text  \cite{Walker-1991} for full details. 
There is a map from the first version to the second. It sends bounding 3-manifold $\tilde{\Sigma}$ to  $\ker(H_1(\Sigma; \R) \to H_1(\tilde{\Sigma};\R))$, which is a Lagrangian subspace and it sends a bounding 4-manifold to the signature of that four manifold, an integer. This map is an equivalence of categories.

Each central extension of the bordism category gives rise to a central extension of the mapping class group for all genera. This arrises as follows. If we fix a surface $\Sigma$, then for each diffeomorphism $f$ of $\Sigma$ we get a bordism from $\Sigma$ to itself, given by twisting the boundary parametrization of the cylinder bordism $\Sigma \times I$ by the given diffeomorphism. This bordism only depends on the diffeomorphism up to isotopy. This realizes a copy of the mapping class group $\Gamma(\Sigma)$ inside the oriented bordism category. 
If we lift $\Sigma$ to an extended surface and look at its automorphism group $\tilde{\Gamma}(\Sigma)$ in one of the above categories.  This fits into a central extension of groups:
\begin{equation*}
	1 \to \Z \to \tilde{\Gamma}(\Sigma) \to \Gamma(\Sigma) \to 1
\end{equation*}
For large genus $H^2(\Gamma; \Z) \cong \Z$ and this gives one mechanism to compare the various central extensions of the bordism categories. For example, Atiyah's extension of the bordism category induces a central extension of the mapping class group corresponding to twelve times the generator \cite{MR1046621}. Walker's extension of the bordism category gives a mapping class group extension corresponding to four times the generator \cite{Walker-1991, MR1329450}. 
In \cite{MR2096678} Gilmer identified an index two subcategory of Walker's category. This subcategory induces the central extension of the mapping class group corresponding to twice the generator of $H^2(\Gamma; \Z)$. A related central extension appears in \cite[5.7]{MR1797619}, though Bakalov-Kirillov only consider invertible 3-dimensional bordisms (equivalently isotopy classes of diffeomorphisms of surfaces) and not general 3-dimensional bordisms. Their central extension also corresponds to twice the generator of $H^2(\Gamma; \Z)$.

In \cite[Rmk.~7.5]{MR3100961} Gilmer and Masbaum ask whether it is possible to find an index four subcategory of Walker's category which would realize the fundamental central extension of the mapping class group? We will now describe how our computations of invertible topological field theories can be used to give a negative answer to Gilmer and Masbaum's question. Indeed we will show that there is no central extension of the bordism category corresponding to the generator of the mapping class group. In particular there is no index four subcategory of Walker's category.

The connection with invertible field theories arises from the simple observation that each central extension of the oriented bordism category can be reinterpreted as an oriented topological field theory valued in a higher category of $\Z$-torsors. More specifically let $\Tor_{\Z}$ denote the symmetric monoidal category of $\Z$-torsors. Let $\Tor_{\Z}$-$\Cat^\circ$ denote the 2-category of \emph{inhabited} $\Tor_{\Z}$-enriched categories, functors, and transformations. Inhabited simply means that the category is non-empty. Any two inhabited $\Tor_{\Z}$-enriched categories are equivalent. In fact any enriched functor between inhabited $\Tor_{\Z}$-enriched categories is an equivalence, and any two enriched functors are naturally isomorphic. It follows that $\Tor_{\Z}$-$\Cat^\circ$ is a model for $K(\Z,2)$. 

Any central extension of the oriented bordism category gives rise to a necessarily invertible oriented topological field theory valued in $\Tor_{\Z}$-$\Cat^\circ \simeq K(\Z,2)$. For example to each closed oriented 3-manifold $M$ we can associate a $\Z$-torsor: the fiber $p^{-1}(M)$. Since we aim to answer Gilmer and Masbaum's question which concerns Walker's extension of the bordism category we will describe the associated invertible topological field theory explicitly in that case. We will use the first variant of Walker's extension of the bordism category in which the surfaces and 3-dimensional bordism are equipped with bounding manifolds. This extension corresponds to the following TFT valued in $\Tor_{\Z}$-$\Cat^\circ$:
\begin{itemize}
	\item To an oriented surface $\Sigma$ we associate the following $\Tor_{\Z}$-enriched category $\cZ(\Sigma)$:
	\begin{itemize}
		\item The objects of $\cZ(\Sigma)$ are oriented 3-manifolds $\tilde{\Sigma}$ with identifications $\partial(\tilde{\Sigma}) \cong \Sigma$. That is they are bounding 3-manifolds for $\Sigma$. 
		\item The morphism from $\tilde{\Sigma}$ to $\tilde{\Sigma}'$ in $\cZ(\Sigma)$ are equivalence classes of 4-manifolds $W$ with $\partial W \cong \overline{\tilde{\Sigma}} \cup_\Sigma \tilde{\Sigma}'$. Two such 4-manifolds $W_0$ and $W_1$ are equivalent if $\overline{W_0} \cup_{\partial W_0} W_1$ is null-bordant, equivalently if they have the same signature. Composition is given by the obvious gluing. 
	\end{itemize}
	\item If we are given a 3-dimensional bordism $M$ from $\Sigma_0$ to $\Sigma_1$, then we get an induced functor $\cZ(M): \cZ(\Sigma_0) \to \cZ(\Sigma_1)$. This functor sends the object $\tilde{\Sigma}$ of $\cZ(\Sigma_0)$ to $\tilde{\Sigma} \cup_{\Sigma} M$. That is $\cZ(M)$ is the functor induced by composition with $M$. 
\end{itemize}

\begin{remark}
	The functor $\cZ$ actually extends `upward' to a functor 
	\begin{equation*}
		\cZ: \Bord_{4;2}^{SO(4)} \to K(\Z,2)
	\end{equation*}
	and downward to a functor
		\begin{equation*}
			\cZ: \Bord_{4;3}^{SO(4)} \to K(\Z,3).
		\end{equation*}
\end{remark}

To summarize, we get a map:
\begin{equation*}
	  \left\{ \begin{array}{c} \textrm{central extensions} \\ \textrm{of the 3D} \\ \textrm{bordism category} \end{array} \right\} \to  
	  \left\{ \begin{array}{c} \textrm{topological field theories} \\
	  \cZ: \Bord_{3;1}^{SO(3)} \to K(\Z,2) \end{array} \right\}  \cong H\Z^4(p_{\geq 2}\Sigma^3MTSO(3)) \cong \Z.
\end{equation*}  
As we computed in Section~\ref{sec:cohom-of-MT-spectra}, this last group is isomorphic to the integers and is generated by a class $\rho$. Moreover under this construction Atiyah's extension corresponds to the characteristic class $p_1$ and Walker's extension corresponds to the signature class $\sigma$. Our cohomology calculations show that $p_1 = 6 \rho$ and $\sigma = 2 \rho$. Gilmer's index two subcategory necessarily corresponds to the class $\rho$, but since $\rho$ is the generator (and not twice another class) there can be no index four subcategory of Walker's category. This proves Theorem~\ref{thm:GM-question} from the introduction.

\appendix

\section{Comparison with others spaces of embedded manifolds} \label{app:comparison}

In this appendix we will compare the topology on the space $\psi_d(\R^m)$ of embedded manifolds constructed in Section~\ref{sec:space_of_man} with the topology constructed by Galatius--Randal-Williams \cite{MR2653727}.  We will denote that topology $\tau_\text{GRW}$ to distinguish it from the plot topology which we denote $\tau_\text{plot}$. The topology $\tau_\text{GRW}$ was also considered in \cite{MR3243393}, where it was shown to be metrizable. 

Compactly generated spaces are a common tool among algebraic topologist. Given a topological space $X$ we get a new topological space $k(X)$ whose open sets are precisely those $U \subseteq X$ such that for all compact Hausdorff spaces $K$ and continuous maps $p: K \to X$, the set $p^{-1}(U) \subseteq K$ is open. The topology on $k(X)$ may be finer and the identity map gives a continuous map $k(X) \to X$. We say $X$ is \emph{compactly generated} if this is a homeomorphism. 

Less well-known but similar in spirit are the $\Delta$-generated spaces. These use the disks $D^k$ instead of the compact Hausdorff spaces $K$. Given a topological space $X$ we get a new topological space $d(X)$ whose open sets are precisely those $U \subseteq X$ such that for all $k$ and all continuous $p: D^k \to X$, the subset $p^{-1}(U) \subseteq D^k$ is open. A space is \emph{$\Delta$-generated} if the canonical comparison map $d(X) \to X$ is a homeomorphism. 
In fact, being $\Delta$-generated is the same as being $\R$-generated, see \cite{Christensen:2013aa}, and so to define $d(X)$ is is sufficient to consider instead all continuous curves $p: \R \to X$.
The category of $\Delta$-generated spaces has may desirable properties, even beyond the category of compactly generated spaces, see \cite{dugger_dgs}.

In section Section~\ref{sec:space_of_man} we provide a diffeology for the space $\psi_d(\R^m)$, that is a collection of smooth plots $U \to \psi_d(\R^m)$ parametrized by finite dimensional smooth manifolds. Topologies induced by diffeologies have been studied before, for example in \cite{Christensen:2013aa} where they are called the D-topology. In this context we have the following:

\begin{proposition}[{\cite[Th.~3.7, Pr.~3.10]{Christensen:2013aa}}]\label{pro:diffeology_by_curves}
	Let $(X, \cP)$ be a diffeological space and $\tau_\cP$ the corresponding diffeology. Then the topological space $(X, \tau_{\cP})$ is $\Delta$-generated and the topology is determined by the smooth curves in the following sense. A subset $A \subseteq X$ is open in $\tau_\cP$ if and only if $p^{-1}(A)$ is open for each smooth plot $p: \R \to X$ (with source $\R$).  \qed
\end{proposition} 

We also have:

\begin{proposition}[{\cite[Pr.~3.11]{Christensen:2013aa}}]\label{pro:identify_delta_gen}
	Every locally path-connected first countable topological space is $\Delta$-generated. \qed
\end{proposition}

We will show:
\begin{theorem}\label{thm:spaceofmanifoldcomparison}
	The identity map is a homeomorphism of topological spaces
	 $(\psi_d(M), \tau_\text{plot}) \cong (\psi_d(M), \tau_\text{GRW})$ between $\psi_d(M)$ equipped with the plot topology and the same set equipped with the topology introduced by Galatius-Randal-Williams \cite{MR2653727}. 
\end{theorem}
	
The key lemma which will establish Theorem~\ref{thm:spaceofmanifoldcomparison} is the following:

\begin{lemma}\label{lem:curve_approx}
	Let $\gamma: \R \to (\psi_d(M), \tau_\text{GRW})$ be a continuous path. Then for each $t \in \R$ and each choice of convergent sequence $t_i \to t$, there exists a convergent subsequence $t_j \to t$ and a smooth plot $p: \R \to (\psi_d(M), \tau_\text{plot})$ such that $p(\frac{1}{j}) = \gamma(t_j)$ and $p(0) = \gamma(t)$.	
\end{lemma}
 
\begin{proof}[Proof of Theorem~\ref{thm:spaceofmanifoldcomparison}, given Lemma~\ref{lem:curve_approx}:]
	The topology $\tau_\text{GRW}$ was shown to metrizable in \cite{MR3243393}. In particular it is first countable. A quick inspection of either \cite{MR3243393} or \cite{MR2653727} shows that  is also locally path connected, and consequently by Prop.~\ref{pro:identify_delta_gen}, the $\tau_\text{GRW}$ topology is $\Delta$-generated. It follows that the closed sets of $\tau_\text{GRW}$ are determined by the continuous paths $\gamma: \R \to (\psi_d(M), \tau_\text{GRW})$, as in Prop.~\ref{pro:identify_delta_gen}. It will be sufficient to show these are same closed sets determined by the smooth plots. 
	
	Let $A \subseteq  (\psi_d(M), \tau_\text{plot})$ be a closed subset in the smooth plot topology. It suffices to show for all curves $\gamma: \R \to \psi_d(M)$ which are continuous in the $\tau_\text{GRW}$-topology,  that $\gamma^{-1}(A) \subseteq \R$ is closed. To that end suppose that $t \in \R$ is a limit point of $\gamma^{-1}(A)$, and hence there exists a convergent sequence $t_i \to t$ with $t_i \in \gamma^{-1}(A)$. By Lemma~\ref{lem:curve_approx} there exists a subsequence $\{ t_j \}$ and a smooth plot $p: \R \to \psi_d(M)$ such that $p(\frac{1}{j}) = \gamma(t_j)$ and $p(0) = \gamma(t)$. In particular $\frac{1}{j} \in p^{-1}(A)$. Now since $p$ is a smooth plot, we have that $p^{-1}(A) \subseteq \R$ is closed, and hence $p(0) = \gamma(t) \in A$ as well. It follows that $\gamma^{-1}(A)$ is closed, as desired. 	
\end{proof}

Galatius--Randal-Williams prove many useful facts about their topology, but we will only need to use two.

\begin{lemma}[\cite{MR2653727}]\label{lem:tub_nbhd}
	Let $W \in \psi_d(\R^m)$, and let fix a tubular neighborhood, that is a open neighborhood $N \subseteq \nu_W$ of the zero section of the normal bundle of $W$ and an embedding $N \subseteq \R^n$ identifying $W$ with the zero section $W \subseteq N$. Consider the space $\Gamma_c(N)$ of compactly supported sections of $N$ as a topological space with the strong $C^\infty$ topology (see \cite[Sect~2.1]{MR2653727}).
Then for all compact subsets $K \subset \R^m$ and all open $U \subseteq \Gamma_c(N)$ the following subset of $\psi_d(\R^m)$ is open in $\tau_\text{GRW}$:
	\begin{equation*}
		M_{K,U} = \left\{ W' \in  \psi_d(\R^m) \; | \; \begin{array}{c}
			\text{there exists an open neighborhood } V \supseteq K, \\
			 W' \cap V \text{ differs from } W \cap V \text{ by an element of } U \subseteq \Gamma_c(N)
		\end{array}  \right\}
	\end{equation*}
	\qed
\end{lemma}

The strong $C^\infty$ topology is well studied. 
We do not need to know much about the strong $C^\infty$ topology on $\Gamma_c(\nu_W)$ except that it is locally convex vector space and the topology is induced by a countable directed family of seminorms $\{\rho_k\}$ with $\rho_k \leq \rho_{k+1}$. 

The second result of Galatius--Randal-Williams that we will use is their smooth approximation lemma. We only state a special restricted case for curves:

\begin{lemma}[\cite{MR2653727}] \label{lem:smooth_approx}
	Let $\gamma: \R \to \psi_d(M)$ be a continuous path with respect to the topology $\tau_\text{GRW}$. 
	Let $V \subseteq \R\times M$ be open and $S$ such that $\overline{V} \subseteq int(S) $. Then there exists a homotopy $F: [0,1] \times \R \to \psi_d(\R^m)$ starting at $\gamma$, which is smooth on $(0,1] \times V \subseteq [0,1] \times \R \times \R^m$ and is constant outside $S$. Furthermore, if $\gamma$ is already smooth on an open set $A\subseteq V$, then the homotopy can be assumed smooth on $[0,1] \times A$ and constant on $[0,1] \times Z$ for any closed subset $Z \subseteq A$. \qed
\end{lemma}

\begin{remark}
	The result stated in Galatius--Randal-Williams does not mention being able to keep the homotopy constant on the closed set $Z$, but this is an easy extension of their proof. 
\end{remark}

\begin{proof}[Proof of Lemma~\ref{lem:curve_approx}]
	Let $\gamma: \R \to \psi_d(M)$ be a path which is continuous in the topology $ \tau_\text{GRW}$, and let  $t_i \to t$ be a convergent sequence in $\R$. We wish to show that there exists a convergent subsequence $t'_j \to t$ and a smooth plot $p: \R \to \psi_d(M)$ such that $p(\frac{1}{j}) = \gamma(s_j)$ and $p(0) = \gamma(t)$.
	
Without loss of generality we may assume that the $t_i$ are distinct and strictly decreasing. That is $t_i > t_{i+1} > t$. Moreover we can replace $\gamma$ with a continuous path which is delayed around each $t_i$ (i.e. is a constant path in a neighborhood around each $t_i$). Thus we may assume that there are sequences $a_i$, $b_i$ with 
\begin{equation*}
	\dots < a_{i+1} < t_{i+1}< b_{i+1} < a_i < t_i < b_i <  \cdots 
\end{equation*}
	such that the restriction of $\gamma$ to $[a_i, b_i]$ is the constant path with value $\gamma(t_i)$. 

Let $W = \gamma(t)$. Fix a tubular neighborhood $N \subseteq M$ of $W$ as in  Lemma~\ref{lem:tub_nbhd}. We may assume $N$ is convex. Fix a metric $g$ on $\nu_W$. Next we choose a nested sequence of compact subsets $C_j \subseteq K_j \subseteq L_j \subseteq W$, and smooth bump functions $\varphi_i: M \to [0,1]$ with $\varphi_i|_{K_i} \equiv 1$ and $\varphi_i|_{W \setminus L_i} \equiv 0$ such that $int(C_j) \neq \emptyset$,$C_j \subseteq int(K_j)$,  $L_j \subseteq C_{j+1}$, and 
\begin{equation*}
	\cup_j C_j = \cup_j K_j = \cup_j L_j = W.
\end{equation*}
and moreover we choose $\epsilon_j$ such that $\overline{B}(\epsilon_j, g) \cap \nu_W|_{L_j} \subseteq N$ where $\overline{B}(\epsilon_j, g)$ is the closed disk bundle of radius $\epsilon_j$ in the metric $g$. The restriction of this disk bundle to $L_j$ is a compact subset of $M$. For future reference we let:
\begin{align*}
	\overline{L}_j &= \overline{B}(\epsilon_j, g) \cap \nu_W|_{L_j} \\
	\overline{K}_j &= \overline{B}(\epsilon_j, g) \cap \nu_W|_{K_j} \\
	\overline{C}_j &= \overline{B}(\epsilon_j, g) \cap \nu_W|_{C_j}
\end{align*}

We will inductively choose the subsequence $t'_j$ as follows. For each $j$ we consider the open subset $U_j = \Gamma_c(N) \cap B(\frac{1}{2 j^j}; \rho_j) \subseteq \Gamma_c(\nu_W)$ which is the intersection of $N$ and the open ball of radius $\frac{1}{2 j^j}$ around the zero section in the $j^\text{th}$ seminorm $\rho_j$. By Lemma~\ref{lem:tub_nbhd} the subset
	\begin{equation*}
		M_{\overline{L}_j,U_j} = \left\{ W' \in  \psi_d(\R^m) \; | \; \begin{array}{c}
			\text{there exists an open neighborhood } V \supseteq \overline{L}_j, \\
			 W' \cap V \text{ differs from } W \cap V \text{ by an element of } U_j \subseteq \Gamma_c(N)
		\end{array}  \right\}
	\end{equation*}
is an open neighborhood of $W = \gamma(t)$. Hence $\gamma^{-1}(M_{\overline{L}_j,U_j})$ is an open neighborhood of $t \in \R$. It follows that there exists some $i$ such that $t_i < t'_{j-1}$ such that for all $s$ with $t \leq s \leq b_i$,  $\gamma(s) \in M_{\overline{L}_j,U_j}$. Set $t'_j = t_i$, $a'_j = a_i$, and $b'_j = b_i$. 

Next we reparametrize $\gamma$ by precomposing with a homeomorphism of $\R$ which sends:
\begin{align*}
	0 & \mapsto t \\ 
	\frac{1}{j} &\mapsto t'_j \\ 
	\frac{j + \frac{2}{3}}{j (j+1)} & \mapsto a'_j \\ 
	\frac{j + \frac{1}{3}}{j (j+1)} & \mapsto b'_{j+1}
\end{align*}
We will still use $\gamma$ to denote this reparametrized path. 

We now define a new continous path $\tilde{\gamma}: \R \to \psi_d(M)$. It satisfies:
\begin{equation*}
 \tilde{\gamma}(t) = \begin{cases}
 	W = \gamma(0) & t \leq 0 \\
	\gamma(\frac{1}{j}) & \frac{j + \frac{2}{3}}{j ( j+1)} \leq t \leq \frac{(j-1) + \frac{1}{3}}{(j-1) j} \\
	\gamma(1) & t \geq 1
 \end{cases}	
\end{equation*}
It remains to define $\tilde{\gamma}$ on the intervals: 
\begin{equation*}
	\left[  \frac{j + \frac{1}{3}}{j ( j+1)},  \frac{j + \frac{2}{3}}{j ( j+1)} \right]
\end{equation*}
On this interval $\tilde{\gamma}(t)$ agrees with $\gamma(t)$ outside of $\overline{L}_j$. By our choice of subsequence $t'_j$ we know that in a neighborhood of $\overline{L}_j$, $\gamma(t)$ is given by the graph of a section $s(t)$ of $\Gamma_c(N)$ over $W$. The same is true for $\tilde{\gamma}(t)$. In the same neighborhood of $\overline{L}_j$, $\tilde{\gamma}(t)$ is the graph of the section $\tilde{s}(t) \in \Gamma(N)$ defined as 
\begin{equation*}
	\tilde{s}(t) = (1- \varphi_j) \cdot s(t)  + \varphi_j |_W ( s(\frac{1}{j+1}) + \phi\left( \frac{t - \frac{j + \frac{1}{3}}{j (j+1)}  }{ \frac{j + \frac{2}{3}}{j (j+1)} - \frac{j + \frac{1}{3}}{j (j+1)} } \right) \cdot (s(\frac{1}{j}) - s(\frac{1}{1+j})) )
\end{equation*}
where $\phi: \R \to [0,1]$ is any fixed smooth map which is 0 on $\{ t \; |\; t \leq 0\}$ and 1 on $\{ t \; |\; t \geq 1\}$, and $\varphi_j$ is the smooth bump function chosen earlier which is identically one on $K_j$ and zero on the compliment of $L_j$. 

Thus near $\overline{K}_j$, in the interval in question,  $\tilde{\gamma}$ is smooth and traces out a straight line path from $\gamma(\frac{1}{j})$ to $\gamma(\frac{1}{j+1})$. Out side of $\overline{K}_j$ it interpolates continuously back to $\gamma(t)$. We have that $\tilde{\gamma}$ is smooth on the union 
\begin{equation*}
	\left(\bigcup_j M \times \left[ \frac{j + 1 + \frac{2}{3}}{(j+1)(j+2)}, \frac{j + \frac{1}{3}}{j ( j+1)} \right]\right) \cup \left( \bigcup_j \overline{K}_j  \times 	\left[  \frac{j + \frac{1}{3}}{j ( j+1)},  \frac{j + \frac{2}{3}}{j ( j+1)} \right] \right)  
\end{equation*}
and is also clearly smooth on $  (-\infty,0) \times M$ and $ (1, \infty) \times M $. In fact $\tilde{\gamma}$ is smooth at $t=0$ as well. This can be seen by computing the derivatives $\tilde(\gamma)$ with respect to $t$, which must vanish at all orders. For $t < \frac{1}{j}$ and in $\overline{K}_j$ we have have 
\begin{equation*}
	(\tilde{s})^{(p)}(t) = \phi^{(p)}\left( \frac{t - \frac{j + \frac{1}{3}}{j (j+1)}  }{ \frac{j + \frac{2}{3}}{j (j+1)} - \frac{j + \frac{1}{3}}{j (j+1)} } \right)\cdot 3^p j^p (j+1)^p \cdot (s(\frac{1}{j}) - s(\frac{1}{1+j})) 
\end{equation*}
which converges to zero as $j \to \infty$ in any of the seminorms $\rho_j$.

Now we apply the smooth approximation Lemma~\ref{lem:smooth_approx} with $V = S = M \times \R$. $A$ is the interior of the union of $(-\infty, 0] \times M$ and
	\begin{equation*}
		\left(\bigcup_j M \times \left[ \frac{(j + 1) + \frac{2}{3}}{(j+1)(j+2)}, \frac{j + \frac{1}{3}}{j ( j+1)} \right]\right) \cup \left( \bigcup_j \overline{K}_j  \times 	\left[  0,  \frac{j + \frac{2}{3}}{j ( j+1)} \right] \right)  
	\end{equation*}
while $Z$ is the union of $(-\infty, 0] \times M$ and
\begin{equation*}
		\left(\bigcup_j M \times \left[ \frac{(j + 1) + \frac{5}{6}}{(j+1)(j+2)}, \frac{j + \frac{1}{6}}{j ( j+1)} \right]\right) \cup \left( \bigcup_j \overline{C}_j  \times 	\left[  0,  \frac{j + \frac{2}{3}}{j ( j+1)} \right] \right)  
\end{equation*}
The result is a smooth plot $p: \R \to \psi_d(M)$ which agrees with $\tilde{\gamma}$ on $Z$. In particular
\begin{equation*}
	p(\frac{1}{j}) = \tilde{\gamma}(\frac{1}{j}) = \gamma(t_j) 
\end{equation*}
as desired. 
\end{proof}

\section{Embedded manifolds and $BDiff(M)$}\label{app:embedded_classifying}

The purpose of this section is to advertise the plot-theoretic approach to the topology on the space of embeddings by providing a simple proof that the space of embeddings is a union of classifying spaces:
\begin{equation*}
	\psi_d(D^\infty) = \colim_p \psi_d(D^p) \simeq \coprod_{[M]} B Diff(M)
\end{equation*}
is a disjoint union, taken over the diffeomorphism types of compact $d$-manifolds $M$, of the classifying space of the diffeomorphism group of $M$. Recall that $\psi_d(D^p)$ is the space of closed embedded $d$-manifolds $W \subseteq D^p$ which are disjoint from $\partial D^p$. In particular $W$ is a compact $d$-manifold. The colimit $\psi_d(D^\infty)$ is similar with manifolds embedded into $D^\infty$. 

Fix a compact $d$-manifold $M$ and let $\psi_d^M(D^p)$ be the subset of $\psi_d(D^p)$ consisting of those embedded manifolds $W$ which are diffeomorphic to $M$ (the diffeomorphism is not specified). 

\begin{lemma}
	The space $\psi_d(D^p)$ decomposes as a disjoint union over the set of diffeomorphism types $[M]$ of compact $d$-manifolds $M$:
	\begin{equation*}
		\psi_d(D^p) = \coprod_{[M]} \psi_d^M(D^p).
	\end{equation*}
\end{lemma}

\begin{proof}
	By Prop.~\ref{pro:diffeology_by_curves} the topology of any diffeological space, such as $\psi_d(D^p)$, is determined by the curves, that is the plots index by $\R$. So we will only consider those. Recall that a  plot $R \to \psi_d(D^p)$ consists of an embedded $(d+1)$-manifold $W \subseteq \R \times \D^p$, disjoint from the boundary, and such that the projection $W \to \R$ is a submersion. Since the projection $\R \times \D^p \to \R$ is proper (since $D^p$ is compact) and $W \subseteq \R \times D^p$ is closed, it follows that $W \to \R$ is also proper. Then by Ehresmann's fibration theorem $W \to \R$ is a locally trivial fiber bundle, and hence trivial since $\R$ is connected and contractible. It follows that $W \cong M \times \R$ (as a space over $\R$) for some compact $d$-manifold $M$ and hence the plot is entirely contained in $\psi_d^M(D^p)$. It follows that the sets $\psi_d^M(D^p)$ are both closed and open.   
\end{proof}

From the proof we also see that the plots for $\psi_d^M(D^p)$ have a nice description. A plot $\R \to \psi_d^M(D^p)$ consists of an embedded manifold $W \subseteq \R \times D^p$, disjoint from the boundary such that there exists a diffeomorphism $W \cong \R \times M$, commuting with the projection to $\R$. This suggests introducing a new space:
\begin{definition}
		Let $Emb(M,D^p)$ denote the set of embeddings $M \hookrightarrow D^p$ which are disjoint from $\partial D^p$, topologized by the following set of plots. A plot $\R \to Emb(M,D^p)$ is an embedding 
		\begin{equation*}
			\phi:\R \times M \to \R \times D^p
		\end{equation*}
		commuting with the projection to $\R$ and disjoint from the boundary. 
\end{definition}

There is a map $Emb(M, D^p) \to \psi_d^M(D^p)$ which sends an embedding to its image $\phi(M) \subseteq D^p$. This clearly sends plots to plots and hence is continuous. The fiber is easily seen to be the diffeomorphism group $Diff(M)$ with plots $\R \to Diff(M)$ given by diffeomorphisms $\R \times M \cong \R \times M$ commuting with the projection to $\R$. 

\begin{lemma}
	The map $Emb(M, D^p) \to \psi_d^M(D^p)$ realizes $Emb(M, D^p)$ as a $Diff(M)$-principle bundle over $\psi_d^M(D^p)$.
\end{lemma}

\begin{proof}
	We need to show that $Emb(M, D^p)$ is locally isomorphic to a trivial $Diff(M)$-principle bundle. Fix an embedding $M \subseteq D^p$. We will also let $M$ denote the image $ M \in \psi_d^M(D^p)$. Let $\nu_{M}$ be the normal bundle of $M$ and fix a tubular neighborhood of $N$. We let $N \subseteq D^p$ denote the tubular neighborhood image of $\nu_{N}$. The compliment $N^c \subseteq D^p$ is a compact subset. Consider the map:
	\begin{align*}
		\psi_d^M(D^p) &\to \cl(N^c)
		W &\mapsto W \cap N^c
	\end{align*}
	This map sends plots to plots and hence is continuous. Moreover, since $N^c$ is compact Lemma~\ref{cor:empty_open} tells us that $\{\emptyset\} \subseteq \cl(N^c)$ is open. It follows that 
	\begin{equation*}
		U_N = \{ W \in \psi_d^M(D^p) \; | \; W \subseteq N \} = ( (-) \cap N^c)^{-1}(\emptyset)
	\end{equation*}
	is an open subset of $\psi_d^M(D^p)$. 

	Let $\pi:N \to M$ be the projection to the zero section. There is an embedding $\Gamma(\nu_{M}) \to U_N$ which takes a section of the normal bundle of $M$ to its graph viewed as a submanifold of $N$. Its image consists of those $W \in U_N$ such that $\pi|_W: W \to M$ is a diffeomorphism. We will now show that this is an open subset of $U_N$. Consider the map
	\begin{align*}
		\rho:U_N & \to \cl(M) \\
		W & \mapsto \pi(\{ w \in W \;|\; w \text{ is a critical point for } \pi|_W  \})
	\end{align*}
which sends $W$ to the image in $M$ of the critical points of $\pi|_W$. The set of critical points is a closed subset of $W$. Since $W$ is compact there exists a $\lambda$ such that $W$ is contained in the image of the $\lambda$-disk bundle of $\nu_M$ (for some chosen metric). It follows that $\pi|_W: W \to M$ is proper and hence the map $\rho$ is well-defined. It also clearly sends plots to plots and so is continuous. Since $M$ is compact $\{\emptyset\} \subseteq \cl(M)$ is open by Lemma~\ref{cor:empty_open}, and hence its inverse image
\begin{equation*}
	V_N = \{ W \in U_N \; |\; \pi|_W: W \to M \text{ is a diffeomorphism} \} \cong \Gamma(\nu_M)
\end{equation*}
	 is an open subset of $\psi_d^M(D^p)$.
	 
	 The restriction of $Emb(M,D^p)$ over $V_N$ consists of those embeddings $\phi:M \hookrightarrow N \subseteq D^p$ such that
	 \begin{equation*}
	 	M \stackrel{\phi}{\to} N \stackrel{\pi}{\to} M
	 \end{equation*}
	 is a diffeomorphism. We have a canonical isomorphism:
	\begin{align*}
		Emb(M,D^p)|_{V_N} &\cong Diff(M) \times V_N \\ 
		(\phi: M \to N) & \mapsto ( \pi \circ \phi,  \phi(M) \subseteq N)
	\end{align*}
	which is clearly compatible with the action of $Diff(M)$ on $Emb(M,D^p)$. This shows that $Emb(M,D^p)$ is a locally trivial principle $Diff(M)$-bundle over  $\psi_d^M(D^p)$.
\end{proof}

\begin{lemma}
	The colimit $Emb(M,D^\infty)$ is contractible. 
\end{lemma}

\begin{proof}
	A standard contraction works in this case. Let $\phi_0 = (\phi_0^{(k)})_{k = 1}^\infty$ be the coordinates of an element $Emb(M,D^\infty) = \colim_p Emb(M,D^p)$. Thus there exists $N$ such that $\phi_0^{(k)} \equiv 0$ is the constant zero function whenever $k > N$. 
	
	 We will build a continous contraction onto $\phi_0$. For an arbitrary element $\phi: M \to D^\infty$, we will let $(\phi^{(k)})_{k = 1}^\infty$ denote its coordinates. 	We consider the family of embeddings $R^\infty \to R^\infty$ which on coordinates is given by
	\begin{equation*}
		e_i \mapsto t e_i + (1-t)e_i
	\end{equation*}
	This is an injection for each $t$ and hence postcomposing with this gives a path from the embedding $(\phi^{(1)},\phi^{(2)}, \dots )$ to the embedding $(0, \phi^{(1)},\phi^{(2)}, \dots )$. A similar path allows us to change the first coordinate to the coordinate function $\phi_0^{(1)}$. Then we shift again to get $(\phi_0^{(1)}, 0 ,\phi^{(1)},\phi^{(2)}, \dots )$, and then change the second coordiante to $\phi_0^{(2)}$. In this way we obtain a sequence of paths:
	\begin{align*}
		(\phi^{(1)},\phi^{(2)}, \dots ) \\
		(0, \phi^{(1)},\phi^{(2)}, \dots ) \\
		(\phi_0^{(1)}, \phi^{(1)},\phi^{(2)}, \dots ) \\
		(\phi_0^{(1)}, 0 ,\phi^{(1)},\phi^{(2)}, \dots ) \\
		(\phi_0^{(1)}, \phi_0^{(2)},\phi^{(1)},\phi^{(2)}, \dots ) \\
		(\phi_0^{(1)}, \phi_0^{(2)}, 0 ,\phi^{(1)},\phi^{(2)}, \dots ) \\
		(\phi_0^{(1)}, \phi_0^{(2)}, \phi_0^{(3)}, ,\phi^{(1)},\phi^{(2)}, \dots ) \\
		\cdots
	\end{align*}
After the $N^\text{th}$ iteration we have arrived at an embedding of the form 
\begin{equation*}
	(\phi_0^{(1)}, \phi_0^{(2)}, \dots, \phi_0^{(N)}, \phi^{(1)},  \phi^{(2)}, \dots)
\end{equation*}	
At this point we run a linear homotopy shrinking all coordiantes in degrees $k > N$ to zero to obtain the embedding:
\begin{equation*}
	(\phi_0^{(1)}, \phi_0^{(2)}, \dots, \phi_0^{(N)}, 0 , 0, \dots) = \phi_0
\end{equation*} 	
as desired. 	
\end{proof}

\begin{corollary}
	$\psi_d^M(D^p) \simeq BDiff(M)$ is a model for the classifying space of the diffeomorphism group $Diff(M)$. \qed
\end{corollary}

\section{Low degree homotopy groups of $MTSO(d)$} \label{app:lowhomotopy}

For $k=0$ we have that $\pi_0 MTSO(d)$ is given by 
 the quotient of the monoid of diffeomorphism classes of oriented $d$-manifolds modulo the relation that $[M] \sim 0$ whenever there exists an oriented $(d+1)$-manifold $W$ with $\partial W = M$ and with a non-vanishing vector field on $W$ which restricts to the inward pointing vector field on $M$. 
The Euler class is the obstruction to finding a non-vanishing vector field which restricts to the inward pointing one on the boundary, and so we can restate the relation as: $[M] \sim 0$ whenever there exists an oriented $(d+1)$-manifold $W$ with $\partial W = M$ and such that each component of $W$ has zero Euler characteristic. 

In dimension $d=1$ we see that the Euler characteristic of $S^1 \times I$ is zero and so $2 \cdot [S^1] \sim 0$,  which the Euler characteristing of a genus $g$ surface with $k$ disks removed is $2 -2g - k$, which can only vanish if $k$ is even (and then only rarely). It follows that $\pi_0 MTSO(1) = \Z/2$, which of course also follows from the description $MTSO(1) \simeq \Sigma^{-1} \S$.

Now consider dimension $d=2$. Let $H_g$ be the solid genus $g$ handle body. The boundary is $\partial H_g = \Sigma_g$ is the surface of genus $g$ and $\chi(H_g) = 1-g$. From this we have:
\begin{equation*}
	\chi(H_g \setminus \sqcup_k D^3) = 1-g + k
\end{equation*}
which will vanish when $k = g-1$. Thus in $\pi_0MTSO(2)$ we obtain the relation $[\Sigma_g] = (1-g) \cdot [S^2]$. Moreover since the Euler characteristic of any connected oriented 3-manifold is zero, it follows that if $M$ is an oriented connected 3-manifold bounding $k$ copies of $S^2$, then $\chi(M) = k$. Which means $k \cdot [S^2] \sim 0$ if and only if $k=0$. Hence $\pi_0MTSO(2) \cong \Z$, and this isomorphism sends an oriented surface to its Euler characteristic divided by two.   

When $d=3$ we will see that $\pi_0MTSO(3) = 0$. Fix an oriented 3-manifold $M$. Then, since $\Omega^\text{or}_3 = 0$, there exists some connected oriented 4-manifold $W$ with $\partial W = M$. If $\chi(W) = 0$, then $W$ witnesses the relation $[M] \sim 0$ and we are done. Otherwise we replace $W$ with a connect sum. Taking the connect sum with $\C\P^2$ raises the Euler characteristic by one without changing the boundary, and hence if $\chi(M) = k$ is negative, we take the connect sum of $W$ with $\C\P^2$ $|k|$-times. We have:
\begin{equation*}
	\chi( W \# \underbrace{\C\P^2 \# \cdots \# \C\P^2}_{\text{$|k|$-times}}) =  0.
\end{equation*}
On the other hand taking the connect sum with $T^4 \# \C\P^2$ lowers the Euler characteristic by one, and so if $k$ is positive, we have:  
 \begin{equation*}
 	\chi( W \# \underbrace{T^4 \#\C\P^2 \# \cdots \# T^4 \#\C\P^2}_{\text{$k$-times}}) =  0.
 \end{equation*}
In both case we obtain a manifold witnessing $[M] \sim 0$. 

The case $d=4$ is more interesting because there are oriented 4-manifolds which do not bound oriented 5-manifolds. The group $\Omega_4^\text{or} \cong \Z$ and is given by the signature. We have a surjective homomorphism:
\begin{equation*}
	\pi_0 MTSO(4) \to \Omega_4^\text{or} \cong \Z
\end{equation*}
The kernel consists of those 4-manifolds which do bound oriented 5-manifolds, modulo those which bound connected oriented 5-manifolds with zero Euler characteristic. Observe that if $W$ is a 5-manifold bounding a 4-manifold $M$, then $\chi(W) = \frac{1}{2} \chi(M)$, so for example $D^5$ is a $5$-manifold with $\partial D^5 = S^4$ and  $\chi(D^5) = +1$,  while $D^3 \times \Sigma_g$ is a $5$-manifold with $\partial(D^3 \times \Sigma_g) = S^2 \times \Sigma_g$ (which has zero signature) and  $\chi(D^3 \times \Sigma_g) = 2 - 2g$. Note also that if $W_1$ and $W_2$ are $5$-manifolds, then $\chi(W_1 \# W_2) = \chi(W_1) + \chi(W_2)$.

Now suppose that $W$ is a connected oriented 5-manifold bounding the 4-manifold $M$. If $k = \frac{1}{2}\chi(M) \leq 0$ is non-positive, then $W \setminus \sqcup_k D^5$ has vanishing Euler characteristic and bounds $M \cup \sqcup_k S^4$, which gives $[M] \sim (-k) \cdot [S^4]$. For example:
\begin{equation*}
	[S^2 \times \Sigma_g] \sim (2-2g) \cdot [S^4].
\end{equation*}
If $k = \frac{1}{2}\chi(M) >  0$ is positive, then taking the connect sum with copies of $D^3 \times \Sigma_g$ and possibly a copy of $D^5$ we get a manifold with zero Euler characteristic whose boundary is the union of $M$ and copies of $S^2 \times \Sigma_g$ and possibly a copy of $S^2$. This gives a linear relation which shows that
\begin{equation*}
	[M] \sim \frac{1}{2}\chi(M) \cdot [S^4].
\end{equation*}
Thus we have shown that there is a short exact sequence (which is necessarily split)
\begin{equation*}
	0 \to \Z \to \pi_0MTSO(4) \to \Z = \Omega_4^\text{or} \to 0
\end{equation*}
On the kernel of the signature map $\sigma: \pi_0MTSO(4) \to \Z \cong \Omega_4^\text{or}$, the identification $\ker(\sigma) \cong \Z$ is given by half the Euler characteristic $\frac{1}{2}\chi$, but that doesn't give a well-defined splitting in general since the Euler characteristic of oriented 4-manifolds can sometimes be odd (For example $\chi(\C\P^2)=3$). 

In fact by Poincare duality we have that 
\begin{equation*}
	\chi(M) = \begin{cases}
		\text{even } & \text{if $\sigma(M)$ is even} \\
		\text{odd } & \text{if $\sigma(M)$ is odd} \\
	\end{cases}.
\end{equation*} 
and so the quantity $\chi(M) + \sigma(M)$ is always even. 
An explicit splitting is given by the isomorphism
\begin{align*}
	\pi_0MTSO(4) & \cong \Z \oplus \Z \\
	M & \mapsto ( \frac{1}{2}(\chi(M) + \sigma(M)), \sigma(M)).
\end{align*}

The same sort of arguments work in higher dimensions and the observed phenomena depend of $d$ mod 4.

\printbibliography

\end{document}